\documentclass[11pt]{amsart}
\usepackage{amsmath}
\usepackage{amssymb}
\usepackage{amsthm}
\usepackage{latexsym}
\usepackage{graphicx}
\usepackage{hyperref}
\usepackage{enumerate}
\usepackage[all]{xy}
\usepackage{chngcntr}

\newtheorem{thm}{Theorem}[section]

\newtheorem{lem}[thm]{Lemma}
\newtheorem{prop}[thm]{Proposition}

\newtheorem{thmintro}{Theorem}

\theoremstyle{definition}

\newtheorem{cond}[thm]{Condition}
\newtheorem{ex}[thm]{Example}

\newcommand{\Z}{\mathbb Z}
\newcommand{\Q}{\mathbb Q}

\newcommand{\F}{\mathbb F}
\newcommand{\mf}{\mathfrak}
\newcommand{\mc}{\mathcal}
\newcommand{\mb}{\mathbf}

\def\cL{{\mathcal L}}

\def\Irr{{\rm Irr}}
\newcommand{\mr}{\mathrm}
\newcommand{\ind}{\mathrm{ind}}
\newcommand{\enuma}[1]{\begin{enumerate}[\textup{(}a\textup{)}] {#1} \end{enumerate}}
\newcommand{\Fr}{\mathrm{Fr}}

\newcommand{\Rep}{\mathrm{Rep}}
\newcommand{\Res}{\mathrm{Res}}

\newcommand{\der}{\mathrm{der}}

\newcommand{\Hom}{\mathrm{Hom}}

\newcommand{\isom}{\xrightarrow{\;\sim\;}}
\newcommand{\mosi}{\xleftarrow{\;\sim\;}}

\newcommand{\ds}{\displaystyle}

\newcommand{\IC}{\mathrm{IC}}

\newcommand{\Qe}{\overline{\Q_\ell}}
\newcommand\mapsfrom{\mathrel{\reflectbox{\ensuremath{\mapsto}}}}
\newcommand{\ICL}[1]{\underline{\mathrm{IC}}( #1 )_{\mathcal L}}

\begin{document}

\title[Representations of disconnected reductive groups over finite fields]
{Endoscopy for representations of disconnected reductive groups over finite fields}
\date{\today}
\author{Maarten Solleveld}
\address{IMAPP, Radboud Universiteit Nijmegen, Heyendaalseweg 135,
6525AJ Nij\-megen, the Netherlands}
\email{m.solleveld@science.ru.nl}
\subjclass[2010]{20C33, 20G40}

\maketitle

\begin{abstract}
Let $G$ be the group of rational points of a connected reductive group over a finite field.
Based on work of Lusztig and Yun, we make the Jordan decomposition for irreducible 
$G$-representations canonical. It comes in the form of an equivalence between the category of
$G$-representations with a fixed semisimple parameter $s$ and a category of essentially
unipotent representations of an endoscopic group of $G$, enriched with an equivariant 
structure with respect to the component group of the centralizer of $s$.

Next we generalize these results, replacing the connected reductive group by a 
smooth group scheme with reductive neutral component. Again we esta\-blish canonical
equivalences for both the rational and the geometric series of $G$-representations, 
in terms of essentially unipotent representations of endoscopic groups.
\end{abstract}

\tableofcontents

\section*{Introduction}

Let $\mc G$ be a connected reductive group over a finite field $\F_q$. We are interested 
in the representations of $G = \mc G (\F_q)$, always on vector spaces over $\overline{\Q}_\ell$
where $q$ and $\ell$ are coprime.

Deligne and Lusztig \cite{DeLu} showed that the irreducible $G$-representations can be 
partitioned in series, indexed by semisimple conjugacy classes in the $\F_q$-group $\mc G^\vee$
dual to $\mc G$. More concretely, starting from a maximal $\F_q$-torus $\mc T$ in $\mc G$ and
a character $\theta$ of $T = \mc T(\F_q)$ they constructed a virtual $G$-representation 
$R_T^G (\theta)$. 
The data $(\mc T,\theta)$ determine a semisimple element $s \in G^\vee = \mc G^\vee (\F_q)$,
up to $G^\vee$-conjugacy. Moreover, it was shown in \cite{DeLu} that $R_T^G (\theta)$ and 
$R_{T'}^G (\theta')$ have no common irreducible constituents if their semisimple parameters 
are not geometrically conjugate. That enables the following definition: the geometric series
with parameter $s$ consists of the irreducible $G$-representations that appear in 
$R_{T}^G (\theta)$ for some $(\mc T, \theta)$ whose associated element in $G^\vee$ 
is geometrically conjugate to $s$. 

Later Lusztig \cite{Lus1} improved on that, by showing that two virtual representations 
$R_T^G (\theta)$ and $R_{T'}^G (\theta')$ are disjoint whenever their parameters in $G^\vee$ 
are not $G^\vee$-conjugate. This lead to the rational (or Lusztig) series $\Irr_s (G)$, 
where $s$ re\-pre\-sents a semisimple (rational) conjugacy class in $G^\vee$. The element 
$s$ is called the semisimple parameter of a representation in $\Irr_s (G)$, uniquely 
determined up to $G^\vee$-conjugacy. The representations in $\Irr_1 (G)$ are especially 
interesting, these are known as unipotent representations. 

In \cite{Lus1,Lus2}, Lusztig proved that there is always a bijection
\begin{equation}\label{eq:1} 
\Irr_s (G) \longleftrightarrow \Irr_1 (Z_{G^\vee}(s)) .
\end{equation}
This works best when $Z_{G^\vee}(s)$ is connected, then the bijection is canonical. Thus 
every irreducible $G$-representation is determined by its semisimple parameter and a unipotent
representation of another reductive $\F_q$-group. This is known as the Jordan decomposition of
irreducible $G$-representations, we refer to \cite{DiMi,GeMa} for much more background.\\

In \cite{LuYu1,LuYu}, Lusztig and Yun set out to make \eqref{eq:1} canonical, in geometric
terms. This required several modifications. Firstly, they do not work with rational but with
geometric series. A maximal $\F_q$-torus $\mc T$ in $\mc G$ is fixed, say
with Weyl group $W$. Characters of $T$ are replaced by rank one character sheaves on $\mc T$.
These are not required to be stable under the action of a Frobenius element $\Fr$, only the 
$W$-orbit of $\mc L$ must be $\Fr$-stable. Such data $(W\mc L,\mc T)$ correspond to precisely
one $\Fr$-stable geometric conjugacy class in $G^\vee$.

Secondly, the main results of Lusztig--Yun involve categories, not only irreducible 
representations. Instead of $\Irr_x (G)$ (given some data $x$ for which this is defined) 
we have $\Rep_x (G)$, the category of finite dimensional representations generated by 
$\Irr_x (G)$.
In particular $\Rep_{W \mc L}(G)$ denotes the geometric series associated to $W\mc L$.

Thirdly, $Z_{G^\vee}(s)$ is replaced by the endoscopic group $H$ dual to 
$Z_{G^\vee}(s)^\circ$ plus some group $\Omega_{\mc L}$ isomorphic to $\pi_0 (Z_{\mc G}(s))^\Fr$.
In this way the use of $G^\vee$ can be avoided in \cite{LuYu1,LuYu}.

With impressive geometric machinery Lusztig and Yun obtain equivalences between
subcategories of $\Rep_{W\mc L}(G)$ (determined by cells) and subcategories of $\Rep_1 (H)$
enriched with $\Omega_{\mc L}$-equivariant structures. Apart from the choice of a Whittaker
datum for $G$, these equivalences are canonical. This can be regarded as an endoscopic
categorical version of Deligne--Lusztig theory. As we will see, it is even more canonical
to use $\Rep_{\mc L} (H)$ instead of $\Rep_1 (H)$. Since $\mc L$ defines a character $\chi_\cL$
of $H$ and $\Rep_{\mc L} (H)$ is obtained from $\Rep_1 (H)$ by tensoring with $\chi_\cL$, 
$\Rep_{\mc L} (H)$ consists of essentially unipotent representations of $H$.\\

The aims of this paper are threefold. Firstly, we want to make the final results of 
\cite{LuYu} more canonical, and explain them in simpler representation theoretic terms. 
Secondly, we want to find a version of these results for rational series, so a canonical 
Jordan decomposition of irreducible $G$-representations, like \eqref{eq:1}. This may have 
been known to Lusztig and Yun, but it is not stated in \cite{LuYu1,LuYu}. 

Thirdly, we want to generalize these canonical equivalences of categories to disconnected 
reductive $\F_q$-groups. The simple Example \ref{ex:4.6} shows that for disconnected $G$ 
unipotent representations of $H$ are not good enough, we have to use the essentially 
unipotent representations from $\Rep_{\mc L}(H)$. In other words, instead of \eqref{eq:1}
we expect a relation between $\Irr_s (G)$ and $\Irr_s (Z_{G^\vee}(s))$.

In a sense, disconnected reductive groups form the natural class of 
groups for such endoscopic results. Namely, \eqref{eq:1} already shows that to analyse
irreducible representations of connected reductive $\F_q$-groups, one has to involve 
(essentially) unipotent representations of disconnected reductive $\F_q$-groups. 
If instead we start with a disconnected reductive $\F_q$-group, then we still need to study  
essentially unipotent representations of other disconnected reductive $\F_q$-groups, 
but we stay within the same class of groups. \\

Now we discuss our main results in more detail, the connected case first. 
Let $\mc G^\circ$ be a connected reductive $\F_q$-group, with a maximally split maximal
$\F_q$-torus $\mc T$. Let $\mc L$ be a rank one character sheaf on $\mc T$, such that 
$W(\mc G^\circ,\mc T) \mc L$ is $\Fr$-stable. 

Let $\mc H$ be the connected endoscopic group of $\mc G^\circ$ determined by $(\mc L,\mc T)$,
see Paragraph \ref{par:endo}.
To $(\mc L,\mc T)$ we also associate a group $\Omega_{\mc L}^\circ$, a subquotient of
$W(\mc G^\circ,\mc T)$ that plays the role of the component group of some 
version of the stabilizer group of $\mc L$. We need a certain set $\mc B_{\mc L}^\circ$ with 
an action of $\Omega_{\mc L}^\circ$, such that $\mf B_{\mc L}^\circ / \Omega_{\mc L}^\circ$
parametrizes the rational conjugacy classes in one geometric conjugacy class in 
$G^{\circ \vee}$. Every $\beta \in \mf B_{\mc L}^\circ$ gives rise to a $\F_q$-form of 
$\mc H$, say $\mc H^\beta$. 

\begin{thmintro}\label{thm:A} \textup{(see Theorem \ref{thm:E.1})} \\
There exists an equivalence of categories
\[
\Rep_{W(\mc G^\circ,\mc T) \mc L} (G^\circ) \cong \big( \bigoplus\nolimits_{\beta \in 
\mf B_{\mc L}^\circ} \Rep_1 ( H^\beta ) \big)^{\Omega_{\mc L}^\circ} .
\]
It is canonical up to the choice of a Whittaker datum for $G^\circ$.
\end{thmintro}

Let us clarify what canonical means in Theorem \ref{thm:A}, and in all the main results of
this paper. There are several choices involved in the construction. Whenever a new choice
is made, it is related to the original choice by a unique isomorphism. This notion of
canonicity permeates the paper, as it does \cite{LuYu}. It already applies to the
construction of $\mc H$.

The superscript $\Omega_{\mc L}^\circ$ in Theorem \ref{thm:A} says that we enrich the 
objects in the category between brackets with $\Omega_{\mc L}^\circ$-equivariant structures. 
Thus Theorem \ref{thm:A} describes $G^\circ$-representations with semisimple geometric 
parameter $W(\mc G^\circ,\mc T) \mc L$ in 
terms of unipotent representations of forms of an endoscopic group $\mc H$ plus a kind of 
action of the finite group $\Omega_{\mc L}^\circ$. Theorem \ref{thm:A} is a consequence of 
\cite[Corollary 12.7]{LuYu}, the main novelty is that formulate it without cells.

Recall that $\Rep_s (G^\circ)$ denoted the category generated by a rational series of 
$G^\circ$-representations, parametrized by a semisimple element $s \in G^{\circ \vee \epsilon}$.

\begin{thmintro}\label{thm:B} \textup{(see Theorem \ref{thm:E.7})} \\
Suppose that the geometric conjugacy class of $s \in G^{\circ \vee \epsilon}$ corresponds to
$W(\mc G^\circ,\mc T) \mc L$. This determines a unique $\beta \in \mf B_{\mc L}^\circ$, 
and Theorem \ref{thm:A} restricts to an equivalence of categories
\[
\Rep_s (G^\circ) \cong \Rep_1 \big( H^\beta \big)^{\Omega_{\mc L,\beta}^\circ} .
\]
\end{thmintro}

The equivalences in Theorems \ref{thm:A} and \ref{thm:B} can be made canonical by replacing
$\Rep_1 (H^\beta)$ with $\Rep_{\mc L}(H^\beta)$, then we do not need the Whittaker data
anymore.

We come to the disconnected case. Let $\mc G$ be a smooth $\F_q$-group scheme
whose neutral component $\mc G^\circ$ is reductive. Only the connected components of $\mc G$
which are stable under a Frobenius element $\Fr$ contribute to $G = \mc G (\F_q)$, so we may 
assume that all connected components of $\mc G$ are $\Fr$-stable. 
We allow $\pi_0 ( \mc G)$ to be infinite, but we require that all irreducible representations
of $\Qe G$ have finite dimension.
We define the geometric series $\Rep_{\mc W (\mc G,\mc T) \mc L}(G)$ as consisting of
the $G$-repre\-sen\-ta\-tions $\pi$ such that $\Res^G_{G^\circ} \pi$ lies in the subcategory 
generated by $\bigcup_{w \in W(\mc G,\mc T)} \Rep_{W(\mc G^\circ,\mc T) w \mc L} (G^\circ)$.

We use the above notations for $\mc G^\circ$, but now without the superscripts $\circ$. In
addition, we have a disconnected $\F_q$-group $\mc H_{\mc L}^\Omega (\beta,\beta)$ with neutral 
component $\mc H^\beta$ and component group $\Omega_{\mc L,\beta}$.

\begin{thmintro}\label{thm:C} \textup{(see Theorem \ref{thm:E.4})} \\
There exists canonical equivalences of categories
\[
\Rep_{W(\mc G,\mc T) \mc L} (G) \cong \big( \bigoplus\nolimits_{\beta \in \mf B_{\mc L}^\circ}
\Rep_\cL ( H^\beta ) \big)^{\Omega_{\mc L}} \cong
\bigoplus\nolimits_{\beta \in \mf B_{\mc L}^\circ / \Omega_{\mc L}} 
\Rep_\cL \big( H_{\mc L}^\Omega (\beta,\beta) \big) .
\]
\end{thmintro}

We define the rational series $\Rep_s (G)$ as the category of $G$-representations $\pi$ 
such that $\Res^G_{G^\circ} \pi$ lies in the subcategory generated by 
$\bigcup_{g \in G} \mr{Ad}(g)_* \Rep_s (G^\circ)$.

\begin{thmintro}\label{thm:D} \textup{(see Theorem \ref{thm:E.8})} \\
Suppose that the geometric conjugacy class of $s \in G^{\circ \vee \epsilon}$ corresponds to
$W(\mc G^\circ,\mc T) \mc L$. This determines a unique $\beta \in \mf B_{\mc L}^\circ$, 
and Theorem \ref{thm:C} restricts to canonical equivalences of categories
\[
\Rep_s (G) \cong \Rep_{\mc L} \big( H^\beta \big)^{\Omega_{\mc L,\beta}} \cong 
\Rep_{\mc L} \big( H_{\mc L}^\Omega (\beta,\beta) \big) .
\]
\end{thmintro}

This can be regarded as a Jordan decomposition for representations of the possibly disconnected
reductive group $G = \mc G (\F_q)$: it expresses $G$-representations with
semisimple parameter $s$ in terms of essentially unipotent representations of a disconnected
endoscopic group $H_{\mc L}^\Omega (\beta,\beta)$, for which the semisimple parameter is
central. This existed already (with unipotent $H^\beta$-representations) when $\mc G$ is 
connected \cite{Lus2}, but even in that case Theorem \ref{thm:D} adds canonicity. 

We establish Theorems \ref{thm:C} and \ref{thm:D} in two steps: first when $\mc G$ is connected
and then when $\mc G$ is disconnected. In the connected case (Section \ref{sec:2}) we vary 
on Theorems \ref{thm:A} and \ref{thm:B} and on the underlying arguments from \cite{LuYu}. 
An important aspect is the notion of equivariant objects, which we discuss in some detail in
Appendix \ref{sec:A}. It turns out that the equivariant structures used in Theorems \ref{thm:A} 
and \ref{thm:B} are not good enough for our further purposes. This is illustrated by the
canonical equivalence 
\begin{equation}\label{eq:2}
\Rep_{\mc L} \big( H^\beta \big)^{H_{\mc L}^\Omega (\beta,\beta) / H^\beta} \cong 
\Rep_{\mc L} \big( H_{\mc L}^\Omega (\beta,\beta) \big) 
\end{equation}
from Lemma \ref{lem:A.2},
which does not hold in the setting of Theorems \ref{thm:A} and \ref{thm:B}. To improve that, 
we revisit many of the technical parts of \cite{LuYu}, about equivariant perverse sheaves. 
We make several of these parts more canonical, and we adjust them so that they involve the 
optimal equivariant structures. The disconnected cases of Theorems \ref{thm:C} and 
\ref{thm:D} follow from the connected cases via 
general properties of equivariant structures worked out in Appendix \ref{sec:A}.\\

We expect that our results will have applications to depth zero representations of reductive
$p$-adic groups. The main reason is that by \cite{MoPr} every supercuspidal depth zero
representation of such a group, say $M$, is obtained from an irreducible representation of
a group of the form $\mc G (\F_q)$. In that setting $\mc G (\F_q)$ arises as a quotient of 
the $M$-stabilizer of a vertex in the reduced Bruhat--Tits building of $M$.

In the recent preprint \cite{Fuj}, depth zero supercuspidal representations of simple $p$-adic
groups were analysed, with several techniques among which some version of a Jordan decomposition
for $\Irr (\mc G (\F_q))$. With Theorems \ref{thm:C} and \ref{thm:D} it should be possible to 
make the results in \cite[\S 4]{Fuj} applicable in larger generality.

\renewcommand{\theequation}{\arabic{section}.\arabic{equation}}
\numberwithin{equation}{section}

\section{Series and unipotent representations for connected reductive groups} 
\label{sec:1}

We start by recalling a part of the setup and the relevant results from \cite{LuYu1,LuYu}.
Let $\mc G$ be a connected reductive group defined over the algebraic closure of a finite
field $\F_q$. Let $\ell$ be a prime number different from $p = \mr{char}(\F_q)$. All our
representations will be on vector spaces over $\Qe$.

Let $\epsilon$ be the action of a geometric Frobenius element $\Fr \in 
\mr{Gal}(\overline{\F_q} / \F_q)$, for some rational structure of $\mc G$. We denote the 
corresponding $\F_q$-group by $\mc G^\epsilon$, and we write
\[
G^\epsilon = \mc G^\epsilon (\F_q) = \mc G \big( \overline{\F_q} \big)^\epsilon .
\]
We fix an $\epsilon$-stable Borel subgroup $\mc B$ of $\mc G$, with unipotent radical
$\mc U$. Let $\mc T$ be an $\epsilon$-stable maximal torus in $\mc B$, so $\mc T^\epsilon$ 
is a maximally split maximal $\F_q$-torus in $\mc G^\epsilon$. Let $\Phi^+$ be the 
corresponding set of positive roots in $\Phi = \Phi (\mc G,\mc T)$. We write 
$W = W(\mc G,\mc T)$ and for each $w \in W$ we fix a representative 
$\dot{w} \in N_{\mc G}(\mc T)$. 

Let $\mc L$ be a rank one character sheaf on $\mc T$, and assume that $W \mc L$ is 
$\epsilon$-stable. Whenever $w \epsilon \mc L = \mc L$, $\mc L$ can be regarded as a character 
sheaf on $\mc T^{w \epsilon}$, stable under the Frobenius action $\mr{Ad}(w) \circ \epsilon$ of
that $\F_q$-torus. Then $\mc L$ is equivalent to the data of a character
$T^{w \epsilon} \to \Qe^\times$. Let $W_{\mc L}^\circ$ be the Weyl group of the root system 
\[
\Phi_{\mc L} = \{ \alpha \in \Phi : (\alpha^\vee)^* \mc L \text{ is trivial} \}
\] 
Then $W_{\mc L}^\circ$ is a normal subgroup of $W_{\mc L}$, the stabilizer of $\mc L$ in $W$.

The orbit $W \mc L$ corresponds to a Fr-stable semisimple conjugacy class in the 
reductive $\F_q$-group $\mc G^\vee$ dual to $\mc G$ \cite[\S 5]{DeLu}.
Deligne--Lusztig theory yields the category $\Rep_{W \mc L}(G^\epsilon)$ generated
by irreducible $G^\epsilon$-representations with semisimple parameter $W \mc L$. 
The set of irreducible representations in $\Rep_{W \mc L}(G^\epsilon)$ is 
called the geometric series with parameter $\mc L$ or $W \mc L$. In particular
$\Rep_1 (G^\epsilon)$ is the category of finite dimensional unipotent 
$G^\epsilon$-representations.

We can decompose the $W$-stabilizer of $\mc L$ as 
\[
W_{\mc L} = W_{\mc L}^\circ \rtimes \Omega_{\mc L} ,
\]
where $\Omega_{\mc L}$ is the $W_{\mc L}$-stabilizer of 
$\Phi_{\mc L}^+ = \Phi^+ \cap \Phi_{\mc L}$. The set
\[
\mf B_{\mc L} = W_{\mc L}^\circ \, \backslash \{ w \in W : w \epsilon \mc L = \mc L \} 
\]
is an $\Omega_{\mc L}$-torsor, for the left multiplication action. The group 
$\Omega_{\mc L} \cong W_{\mc L}^\circ \, \backslash W_{\mc L}$ also acts on 
$\mc B_{\mc L}$ by 
\begin{equation}\label{eq:E.26}
\mr{Ad}_\epsilon (\omega) w = \omega w \epsilon (\omega)^{-1}.
\end{equation}
Each class $\beta \in \mf B_{\mc L}$ contains a unique element $w^\beta$ of 
minimal length, which sends $\epsilon \Phi_{\mc L}^+$ to $\Phi_{\mc L}^+$.
The set $\{ w^\beta : \beta \in \mf B_{\mc L} \}$ is stable under
$\mr{Ad}_\epsilon (\Omega_{\mc L})$.

Any other $\epsilon$-stable maximal torus in $\mc G$ has the form 
$g \mc T g^{-1}$ for some $g \in \mc G$. There are bijections
\[
\begin{array}{cccccc}
W_{\mc L} & \to & W_{g \mc L} \, : & w & \mapsto & g w g^{-1} , \\
\mf B_{\mc L} & \to & \mf B_{g \mc L} \, : & \beta & \mapsto & g \beta \epsilon (g)^{-1} ,
\end{array}
\]
which induce a bijection 
\[
\mf B_{\mc L} / \mr{Ad}_\epsilon (\Omega_{\mc L}) \to 
\mf B_{g \mc L} / \mr{Ad}_\epsilon (\Omega_{g \mc L}) .
\]
The data $(\mc L, \mc T, \beta)$ are considered equivalent to $(g \mc L, g \mc T g^{-1}, 
g \beta \epsilon (g)^{-1})$, and all the below constructions from these two sets of data 
related by conjugation by $g$.

\subsection{The endoscopic group and geometric series} \
\label{par:endo}

Let $\mc H$ be the connected reductive $\overline{\F_q}$-group with maximal torus 
$\mc T$ and root system $\Phi_{\mc L}$. This $\mc H$ is called the endoscopic group 
of $\mc G$ associated to $\mc L$. 

We endow $\mc H$ with a pinning relative to $\mc G$, as in \cite[\S 10.2]{LuYu}. Let 
$\mc B_{\mc H} \supset \mc T$ be the Borel subgroup of $\mc H$ determined by 
$\Phi_{\mc L}^+$. As explained in \cite[\S 12.1]{LuYu}, the relative pinning and the lift 
$\dot{w}^\beta$ of $w^\beta$ determine an automorphism 
\begin{equation}\label{eq:E.50}
\sigma_{\beta \epsilon} \quad \text{of} \quad (\mc H, \mc T, \mc B_{\mc H}),
\end{equation} 
whose action on $\mc T$ and on the root subgroups in $\mc H$ corresponds to the action of 
$\mr{Ad}(\dot{w}^\beta) \circ \epsilon$ on $\mc T$ and on the root subgroups in 
$\mc G$. This provides a Frobenius action on $\mc H$, and hence 
a $\F_q$-form $\mc H^{\sigma_{\beta \epsilon}}$ of $\mc H$. Up to a unique isomorphism,
it does not depend on the choice of $\dot{w}^\beta$.

We can identify $N_{\mc H}(\mc T)$ with  the inverse image of 
$W_{\mc L}^\circ$ in $N_{\mc G}(\mc T)$.
Let $N_{\mc G}(\mc T)_{\mc L} = \mr{Stab}_{\mc G}(\mc T,\mc L)$ be the inverse image of
$W_{\mc L}$ in $N_{\mc G}(\mc T)$ and let 
\[
N_{\mc G}(\mc T)_{\mc L}^+ = \mr{Stab}_{\mc G}(\mc T,\mc L, \Phi_{\mc L}^+)
\] 
be the inverse image of $\Omega_{\mc L}$ in $N_{\mc G}(\mc T)$. We record that
\begin{equation}\label{eq:E.52}
N_{\mc G}(\mc T)_{\mc L} / \mc T = W_{\mc L}, \quad 
N_{\mc G}(\mc T)_{\mc L}^+ / \mc T = \Omega_{\mc L} \quad \text{and} \quad
N_{\mc G}(\mc T)_{\mc L} = N_{\mc G}(\mc T)_{\mc L}^+ \overset{\mc T}{\times} N_{\mc H}(\mc T) .
\end{equation}
There are several useful groups (and groupoids) with neutral component $\mc H$.
Every $n \in N_{\mc G}(\mc T)_{\mc L}^+$ determines an automorphism of 
$(\mc H,\mc T, \mc B_{\mc H})$, which is $\mr{Ad}(n)$ on $\mc T$ and on every root subgroup. 
That yields the $\overline{\F_q}$-group
\begin{equation}\label{eq:E.51}
N_{\mc G}(\mc T)_{\mc L}^+ \overset{\mc T}{\times} \mc H = 
\mc H \overset{\mc T}{\times} N_{\mc G}(\mc T)_{\mc L}^+ ,
\end{equation}
which has component group $\Omega_{\mc L}$ and contains 
$N_{\mc G}(\mc T)_{\mc L} = N_{\mc G}(\mc T)_{\mc L}^+ N_{\mc H}(\mc T)$.
The group \eqref{eq:E.51} acts naturally on $\mf B_{\mc L}$, with $\mc H$ stabilizing
every element. That gives rise to the groupoid $\mc H_{\mc L}^\Omega$ over 
$\mf B_{\mc L}$ with morphisms
\[
\mc H_{\mc L}^\Omega (\beta,\beta') = \{ n \in N_{\mc G}(\mc T)_{\mc L}^+ : 
\mr{Ad}_\epsilon (n \mc T) \beta = \beta' \} \overset{\mc T}{\times} \mc H .
\]
Let $\Omega_{\mc L,\beta}$ be the $\Omega_{\mc L}$-stabilizer of $\beta$ and let
$N_{\mc G}(\mc T)_{\mc L}^\beta = \mr{Stab}_{\mc G}(\mc T,\mc L,\beta)$ be its inverse
image in $N_{\mc G}(\mc T)_{\mc L}$. With these notations, the stabilizer of $\beta$ in 
$N_{\mc G}(\mc T)_{\mc L}^+ \overset{\mc T}{\times} \mc H$ is the possibly disconnected group
\[
\mc H_{\mc L}^\Omega (\beta,\beta) = 
N_{\mc G}(\mc T)_{\mc L}^\beta \overset{\mc T}{\times} \mc H ,
\]
with component group $\Omega_{\mc L,\beta}$.

We also need $\F_q$-structures on the above groups and groupoids. We can regard
$\mc G \epsilon = \epsilon \mc G$ as a coset of $\mc G$ in a semidirect product
$\mc G \rtimes \Z$. For each $n \in N_{\mc G}(\mc T)_{\mc L}^+$ and $\beta \in \mf B_{\mc L}$
we have 
\[
n \dot w^\beta \epsilon n^{-1} = 
n \dot w^\beta \epsilon (n)^{-1} \epsilon \in \mc T \dot w^{\beta'} \epsilon 
\quad \text{where} \quad \beta' = \mr{Ad}_\epsilon (n \mc T) \beta \in \mf B_{\mc L} .
\]
Lang's theorem ensures the existence of a $t \in \mc T$ such that
$t n \dot w^\beta \epsilon n^{-1} t^{-1} = \dot w^{\beta'} \epsilon$,
and it is unique up to left multiplication by $T^{\sigma_{\beta' \epsilon}}$. This means
that for each $\omega \in \Omega_{\mc L}$ and each $\beta \in \mf B_{\mc L}$ with
$\mr{Ad}_\epsilon (\omega) \beta = \beta'$, we can select a representative $n(\omega,\beta) 
\in N_{\mc G}(\mc T)_{\mc L}$ for $\omega$ (unique up $T^{\sigma_{\beta' \epsilon}}$
from the left and up to $T^{\sigma_{\beta \epsilon}}$ from the right), such that 
\begin{equation}\label{eq:E.39}
\begin{aligned}
& \dot w^{\beta'} \epsilon = n (\omega,\beta) \dot w^\beta \epsilon \, n(\omega,\beta)^{-1} , \\
& \sigma_{\mr{Ad}_\epsilon (\omega) \beta \epsilon} = 
\mr{Ad}(n(\omega,\beta)) \sigma_{\beta \epsilon} \mr{Ad}(n(\omega,\beta))^{-1} .
\end{aligned}
\end{equation}
We note that $\mr{Ad}(n(\omega,\beta))$ gives group isomorphisms
\[
\mc H^{\sigma_{\beta \epsilon}} \to \mc H^{\sigma_{\beta' \epsilon}} \quad \text{and} \quad
H^{\sigma_{\beta \epsilon}} \to H^{\sigma_{\beta' \epsilon}}.
\]
Next we define the groupoid $H^\Omega_{\mc L}$ over $\mf B_{\mc L}$ with morphisms
\[
H^\Omega_{\mc L}(\beta,\beta') = \{ n (\omega,\beta) : \omega \in \Omega_{\mc L},
\mr{Ad}_\epsilon (\omega) \beta = \beta' \} H^{\sigma_{\beta \epsilon}} .
\]
For $\beta' = \beta$ we obtain the groups
\begin{equation}\label{eq:3.52}
\begin{aligned}
& H^\Omega_{\mc L}(\beta,\beta) = N_G (T)_{\mc L}^{\beta,\sigma_{\beta \epsilon}}
\overset{T^{\sigma_{\beta \epsilon}}}{\times} H^{\sigma_{\beta \epsilon}} , \\
& N_G (T)_{\mc L}^{\beta,\sigma_{\beta \epsilon}} =
\{ n (\omega,\beta) : \omega \in \Omega_{\mc L,\beta} \} T^{\sigma_{\beta \epsilon}} ,
\end{aligned}
\end{equation}
which we regard as the $\F_q$-points of $\mc H^\Omega_{\mc L} (\beta,\beta)$ and of
$N_{\mc G}(\mc T)_{\mc L}^\beta$, respectively. We note that $N_{\mc G}(\mc T)_{\mc L}^\beta$
and $N_{G}(T)_{\mc L}^{\beta, \sigma_{\beta \epsilon}}$ have the same quotient 
\begin{equation}\label{eq:3.19}
\Omega_{\mc L,\beta} = N_{G}(T)_{\mc L}^{\beta, \sigma_{\beta \epsilon}} 
/ T^{\sigma_{\beta \epsilon}} = N_{\mc G}(\mc T)_{\mc L}^{\beta} / \mc T .
\end{equation}
The following result is a version of \cite[Corollary 12.7]{LuYu} without cells.

\begin{thm}\label{thm:E.1}
Let $\mc G^\epsilon$ be a connected reductive $\F_q$-group and let $W \mc L$ be an 
$\epsilon$-stable orbit of character sheaves on $\mc T$. There are
equivalences of categories
\[
\Rep_{W \mc L} (G^\epsilon) \cong \Big( \bigoplus_{\beta \in \mf B_{\mc L}} 
\Rep_1 (H^{\sigma_{\beta \epsilon}}) \Big)^{\Omega_{\mc L}} 
\cong \bigoplus_{\beta \in \mf B_{\mc L} / \mr{Ad}_\epsilon (\Omega_{\mc L})} 
\Rep_1 (H^{\sigma_{\beta \epsilon}})^{\Omega_{\mc L,\beta}} .
\]
Once a Whittaker datum for $\mc G$ has been chosen, these equivalences are canonical.
\end{thm}
\textbf{Remark.} Here the analysis of the equivariant structures for $\Omega_{\mc L}$ and 
$\Omega_{\mc L,\beta}$ involves rigidified minimal IC-sheaves \cite[\S 5]{LuYu}, which 
rely on a Whittaker datum. These actions are transferred from a geometric setting
to categories of representations by means of equivalences in \cite[\S 11--12]{LuYu}.
\begin{proof}
Let $\mb c$ be a cell in $W_{\mc L}^\circ$ and let $[\mb c] \subset W \times W \mc L$ be
the associated two-sided cell, as in \cite[\S 11.4]{LuYu}. They are related by
\begin{equation}\label{eq:E.3}
[\mb c] \cap (W_{\mc L}^\circ \times \{\mc L\}) = \Omega_{\mc L} \mb c .
\end{equation}
By construction
\begin{equation}\label{eq:E.2}
\Rep_{W \mc L} (G^\epsilon) = 
 \bigoplus\nolimits_{[\mb c]} \Rep_{W \mc L}^{[\mb c]} (G^\epsilon) . 
\end{equation}
Let $\Omega_{\mb c}$ be the stabilizer of $\mb c$ in $\Omega_{\mc L}$ and write
\[
\mf B_{\mc L,\mb c} = \{ \beta \in \mf B_{\mc L} : w^\beta \epsilon \text{ stabilizes } 
\mb c \} .
\]
By \cite[Corollary 12.7]{LuYu} and Lemma \ref{lem:A.1}.a 
there are canonical equivalences of categories 
\begin{equation}\label{eq:E.1}
\Rep_{W \mc L}^{[\mb c]} (G^\epsilon) \cong 
\bigoplus_{\beta \in \mf B_{\mc L,\mb c} / \mr{Ad}_\epsilon (\Omega_{\mb c})} \Rep_1^{\mb c}
\big( H^{\sigma_{\beta \epsilon}} \big)^{\Omega_{\mb c,\beta}} \cong
\Big( \bigoplus_{\beta \in \mf B_{\mc L, \mb c}} \Rep_1^{\mb c} \big( H^{\sigma_{\beta 
\epsilon}} \big) \Big)^{\Omega_{\mb c}} .
\end{equation}
When $\beta \in \mf B_{\mc L} \setminus \mf B_{\mc L,\mb c}$, $H^{\sigma_{\beta \epsilon}}$
has no nonzero representations associated to the cell $\mb c$, because the cell of an
irreducible representation is unique. Hence the right hand side of \eqref{eq:E.1} can
be rewritten as
\[
\Big( \bigoplus_{\beta \in \mf B_{\mc L}} \Rep_1^{\mb c} \big( H^{\sigma_{\beta \epsilon}} 
\big) \Big)^{\Omega_{\mb c}} \cong
\Big( \bigoplus_{\beta \in \mf B_{\mc L}} \Rep_1^{\Omega_{\mc L} \mb c} 
\big( H^{\sigma_{\beta \epsilon}} \big) \Big)^{\Omega_{\mc L}} .
\]
Now we take the direct sum over all two-sided cells $[\mb c]$ and with \eqref{eq:E.2}
and \eqref{eq:E.3} we obtain
\begin{equation}\label{eq:E.30}
\Rep_{W \mc L} (G^\epsilon) \cong \bigoplus_{[\mb c]} \Big( \bigoplus_{\beta \in 
\mf B_{\mc L}} \Rep_1^{\Omega_{\mc L} \mb c} \big( H^{\sigma_{\beta \epsilon}} \big) 
\Big)^{\Omega_{\mc L}} = \Big( \bigoplus_{\beta \in \mf B_{\mc L}} 
\Rep_1 \big( H^{\sigma_{\beta \epsilon}} \big) \Big)^{\Omega_{\mc L}} . 
\end{equation}
Let $\Omega_{\mc L,\beta}$ be the stabilizer of $\beta$ in $\Omega_{\mc L}$, for the 
action \eqref{eq:E.26}. By Lemma \ref{lem:A.1}.a, the right hand side of 
\eqref{eq:E.30} is canonically equivalent to 
$\bigoplus_{\beta \in \mf B_{\mc L} / \mr{Ad}_\epsilon 
(\Omega_{\mc L})} \Rep_1 (H^{\sigma_{\beta \epsilon}})^{\Omega_{\mc L,\beta}}$.
\end{proof}

\subsection{Rational series} \

We turn our attention to rational series for $G^\epsilon$. Let $\mc G^\vee$ be
the dual group of $\mc G$. It is endowed with the $\F_q$-structure from the Frobenius 
action dual to $\epsilon$, which we shall also denote $\epsilon$. 
Recall that every rational series is parametrized by the (rational) conjugacy 
class of a semisimple element $s$ of $G^{\vee \epsilon} = \mc G^\vee (\F_q)$. 

We fix an embedding
\[
\overline{\F_q}^\times \to \Qe^\times
\]
and a group isomorphism
\[
\overline{\F_q}^\times \cong \{ m/n \in \Q : m,n \in \Z, n \notin p \Z \} / \Z .
\]
These give rise to an isomorphism
\begin{equation}\label{eq:E.38}
\begin{array}{rcl}
\text{group of character sheaves on } \mc T & \isom & \mc T^\vee (\overline \F_q) =
X^* (\mc T) \otimes_\Z \overline{\F_q}^\times \\
\mc L & \mapsto & s_{\mc L} 
\end{array},
\end{equation}
see \cite[\S 2.2]{LusChar1}. The isomorphism \eqref{eq:E.38} is equivariant for the
actions on $\epsilon$ and of $W = W(\mc G,\mc T) \cong W (\mc G^\vee,\mc T^\vee)$. 
In particular
\[
\begin{array}{lll}
\epsilon \mc L = \mc L & \Longleftrightarrow & s_{\mc L} \in 
T^{\vee \epsilon} = \mc T^\vee (\overline \F_q)^\epsilon ,\\
W \mc L \text{ is } \epsilon\text{-stable} & \Longleftrightarrow &
W s_{\mc L} \subset \mc T^\vee \text{ is } \epsilon\text{-stable} .
\end{array}
\]
We recall that $W s_{\mc L}$ is the intersection of the geometric conjugacy
class $\mr{Ad}(\mc G^\vee) s_{\mc L}$ with $\mc T^\vee$. This is a union of a finite
number of rational conjugacy classes in $\mc G^\vee$, and via \eqref{eq:E.38} 
it corresponds to $W \mc L$. Analgously $\Rep_{W \mc L}(G^\epsilon)$ is a direct sum of 
finitely many rational series. We want to parametrize those in terms of $\mc L$.

\begin{prop}\label{prop:E.6}
Assume that $W \mc L$ is $\epsilon$-stable.
The quotient set $\mf B_{\mc L} / \mr{Ad}_\epsilon (\Omega_{\mc L})$ canonically 
parametrizes the $G^{\vee \epsilon}$-conju\-ga\-cy classes in 
$(\mr{Ad}(\mc G^\vee) s)^\epsilon$. 
The parametrization can be realized as follows. For $w \in W$ representing 
$\beta \in \mf B_{\mc L}$, find $h_w \in \mc G^\vee$ such that $h_w \epsilon (h_w^{-1})$ 
represents $w$ as element of $W \cong N_{\mc G^\vee}(\mc T^\vee) / \mc T^\vee$. 
Then the bijection sends $\beta = [w] \in \mf B_{\mc L} / \mr{Ad}_\epsilon 
(\Omega_{\mc L})$ to $\mr{Ad}(G^{\vee \epsilon}) (h_w^{-1} s_{\mc L} h_w)$.
\end{prop} 
\begin{proof}
For any lift $\dot w \in N_{\mc G^\vee}(\mc T^\vee)$ of $w \in W$, the condition
$h_w \epsilon (h_w^{-1}) = \dot w$ determines $h_w$ up to right multiplication by 
elements of $G^{\vee \epsilon}$. In view of the surjectivity of
\[
\mc T^\vee \to \mc T^\vee : t \mapsto t w \epsilon (t)^{-1} ,
\]
the condition $h_w \epsilon (h_w^{-1}) \in \dot w \mc T^\vee$ determines precisely the
double coset $\mc T h_w G^{\vee \epsilon}$. For all $g \in G^{\vee \epsilon}, 
t \in \mc T^\vee$ we have
\[
\mr{Ad}(G^{\vee \epsilon}) (g^{-1} h_w^{-1} t^{-1} s_{\mc L} t h_w g) = 
\mr{Ad}(G^{\vee \epsilon}) (g^{-1} h_w s_{\mc L} h_w g) = 
\mr{Ad}(G^{\vee \epsilon}) (h_w^{-1} s_{\mc L} h_w) .
\]
We also note that
\begin{equation}
\epsilon (h_w^{-1} s_{\mc L} h_w) = h_w^{-1} \dot w \epsilon s_{\mc L} \dot w^{-1} h_w
= h_w^{-1} s_{\mc L} h_w . 
\end{equation}
This shows that the map
\begin{equation}\label{eq:E.40}
\begin{array}{ccc}
\{ w \in W : w \epsilon \mc L = \mc L \} & \to & (\mr{Ad}(\mc G^\vee) s_{\mc L})^\epsilon /
\mr{Ad}(G^{\vee \epsilon}) \\
w & \mapsto & h_w^{-1} s_{\mc L} h_w
\end{array}
\end{equation}
is well-defined. For any $w_2 \in W_{\mc L } = W_{s_{\mc L}}$, we have
\[
\{ h \in \mc G^\vee : h \epsilon (h^{-1}) \in w_2 w \epsilon (w_2^{-1}) \mc T^\vee \} = 
w_2 \{ h \in \mc G^\vee h \epsilon (h^{-1}) \in w \mc T^\vee \} = 
w_2 \mc T^\vee h_w G^{\vee \epsilon} .
\]
Any such $h = \dot w_2 t h_w g$ is sent to 
\[
h^{-1} s_{\mc L} h = g^{-1} h_w^{-1} t^{-1} \dot w_2^{-1} s_{\mc L} \dot w_2 t h_w g =
h^{-1} s_{\mc L} h = g^{-1} h_w^{-1} s_{\mc L} h_w g = h_w^{-1} s_{\mc L} h_w .
\]
Thus \eqref{eq:E.40} induces a map
\begin{equation}\label{eq:E.41}
\mf B_{\mc L} / \mr{Ad}_\epsilon (\Omega_{\mc L}) \to 
(\mr{Ad}(\mc G^\vee) s_{\mc L})^\epsilon / \mr{Ad}(G^{\vee \epsilon}) .
\end{equation}
This map is canonical because it is independent of the choices in the construction and
because the setup is basically inversion of the Lang map
\[
\mc G^\vee / G^{\vee \epsilon} \to \mc G^\vee : g \mapsto g \epsilon (g^{-1}) .
\]
Now we check that \eqref{eq:E.41} is a bijection. We pick any $\beta \in \mf B_{\mc L}$,
with representative $\dot w^\beta \in N_{\mc G^\vee}(\mc T^\vee)$. Then
$w^\beta \epsilon \mc L = \mc L$ and $w^\beta \epsilon (s_{\mc L}) = s_{\mc L}$. We also
pick any $h_\beta \in \mc G^\vee$ such that $h_\beta \epsilon (h_\beta^{-1}) = \dot w^\beta$,
and we put
\[
s_\beta = h_\beta^{-1} s_{\mc L} h_\beta \in G^{\vee \epsilon} .
\]
As $\Fr (s_\beta) = \epsilon (s_\beta) = s_\beta$, the $\mc G^\vee$-conjugacy class of 
$s_\beta$ and of $s_{\mc L}$ is isomorphic to $\mc G^\vee / Z_{\mc G^\vee}(s_\beta)$ as 
$\F_q$-variety. This means that we need to parametrize
\begin{equation}\label{eq:E.42}
\big( \mc G^\vee / Z_{\mc G^\vee}(s_\beta) \big)^\Fr \big/ 
\mr{Ad} \big( \mc G^{\vee ^\Fr} \big) .
\end{equation}
Galois cohomology provides an exact sequence
\begin{equation}\label{eq:E.11}
Z_{\mc G^\vee}(s_\beta)^\Fr \to \mc G^{\vee \Fr} \to \big( \mc G^\vee / 
Z_{\mc G^\vee}(s_\beta) \big)^\Fr \to H^1 (\Fr, Z_{\mc G^\vee}(s_\beta)) 
\to H^1 (\Fr ,\mc G^\vee) .
\end{equation}
Here $H^1 (\Fr ,\mc G^\vee)$ is a short notation for the cohomology of 
$\mr{Gal}(\overline{\F_q}/\F_q)$ with coefficients in $\mc G^\vee$. By the surjectivity of
Lang's map on connected $\F_q$-groups \cite[Theorem 1.4.8]{GeMa}, $H^1 (\Fr ,\mc G^\vee)$ 
is trivial and 
\[
H^1 (\Fr, Z_{\mc G^\vee}(s_\beta)) \cong H^1 \big( \Fr, \pi_0 (Z_{\mc G^\vee}(s_\beta)) \big)
\cong \pi_0 (Z_{\mc G^\vee}(s_\beta)) / 
\mr{Ad}_\epsilon \big( \pi_0 (Z_{\mc G^\vee}(s_\beta)) \big) .
\]
Thus \eqref{eq:E.11} says that the set \eqref{eq:E.42} is parametrized naturally by
$H^1 \big( \Fr, \pi_0 (Z_{\mc G^\vee}(s)) \big)$. This bijection sends 
\begin{equation}\label{eq:E.12}
g^{-1} s_\beta g \in (\mr{Ad}(\mc G^\vee) s_\beta)^\Fr \quad \text{to} \quad 
[\Fr \mapsto g \Fr (g^{-1})] \in H^1 (\Fr, Z_{\mc G^\vee}(s_\beta)).
\end{equation}
We need to pass back from $s_\beta$ to $s_{\mc L}$.
For $z \in Z_{\mc G^\vee}(s_\beta)$ we compute
\begin{multline*}
(h_\beta z \epsilon (h_\beta^{-1})) \epsilon (s_{\mc L}) (h_\beta z \epsilon
(h_\beta^{-1}))^{-1} = h_\beta z \epsilon (h_\beta^{-1} s_{\mc L} z^{-1} h_\beta^{-1} \\
= h_\beta z \epsilon (s_\beta) z^{-1} h_\beta^{-1} = h_\beta z s_\beta z^{-1} 
h_\beta^{-1} = h_\beta s_\beta h_\beta^{-1} = s_{\mc L}. 
\end{multline*}
Thus $\mr{Ad}_\epsilon (h_\beta)$ provides maps
\begin{align}
\nonumber & Z_{\mc G^\vee}(s_\beta) \to \{ g \in \mc G^\vee : \mr{Ad}(g) 
\epsilon (s_{\mc L}) = s_{\mc L} \} , \\
\label{eq:E.43} & \pi_0 (Z_{\mc G^\vee}(s_\beta)) \to \pi_0 \big( \{ g \in \mc G^\vee :
\mr{Ad}(g) \epsilon (s_{\mc L}) = s_{\mc L} \} \big) / \mr{Ad}_\epsilon 
\big( \pi_0 (Z_{\mc G^\vee}(s_{\mc L})) \big) .
\end{align}
In \eqref{eq:E.43}, the image $\mr{Ad}_\epsilon (h_\beta) z = h_\beta z \epsilon 
(h_\beta^{-1})$ does not change if we replace $h_\beta$ by any other $h_2 \in \mc G^\vee$
with $h_2 s_\beta h_2^{-1} = s_{\mc L}$, because we divided out by 
$\mr{Ad}_\epsilon \big( \pi_0 (Z_{\mc G^\vee}(s_{\mc L})) \big)$. Furthermore, for 
$h \in Z_{\mc G^\vee}(s_\beta)$ we have
\[
\mr{Ad}_\epsilon (h_\beta) (\mr{Ad}_\epsilon (h) z) =
\mr{Ad}_\epsilon (h_\beta h h_\beta^{-1}) \mr{Ad}_\epsilon (h_\beta) z \quad \text{and}
\quad h_\beta h h_\beta^{-1} \in Z_{\mc G^\vee} (s_{\mc L}) .
\]
It follows that $\mr{Ad}_\epsilon (h_\beta)$ defines a canonical map
\begin{equation}\label{eq:E.44}
H^1 \big( \Fr, \pi_0 (Z_{\mc G^\vee}(s_\beta)) \big) \to 
\pi_0 \big( \{ g \in \mc G^\vee : \mr{Ad}(g) \epsilon (s_{\mc L}) = s_{\mc L} \} \big) 
\big/ \mr{Ad}_\epsilon \big( \pi_0 (Z_{\mc G^\vee}(s_{\mc L})) \big) .
\end{equation}
This is in fact a bijection, with inverse $\mr{Ad}_\epsilon (h_\beta^{-1})$.
By Steinberg's description of the centralizers of semisimple elements in reductive groups
\cite{Ste}, 
\begin{equation}\label{eq:E.45}
\pi_0 (Z_{\mc G^\vee}(s_{\mc L})) \cong W(Z_{\mc G^\vee}(s_{\mc L})^\circ,\mc T^\vee) 
\backslash W_{s_{\mc L}} \cong W_{\mc L}^\circ \backslash W_{\mc L} \cong \Omega_{\mc L} .
\end{equation}
As $\mf B_{\mc L}$ is an $\Omega_{\mc L}$-torsor, we obtain bijections
\begin{multline*}
\mf B_{\mc L} \longleftrightarrow W_{\mc L}^\circ \backslash \{ w \in W : w \epsilon \mc L = 
\mc L \} \longleftrightarrow W(Z_{\mc G^\vee}(s_{\mc L})^\circ,\mc T^\vee) \backslash 
\{ w \in W : w \epsilon (s_{\mc L}) = s_{\mc L} \} \\
\longleftrightarrow 
\pi_0 \big( \{ g \in \mc G^\vee : \mr{Ad}(g) \epsilon (s_{\mc L}) = s_{\mc L} \} \big) .
\end{multline*}
This entails that the target of \eqref{eq:E.43} and \eqref{eq:E.44} is canonically
in bijection with $\mf B_{\mc L} / \mr{Ad}_\epsilon (\Omega_{\mc L})$.
Altogether we constructed bijections
\[
\begin{array}{ccccc}
(\mr{Ad}(\mc G^\vee) s)^\epsilon / G^{\vee \epsilon} & \leftrightarrow & 
H^1 \big( \Fr, \pi_0 (Z_{\mc G^\vee}(s)) \big) & \leftrightarrow &
\mf B_{\mc L} / \mr{Ad}_\epsilon (\Omega_{\mc L}) \\
g^{-1} s_\beta g = g^{-1} h_\beta^{-1} s_{\mc L} h_\beta g & \mapsto &
[\Fr \mapsto g \epsilon (g^{-1})] & \mapsto & h_\beta g \epsilon (h_\beta g)^{-1} \\
h_w^{-1} s_{\mc L} h_w & \mapsfrom &
[\Fr \mapsto h_\beta h_w \epsilon (h_w^{-1}) \epsilon (h_\beta^{-1}) ] & \mapsfrom &
w = h_w \epsilon (h_w^{-1}) \mc T^\vee 
\end{array}
\]
Here the last line recovers \eqref{eq:E.41} and shows that that map is bijective.
\end{proof}

Proposition \ref{prop:E.6} and \eqref{eq:E.38} provide canonical bijections between: 
\begin{itemize}
\item the set of semisimple conjugacy classes in $G^{\vee \epsilon}$,
\item the $W$-orbits of pairs $(\mc L,\beta)$, where $\mc L$ is a character sheaf on 
$\mc T$ with $\epsilon$-stable $W$-orbit and 
$\beta \in \mf B_{\mc L} / \mr{Ad}_\epsilon (\Omega_{\mc L})$,
\item triples $(\mc L', \mc T', \beta')$ where $\mc T'$ is an $\epsilon$-stable maximal 
torus in $\mc G$, $(\mc L', \beta')$ is as in the previous bullet and any such triple of 
the form $(g \mc L', g \mc T' g^{-1}, g \beta' \epsilon (g)^{-1})$ with 
$g \in \mc G$ is considered as equivalent to $(\mc L', \mc T', \beta')$.
\end{itemize}
Let us make this more explicit, starting from $W (\mc L,\beta)$. As in the proof of
Proposition \ref{prop:E.6}, pick $h_\beta \in \mc G^\vee$ such that
$h_\beta \epsilon (h_\beta^{-1}) \in w^\beta \mc T^\vee$. Then Proposition 
\ref{prop:E.6} produces $s_\beta = h_\beta^{-1} s_{\mc L} h_\beta \in G^{\vee \epsilon}$.
We record this as 
\begin{equation}\label{eq:E.23}
W (\mc L, \beta) \longleftrightarrow s_\beta .
\end{equation}

\begin{lem}\label{lem:2.4}
The group $\Omega_{\mc L, \beta}$ is isomorphic to 
$\pi_0 (Z_{\mc G^\vee}(s_\beta))^\epsilon$.
\end{lem}
\begin{proof}

Analogous to for $\mc L$, we define
\begin{align*}
& W_{s_{\mc L}}^\circ = W (Z_{\mc G^\vee}(s_{\mc L})^\circ, \mc T^\vee), \\
& \Omega_{s_{\mc L}} = W_{s_{\mc L}}^\circ \backslash W_{s_{\mc L}} ,\\
& \mf B_{s_{\mc L}} = W_{s_{\mc L}}^\circ \backslash \{ w \in W (\mc G^\vee,\mc T^\vee) :
w \epsilon (s_{\mc L}) = s_{\mc L} \} .
\end{align*}
Then \eqref{eq:E.38} and $W(\mc G,\mc T) \cong W(\mc G^\vee,\mc T^\vee)$ induce isomorphisms
\[
W_{\mc L}^\circ \to W_{s_{\mc L}}^\circ ,\; \Omega_{\mc L} \to \Omega_{s_{\mc L}} ,\;
\mf B_{\mc L} \to \mf B_{s_{\mc L}} ,\; \Omega_{\mc L,\beta} \to \Omega_{s_{\mc L},\beta} .
\]
With this we can rewrite
\begin{multline}\label{eq:E.47}
\Omega_{\mc L,\beta} \cong \Omega_{s_{\mc L},\beta} = 
\{ \omega \in \Omega_{s_{\mc L}} : \omega w^\beta \epsilon (\omega^{-1}) = w^\beta \} =\\
\{ \omega \in \Omega_{s_{\mc L}} : \omega w^\beta \epsilon \omega^{-1} = w^\beta \epsilon \} =
\{ \omega \in W_{s_{\mc L}}^\circ \backslash W_{s_{\mc L}} : 
\omega w^\beta \epsilon \omega^{-1} \in W_{s_{\mc L}}^\circ w^\beta \epsilon \} .
\end{multline}
Via \eqref{eq:E.45}, this becomes
\begin{equation}\label{eq:E.46}
Z_{\mc G^\vee}(s_{\mc L})^\circ \; \backslash 
\{ g \in Z_{\mc G^\vee}(s_{\mc L}) : g \dot w^\beta \epsilon g^{-1} \in 
Z_{\mc G^\vee}(s_{\mc L})^\circ \dot w^\beta \epsilon \} .
\end{equation}
Recall that $h_\beta^{-1} s_{\mc L} h_\beta = s_\beta$ and $h_\beta \epsilon (h_\beta^{-1})
= \dot w^\beta$. Conjugation with $h_\beta^{-1}$ sends \eqref{eq:E.46} to
\begin{equation}\label{eq:E.48}
Z_{\mc G^\vee}(s_\beta)^\circ \; \backslash 
\{ g \in Z_{\mc G^\vee}(s_\beta) : g (h_\beta^{-1} \dot w^\beta \epsilon h_\beta) g^{-1} 
\in Z_{\mc G^\vee}(s_\beta)^\circ \dot (h_\beta^{-1} \dot w^\beta \epsilon h_\beta) \} .
\end{equation}
Since $h_\beta^{-1} \dot w^\beta \epsilon h_\beta = \epsilon$, this simplifies to
\begin{equation}\label{eq:E.49}
Z_{\mc G^\vee}(s_\beta)^\circ \; \backslash 
\{ g \in Z_{\mc G^\vee}(s_\beta) : g \epsilon g^{-1} 
\in Z_{\mc G^\vee}(s_\beta)^\circ \epsilon \} = 
\{ x \in \pi_0 (Z_{\mc G^\vee}(s_\beta)) : x \epsilon = \epsilon x \} .
\end{equation}
This group equals $\pi_0 (Z_{\mc G^\vee}(s_\beta))^\epsilon$.
Finally, we combine that with \eqref{eq:E.47}--\eqref{eq:E.49}.
\end{proof}

We denote the rational series (in $\Rep (G^\epsilon)$) associated to $s_\beta \in G^{\vee
\epsilon}$ by $\Rep_{s_\beta} (G^\epsilon)$. In view of the correspondence \eqref{eq:E.23} 
we may write $\Rep_{s_\beta}(G^\epsilon) = \Rep_{W(\mc L,\beta)} (G^\epsilon)$. 
Proposition \ref{prop:E.6} shows how to write a geometric series as a direct sum of 
rational series, with $s_\beta$ as in \eqref{eq:E.23}:
\begin{equation}\label{eq:E.14}
\Rep_{W\mc L}(G^\epsilon) = \bigoplus_{\beta \in \mf B_{\mc L} / 
\mr{Ad}_\epsilon (\Omega_{\mc L})} \Rep_{W(\mc L,\beta)}(G^\epsilon) =
\bigoplus_{\beta \in \mf B_{\mc L} / 
\mr{Ad}_\epsilon (\Omega_{\mc L})} \Rep_{s_\beta}(G^\epsilon) ,
\end{equation}
This is reminiscent of the decomposition of $\Rep_{W\mc L}(G^\epsilon)$ in
Theorem~\ref{thm:E.1}, and indeed will show that these two decompositions coincide.
Let $\mc G \to \mc G_c$ be a regular embedding of reductive $\F_q$-groups 
\cite[\S 1.7]{GeMa}, so $Z(\mc G_c)$ is connected and $\mc G_\der = \mc G_{c,\der}$. 
As is common, we will analyse series for $G^\epsilon$ in terms of those for 
$G_c^\epsilon$. The latter is easier, because the geometric and rational series for
$G_c^\epsilon$ coincide. 

Let $\mc T_c \subset \mc G_c$ be the maximal torus with $\mc T_c \cap \mc G = \mc T$.
The group $W(\mc G_c,\mc T_c) \cong
W(\mc G_c^\vee, \mc T_c^\vee)$ can be identified with $W$. 
Let $(\mc L,\beta)$ be as above, corresponding to the $G^{\vee \epsilon}$-conjugacy
class of $s_\beta$. The quotient map 
\begin{equation}\label{eq:E.31}
G_c^{\vee \epsilon} \to G^{\vee \epsilon} \quad \text{is surjective,}
\end{equation} 
because $\ker (\mc G_c^\vee \to \mc G^\vee)$ is a torus and in particular
connected. Hence we can
find a lift $s_c \in G_c^{\vee \epsilon}$ of $s_\beta$. Via \eqref{eq:E.23} for
$\mc G_c$, $s_c$ corresponds to some $W(\mc L_c, \beta_c)$. Comparing with \eqref{eq:E.23}
for $\mc G^\epsilon$, we see that $\mc L_c$ is an extension to $\mc T_c$ of some 
element of $W \mc L$ and that $\beta_c = \beta$ as subsets of $W = W(\mc G_c,\mc T_c)$.

Since $\mc G_c$ has connected centre, $\mc H_c = Z_{\mc G_c}(s_c)$ is connected and
\[
W_{\mc L_c} = W_{\mc L_c}^\circ , \quad \Omega_{\mc L_c} = \{1\}.
\]
As $\Rep_{s_c}(G_c^\epsilon) = \Rep_{W (\mc L_c, \beta_c)}(G_c^\epsilon)$ is both 
a rational and geometric series for $G_c^\epsilon$, $\beta_c$ is the only element of
$\mf B_{\mc L_c}$. By \cite{Lus2} we have
\begin{equation}\label{eq:E.15}
\Rep_{s_\beta} (G^\epsilon) = 
\Res_{G^\epsilon}^{G_c^\epsilon} \big( \Rep_{s_c}(G_c^\epsilon) \big)
= \Res_{G^\epsilon}^{G_c^\epsilon} \big( \Rep_{W \mc L_c}(G_c^\epsilon) \big) .
\end{equation}
Theorem \ref{thm:E.1} for $G_c^\epsilon$ simplifies to 
\begin{equation}\label{eq:E.16}
\Rep_{W \mc L_c}(G_c^\epsilon) \cong \Rep_1 \big( H_c^{\sigma_{\beta_c \epsilon}} \big) .
\end{equation}
The group $\mc H_c$ is generated by $\mc T_c$ and the root subgroups $U_\alpha$ with
$\alpha^\vee (s_c) = 1$. The latter condition depends only on $s_\beta$, so these 
are precisely the roots in $\Phi_{\mc L}$ and $\mc H_c$ is the full preimage of 
$\mc H$ in $\mc G_c$.

Recall that in Theorem \ref{thm:E.1} we use $\Omega_{\mc L,\beta}$-equivariant objects
in $\Rep_1 (H^{\sigma_{\beta \epsilon}})$ with respect to the action constructed in
\cite{LuYu}.

\begin{thm}\label{thm:E.7}
Suppose that $W(\mc L,\beta)$ corresponds to $\mr{Ad}(G^{\vee \epsilon}) s_\beta$
via \eqref{eq:E.23}. Theorem \ref{thm:E.1} restricts to an equivalence 
of categories 
\[
\Rep_{s_\beta} (G^\epsilon) = \Rep_{W(\mc L,\beta)} (G^\epsilon) \cong 
\Rep_1 (H^{\sigma_{\beta \epsilon}} )^{\Omega_{\mc L, \beta}} .
\]
Once a Whittaker datum for $\mc G$ has been chosen, the equivalence is canonical.
\end{thm}
\begin{proof}
We will use the functorial properties of the constructions of Lusztig and Yun. The functor
\[
\Res_{G^\epsilon}^{G_c^\epsilon} : \Rep_{W \mc L_c}(G_c^\epsilon) \to \Rep_{W \mc L}(G^\epsilon)
\]
corresponds to restriction of sheaves from $\mc G_c$ to $\mc G$ in \cite{LuYu}. Via 
\eqref{eq:E.16} and Theorem \ref{thm:E.1}, it becomes a functor
\begin{equation}\label{eq:E.18}
\Rep_1 (H_c^{\sigma_{\beta_c \epsilon}}) \longrightarrow
\bigoplus\nolimits_{\beta \in \mf B_{\mc L} / \mr{Ad}_\epsilon (\Omega_\mc L)} 
\Rep_1 \big( H^{\sigma_{\beta \epsilon}} \big)^{\Omega_{\mc L,\beta}} .
\end{equation}
From \cite[\S 12.10]{LuYu} one can see what \eqref{eq:E.18} does. Namely, first it restricts
representations of $H_c^\epsilon$ to $H^\epsilon$, and then it endows them with some
equivariant structure for $\Omega_{\mc L,\beta} = \mr{Stab}_{\Omega_{\mc L}}(\beta)$. 
In particular the image of \eqref{eq:E.18} lies entirely in
the summand $\beta$ of the codomain. Together with \eqref{eq:E.15}, it follows that 
Theorem \ref{thm:E.1} sends 
\begin{equation}\label{eq:E.24}
\Rep_{s_\beta}(G^\epsilon) = \Res_{G^\epsilon}^{G_c^\epsilon} \big( \Rep_{W \mc L_c}
(G_c^\epsilon) \big) \quad \text{to} \quad
\Rep_1 \big( H^{\sigma_{\beta \epsilon}} \big)^{\Omega_{\mc L,\beta}}.
\end{equation}
If we take the direct sum over $\beta \in \mf B_{\mc L} / \mr{Ad}_\epsilon 
(\Omega_{\mc L})$ of the functors \eqref{eq:E.24}, then by \eqref{eq:E.14} we recover 
the equivalence from Theorem \ref{thm:E.1}. Hence \eqref{eq:E.24} is an equivalence 
of categories for every $\beta \in \mf B_{\mc L}$.
\end{proof}

\section{Series and essentially unipotent representations}
\label{sec:2}

In this section we aim to improve the canonicity of the Jordan decomposition of
representations of a connected reductive group over a finite field, as in Theorems
\ref{thm:E.1} and \ref{thm:E.7}. Motivated by \cite[\S 2]{Kal1}, we will do so by
relating $\Rep_{W \mc L} (G^\epsilon)$ not to $\Rep_1 (H^{\sigma})$ but to
$\Rep_{\mc L}(H^\sigma)$, for some Frobenius actions $\sigma$ on $\mc H$. The objects 
in the latter category are called essentially unipotent representations, because 
they differ from unipotent representations only
by tensoring with a character of $H^\sigma$. 

We preserve the notations from the previous section. We have
\[
W_{\mc L} \cap W( \mc H,\mc T) = W_{\mc L}^\circ = W(\mc H,\mc T) ,
\]
so $\Phi (\mc H,\mc T) = \Phi_{\mc L}$ and $W(\mc H,\mc T) \mc L = \{ \mc L \}$.
Therefore $\mc L$ corresponds to a central element of $\mc H^\vee$. This 
determines a character $\chi_{\mc L}$ of $H^{\sigma_{\beta \epsilon}}$, which
extends the character of $T^{\sigma_{\beta \epsilon}}$ associated to $\mc L$
\cite[Proposition 2.5.20]{GeMa}. There is a canonical equivalence of categories
\begin{equation}\label{eq:3.1}
\begin{array}{ccc}
\Rep_1 (H^{\sigma_{\beta \epsilon}}) & \to & \Rep_{\mc L}(H^{\sigma_{\beta \epsilon}}) \\
\pi & \mapsto & \pi \otimes \chi_{\mc L}  
\end{array}.
\end{equation}
Further $\Omega_{\mc L}$ for $\mc H$ is trivial and, as $\sigma_{\beta \epsilon}
\mc L = \mc L$, the set $\mf B_{\mc L}$ for $\mc H^{\sigma_{\beta \epsilon}}$ reduces
to the single element $W_{\mc L}^\circ \backslash W_{\mc L}^\circ$. Thus \eqref{eq:3.1}
can be regarded as an easy special case of Theorem \ref{thm:E.7}. The group 
$\Omega_{\mc L,\beta}$ stabilizes $\mc L$, so fixes $\chi_{\mc L}$. This implies that
\eqref{eq:3.1} is equivariant for the weak actions of $\Omega_{\mc L,\beta}$. In fact
\eqref{eq:3.1} is also equivariant for the strict actions of 
$N_{G}(T)_{\mc L}^{\beta,\sigma_{\beta \epsilon}}$.
However, it is not obvious that \eqref{eq:3.1} extends to an equivalence of categories
\[
\Rep_1 (H^{\sigma_{\beta \epsilon}})^{\Omega_{\mc L,\beta}} \to 
\Rep_{\mc L}(H^{\sigma_{\beta \epsilon}})^{\Omega_{\mc L,\beta}} 
\]
with the left hand side as in Theorem \ref{thm:E.7} and on the right hand side equivariant
structures as in \eqref{eq:A.6}. Firstly, it is not clear whether the equivariant
structures in \cite{LuYu} and Theorem \ref{thm:E.7} come from the strict action of 
$N_G (T)_{\mc L}^{\beta,\sigma_{\beta \epsilon}}$. Secondly, even if this holds,
the weak $\Omega_{\mc L,\beta}$-actions on the two sides of \eqref{eq:3.1} still differ. 
Namely, the two morphisms $\phi^c_{n,\pi} = \pi (n)$ from \eqref{eq:A.4} and \eqref{eq:A.6}, 
which compensate for the non-strictness of the actions, differ.
Therefore the notions of equivariant objects on the two sides differ.
This could be resolved with an extension of $\chi_{\mc L}$ from $T^{\sigma_{\beta
\epsilon}}$ to $N_{G}(T)_{\mc L}^{\beta, \sigma_{\beta \epsilon}}$, but it
is unclear whether such an extension exists and is canonical.

Next we want to define equivariant structures on $\bigoplus_{\beta \in \mf B_{\mc L}} 
\Rep_{\mc L} (H^{\sigma_{\beta \epsilon}})$, for the weak action of $\Omega_{\mc L}$. 
Recall the representatives $n(\omega,\beta)$ of $\omega \in \Omega_{\mc L}$ from 
\eqref{eq:E.39}. Consider the category 
$\tilde{\mf B}_{\mc L}$ with set of objects $\mf B_{\mc L}$, morphisms
\begin{equation}\label{eq:E.35}
\tilde{\mf B}_{\mc L}(\beta,\beta') = \{ n (\omega,\beta) : 
\omega (\beta) = \beta' \} T^{\sigma_{\beta \epsilon}}  
\end{equation}
and multiplication as in $N_{\mc G}(\mc T)$. Let $\tilde T$ be the subgroupoid of
$\tilde{\mf B}_{\mc L}$ with $\tilde T (\beta,\beta) = T^{\sigma_{\beta \epsilon}}$
and $\tilde T (\beta,\beta') = \emptyset$ if $\beta \neq \beta'$. Then $\mf B_{\mc L} /
\tilde T$ is the groupoid $\tilde \Omega_{\mc L}$ with 
\[
\tilde \Omega_{\mc L}(\beta,\beta') = \{ \omega \in \Omega_{\mc L} : 
\mr{Ad}_\epsilon (\omega) \beta = \beta' \} .
\]
An $\Omega_{\mc L}$-equivariant structure on $\pi = \bigoplus_{\beta \in \mf B_{\mc L}} 
\pi_\beta$ is the same as an $\tilde \Omega_{\mc L}$-equivariant structure. It consists 
of a morphism from $\mr{Ad}(n(\omega,\beta)) \pi_\beta$ to $\pi_\beta$ for each 
$\omega \in \Omega_{\mc L}$, multiplicative in $\omega$ up to canonical isomorphisms. 
To make that condition explicit we need, for each representative $\dot \omega \in 
\tilde{\mf B}_{\mc L} (\beta,\beta')$, a morphism
\begin{equation}\label{eq:E.28}
\alpha_{\beta,\pi} (\dot \omega) : \mr{Ad}(\dot \omega) \pi_\beta \to \pi_\beta 
\end{equation}
such that:
\begin{itemize}
\item $\alpha_{\beta,\pi} (t) = \pi_\beta (t)$ for all $t \in T^{\sigma_{\beta \epsilon}}$,
\item $\alpha_{\beta,\pi} \big( \dot \omega' \dot \omega \big) = 
\alpha_{\mr{Ad}_\epsilon (\omega) \beta, \pi} ( \dot \omega')  
\alpha_{\beta,\pi} (\dot \omega)$ for all eligible $\dot \omega'$.
\end{itemize}
We note that $\tilde{\mf B}_{\mc L}(\beta,\beta) =
N_{G}(T)_{\mc L}^{\beta, \sigma_{\beta \epsilon}}$, so
\[
H^\Omega_{\mc L}(\beta,\beta) = N_{G}(T)_{\mc L}^{\beta, \sigma_{\beta \epsilon}} 
\overset{T^{\sigma_{\beta \epsilon}}}{\times} H^{\sigma_{\beta \epsilon}} =
\tilde{\mf B}_{\mc L}(\beta,\beta) 
\overset{\tilde T (\beta,\beta)}{\times} H^{\sigma_{\beta \epsilon}} . 
\]
This gives a disconnected $\F_q$-group with neutral component $\mc H^{\sigma_{\beta \epsilon}}$
and component group $\Omega_{\mc L,\beta}$.
We define $\Rep_{\mc L}(H^\Omega_{\mc L}(\beta,\beta))$ as the category of
$H^\Omega_{\mc L}(\beta,\beta)$-representations whose restriction to $H^{\sigma_{\beta \epsilon}}$
lies in $\Rep_{\mc L} (H^{\sigma_{\beta \epsilon}})$. Then Lemma \ref{lem:A.2} provides
a canonical equivalence of categories
\begin{equation}\label{eq:3.50}
\Rep_{\mc L} (H^{\sigma_{\beta \epsilon}})^{\Omega_{\mc L,\beta}} \cong
\Rep_{\mc L} (H^\Omega_{\mc L} (\beta,\beta)) .
\end{equation}
Let $\underline{CS}_{W \mc L} (\mc G;\epsilon)$ be the category of direct sums of
$\epsilon$-twisted character sheaves on $\mc G$. These are certain $\mc G$-equivariant
perverse sheaves on $\mc G \epsilon$. The group $\mc G$ acts strictly on
$\underline{CS}_{W \mc L} (\mc G;\epsilon)$, via the conjugation action on $\mc G \epsilon$.
The equivariant structure on an object $\mc F \in \underline{CS}_{W \mc L} (\mc G;\epsilon)$
provides isomorphisms
\begin{equation}\label{eq:3.40}
\alpha_{g,\mc F} : \mr{Ad}(g)_* \mc F \to \mc F \qquad g \in \mc G .
\end{equation}

\begin{lem}\label{lem:3.1}
\enuma{
\item There are canonical equivalences of categories
\begin{align*}
& \underline{CS}_{W \mc L} (\mc G ; \epsilon) \,\, \cong \Rep_{W \mc L} (G^\epsilon),\\
& \underline{CS}_{\mc L} (\mc H ;\sigma_{\beta \epsilon}) \cong 
\Rep_{\mc L} (H^{\sigma_{\beta \epsilon}}) .
\end{align*}
\item The second equivalence in part (a) is $N_{G}(T)_{\mc L}^{\beta, 
\sigma_{\beta \epsilon}}$-equivariant and induces a canonical 
equivalence of categories
\[
\underline{CS}_{\mc L} (\mc H ;\sigma_{\beta \epsilon})^{\Omega_{\mc L,\beta}} \cong 
\Rep_{\mc L} (H^{\sigma_{\beta \epsilon}})^{\Omega_{\mc L,\beta}} .
\] 
\item By summing over $\beta \in \mf B_{\mc L} / \mr{Ad}_\epsilon (\Omega_{\mc L})$, 
part (b) induces a canonical equivalence of categories
\[
\big( \bigoplus\nolimits_{\beta \in \mf B_{\mc L}} \underline{CS}_{\mc L} (\mc H ;
\sigma_{\beta \epsilon})\big)^{\Omega_{\mc L}} \cong 
\big( \bigoplus\nolimits_{\beta \in \mf B_{\mc L}} 
\Rep_{\mc L} (H^{\sigma_{\beta \epsilon}}) \big)^{\Omega_{\mc L}} .
\] 
}
\end{lem}
\begin{proof}
(a) The first equivalence is taking the fibre of an equivariant perverse sheaf
on $\mc G \epsilon$ at $\epsilon$, see \cite[(12.1)]{LuYu}. As $W(\mc H,\mc T)$ fixes
$\mc L$, the same functor gives the second equivalence.\\
(b) The group $N_{G}(T)_{\mc L}^{\beta, \sigma_{\beta \epsilon}}$ acts naturally
on both sides, namely by the functors $\mr{Ad}(g)_*$. The explicit form of the second
equivalence in part (a) shows that it is $N_{G}(T)_{\mc L}^{\beta, 
\sigma_{\beta \epsilon}}$-equivariant. For any $h \in N_{G}(T)_{\mc L}^{\beta,
\sigma_{\beta \epsilon}} \cap \mc H$ there are canonical isomorphisms
\begin{equation}\label{eq:3.18}
\phi_h^c : \mr{Ad}(h)_* \to \mr{id} \text{ on } \Rep_{\mc L}(H^{\sigma_{\beta \epsilon}})
\text{ and on } \underline{CS}_{\mc L} (\mc H ;\sigma_{\beta \epsilon}) .
\end{equation}
Namely, for a representation $\pi$ we use $\pi (h)$ and for a character sheaf $\mc F_\pi$
we use $\alpha_{h,\mc F_\pi}$ from \eqref{eq:3.40}. These actions correspond under
taking the fibres of character sheaves on $\mc H \sigma_{\beta \epsilon}$ at 
$\sigma_{\beta \epsilon}$. Now we see from 
\eqref{eq:A.7} and \eqref{eq:3.19} that the second equivalence in part (a) remains valid 
when we enrich both sides with $\Omega_{\mc L,\beta}$-equivariant structures.\\
(c) By Lemma \ref{lem:A.1}.a and part (b) we have canonical equivalences
\begin{multline}
\big( \bigoplus\nolimits_{\beta \in \mf B_{\mc L}} \underline{CS}_{\mc L} (\mc H ;
\sigma_{\beta \epsilon})\big)^{\Omega_{\mc L}} \cong
\bigoplus\nolimits_{\beta \in \mf B_{\mc L} / \mr{Ad}_\epsilon (\Omega_{\mc L})} 
\underline{CS}_{\mc L} (\mc H ;\sigma_{\beta \epsilon})^{\Omega_{\mc L,\beta}} \\
\cong \bigoplus\nolimits_{\beta \in \mf B_{\mc L} / \mr{Ad}_\epsilon (\Omega_{\mc L})} 
\Rep_{\mc L} (H^{\sigma_{\beta \epsilon}})^{\Omega_{\mc L,\beta}} \cong
\big( \bigoplus\nolimits_{\beta \in \mf B_{\mc L}} 
\Rep_{\mc L} (H^{\sigma_{\beta \epsilon}}) \big)^{\Omega_{\mc L}} . \qedhere
\end{multline}
\end{proof}

\subsection{Structure of the main argument} \

In view of Theorem \ref{thm:E.1} and Lemma \ref{lem:3.1}, the goal of this section
becomes a cano\-nical equivalence of categories
\begin{equation}\label{eq:3.2}
\underline{CS}_{W \mc L} (\mc G ;\epsilon) \cong \Big( \bigoplus_{\beta \in 
\mf B_{\mc L}} \underline{CS}_{\mc L} (\mc H ; \sigma_{\beta \epsilon} ) 
\Big)^{\Omega_{\mc L}} \cong \bigoplus_{\beta \in 
\mf B_{\mc L} / \mr{Ad}_\epsilon (\Omega_{\mc L})} \underline{CS}_{\mc L} (\mc H ; 
\sigma_{\beta \epsilon} )^{\Omega_{\mc L,\beta}} 
\end{equation}
Our strategy involves modifying the arguments from \cite{LuYu}. 
Let $\mc U$ be the unipotent radical of $\mc B$, and let $\mc B_{\mc H} = \mc T
\mc U_{\mc H}$ be the Borel subgroup of $\mc H$ determined by the set of
positive roots $\Phi_{\mc L}^+ = \Phi^+ \cap \Phi_{\mc L}$. 
We need several further notations from \cite{LuYu}, in particular
\begin{align*}
& _{\mc L'} \mc D_{\mc L} = \mc D^b_{(\mc T \times \mc T, \mc L' \boxtimes \mc L^{-1}),m} 
(\mc U \backslash \mc G / \mc U) , \\
& \mc D_{\mc H} = \mc D^b_m (\mc B_{\mc H} \backslash \mc H / \mc B_{\mc H}) ,\\
& _{\mc L'} \mc D^{\mc H}_{\mc L} = \mc D^b_{(\mc T \times \mc T, 
\mc L' \boxtimes \mc L^{-1}),m} (\mc U_{\mc H} \backslash \mc H / \mc U_{\mc H}) .
\end{align*}
The non-mixed version of these categories are indicated by underlining $\mc D$. 
Every coset in $W_{\mc L} / W_{\mc L}^\circ$ determines a block of ${}_{\mc L}
\mc D_{\mc L}$ \cite[\S 4]{LuYu}, and the trivial coset $W_{\mc L}^\circ$ gives the
neutral block  ${}_{\mc L} \mc D^\circ_{\mc L}$. There are semisimple abelian 
subcategories 
\[
{}_{\mc L'} \underline{S}_{\mc L} \subset {}_{\mc L'} \underline{\mc D}_{\mc L} ,\quad
{}_{\mc L} \underline{S}^\circ_{\mc L} \subset {}_{\mc L} \underline{\mc D}^\circ_{\mc L} 
,\quad \underline{S}_{\mc H} \subset \underline{\mc D}_{\mc H} ,\quad {}_{\mc L'}
\underline{S}^{\mc H}_{\mc L} \subset {}_{\mc L'} \underline{\mc D}^{\mc H}_{\mc L} ,
\]
which by definition consist of direct sums of the intersection cohomology complexes
$\underline{\IC}(w)_{\mc L}$ for the appropriate $w \in W$. In such categories there
is the notion of the twisted categorical centre \cite[\S 11.7]{LuYu}, denoted
$Z ({}_{\mc L} \underline{S}_{\epsilon \mc L} ;\epsilon)$ etc.

The proof of \cite[Theorem 12.7]{LuYu}, in \cite[\S 11.14 and \S 12.10]{LuYu}, can
be summarized by the following sequence of equivalences of categories: 
\begin{equation}\label{eq:3.3}
\begin{aligned}
& \underline{CS}_{W \mc L}^{[\bf c]} (\mc G ;\epsilon) \isom 
Z ({}_{\mc L} \underline{S}^{[\bf c]}_{\epsilon \mc L} ;\epsilon) \isom
\bigoplus\nolimits_{\beta \in \mf B_{\mc L,\bf c} / \mr{Ad}_\epsilon (\Omega_{\bf c})} 
Z ({}_{\mc L} \underline{S}^{\circ,\bf c}_{\mc L} ; \mr{Ad}(\dot w^\beta) \circ \epsilon 
)^{\Omega_{\bf c, \beta}} \\
& \! \mosi \bigoplus\nolimits_{\beta \in \mf B_{\mc L,\bf c} / \mr{Ad}_\epsilon 
(\Omega_{\bf c})} Z ( \underline{S}^{\bf c}_{\mc H} ; \sigma_{\beta \epsilon} 
)^{\Omega_{\bf c, \beta}}
\mosi \bigoplus\nolimits_{\beta \in \mf B_{\mc L, \bf c} / \mr{Ad}_\epsilon 
(\Omega_{\bf c})} \underline{CS}^{\bf c}_{1} (\mc H; \sigma_{\beta \epsilon} 
)^{\Omega_{\bf c, \beta}} . 
\end{aligned} 
\end{equation}
Here the sub- and superscripts {\bf c} indicate a cell in $W$, and in the third term
the equivariant structure is defined by means of Whittaker data and rigidified 
minimal IC-sheaves. The first arrow in 
\eqref{eq:3.4} is \cite[Theorem 12.4]{LuYu} followed by \cite[(12.4)]{LuYu}. The second 
arrow is \cite[Lemma 11.13 and
\S 12.10]{LuYu}. The third arrow is \cite[Theorem 9.2]{LuYu} and the final arrow 
in \eqref{eq:3.4} is \cite[Theorem 12.4]{LuYu} (extended by taking equivariant objects).
We point out that all the equivalences in \eqref{eq:3.3} and \eqref{eq:3.4} are 
canonical in the following sense: whenever another choice is made in the underlying
constructions, it is related to the initial choice by a unique isomorphism.

Like in the proof of Theorem \ref{thm:E.1}, in \eqref{eq:3.3} we may also take the 
direct sum over over all $\beta \in \mf B_{\mc L,\bf c}$ and next over all cells $\bf c$. 
As compensation we take $\Omega_{\mc L}$-equivariant structures. Then Lemma 
\ref{lem:A.1} says that \eqref{eq:3.3} becomes 
\begin{equation}\label{eq:3.4}
\begin{aligned}
& \underline{CS}_{W \mc L} (\mc G; \epsilon) \isom 
Z ({}_{\mc L} \underline{S}_{\epsilon \mc L} ;\epsilon) \isom
\big( \bigoplus\nolimits_{\beta \in \mf B_{\mc L}} 
Z ({}_{\mc L} \underline{S}^{\circ}_{\mc L} ; \mr{Ad}(\dot w^\beta) \circ \epsilon 
\big)^{\Omega_{\mc L}} \\
& \mosi \big( \bigoplus\nolimits_{\beta \in \mf B_{\mc L}} 
Z ( \underline{S}_{\mc H} ; \sigma_{\beta \epsilon} ) \big)^{\Omega_{\mc L}}
\mosi \big( \bigoplus\nolimits_{\beta \in \mf B_{\mc L}} 
\underline{CS}_{1} (\mc H; \sigma_{\beta \epsilon} ) \big)^{\Omega_{\mc L}} .
 \end{aligned}
\end{equation}
We will analyse all the equivalences in \eqref{eq:3.4} in detail, to make them 
more canonical and equivariant for $N_{\mc G}(\mc T)_{\mc L}^+$.

As general preparation, we take a closer look at the intersection cohomology complexes
from \cite{LuYu}. For every $\dot w \in N_{\mc G}(\mc T)$ representing $w \in W_{\mc L}$, 
we can form
\[
\overline{\Q_\ell} \langle \ell (w) \rangle \in \mc D^b_{\Gamma (w),m} (\{\dot w\}) , 
\]
where $\Gamma (w) = \{ (w(t),t) : t \in \mc T \}$. By \cite[Lemma 2.10]{LuYu} the
functor
\begin{equation}\label{eq:3.5}
i_{\dot w}^* : {}_{\mc L} \mc D (w)_{\mc L} \to \mc D^b_{\Gamma (w),m} (\{\dot w\})
\end{equation}
is an equivalence of categories, where ${}_{\mc L} \mc D (w)_{\mc L}$ is the version
of ${}_{\mc L} \mc D_{\mc L}$ on $\mc B w \mc B$ instead of on $\mc G$. This gives
\[
C(\dot w) = (i_{\dot w}^*)^{-1} \overline{\Q_\ell} \langle \ell (w) \rangle \in 
{}_{\mc L} \mc D (w)_{\mc L} .
\]
With the embedding $j_w : \mc B w \mc B \to \mc G$ we form
\begin{equation}\label{eq:3.6}
\IC (\dot w )_{\mc L} = j_{w,!*} C(\dot w)_{\mc L} \in {}_{\mc L} \mc D_{\mc L} .
\end{equation}
This is a particular version of the intersection cohomology complex determined by 
$w$ and $\mc L$. For a different representative $\dot w t$ with $t \in \mc T$ there
is a canonical isomorphism \cite[Remark 2.12]{LuYu} 
\[
\IC (\dot w t)_{\mc L} \cong \IC (\dot w)_{\mc L} \otimes \mc L_t .
\]
All this works just as well on $\mc G \epsilon$ (a coset of $\mc G$ in a semidirect
product $\mc G \rtimes \Z$.) Thus we get $\IC (n)$ for any $n \in N_{\mc G}(\mc T)_{\mc L}
\cup N_{\mc G \epsilon}(\mc T)_{\mc L}$.

Let $\omega \in \Omega_{\mc L}$ be a minimal length element in $W_{\mc L}^\circ \omega
\subset W_{\mc L}$, with a representative $\dot \omega \in N_{\mc G}(\mc T)_{\mc L}^+$. 
Then \cite[Proposition 5.2 and \S 5.5]{LuYu} provide canonical isomorphisms
\begin{equation}\label{eq:3.7}
\begin{aligned}
& \IC (n)_{\mc L} * \IC (\dot \omega )_{\mc L} \cong \IC (n \dot \omega )_{\mc L} , \\
& \IC (\dot \omega)_{\mc L} * \IC (n)_{\mc L} \cong \IC (\dot \omega n )_{\mc L} . 
\end{aligned}
\end{equation}
The monoidal unit in ${}_{\mc L} \mc D_{\mc L}$ is $\IC (e)_{\mc L}$,
and \eqref{eq:3.7} shows that
\[
\IC (\dot \omega )_{\mc L}^{-1} = \IC (\dot \omega^{-1} )_{\mc L} .
\]

\begin{lem}\label{lem:3.2}
\enuma{
\item For any lift $\dot \omega \in N_{\mc G}(\mc T)_{\mc L}^+$ of $\omega \in \Omega_{\mc L}$ 
and any $n \in N_{\mc G}(\mc T)_{\mc L} \cup  N_{\mc G \epsilon}(\mc T)_{\mc L}$, 
the two orders of convolution provide two isomorphisms
\[
\IC (\dot \omega )_{\mc L} * \IC (n)_{\mc L} * \IC (\dot \omega^{-1})_{\mc L}
\cong \IC (\dot \omega n \dot \omega^{-1})_{\mc L} .
\]
which differ only by the associativity structure. 
\item For $\dot \omega = t \in \mc T$ these coincide to a single, canonical isomorphism
\[
\IC (t )_{\mc L} * \IC (n)_{\mc L} * \IC (t^{-1})_{\mc L}
\cong \IC (t n t^{-1})_{\mc L} = \mr{Ad}(t)_* \IC (n)_{\mc L} .
\]
}
\end{lem}
\begin{proof}
The first isomorphism follows from \eqref{eq:3.7}. 

When the lengths of two elements $w_1,w_2 \in W_{\mc L}$ (or one from 
$N_{\mc G \epsilon}(\mc T)_{\mc L} / \mc T$) add up to $\ell (w_1 w_2)$, 
\cite[Lemma 3.4]{LuYu} shows that \eqref{eq:3.7} comes directly from the isomorphism
\[
\mc B w_1 \mc B \, \overset{\mc B}{\times}\, \mc B w_2 \mc B \to \mc B w_1 w_2 \mc B .
\]
More precisely, for representatives $\dot w_1, \dot w_2$ the isomorphism
\begin{equation}\label{eq:3.8}
\IC (\dot w_1)_{\mc L} * \IC (\dot w_2)_{\mc L} \cong \IC (\dot w_1 \dot w_2)_{\mc L}
\end{equation}
boils down to an isomorphism for the stalks at these representatives, which is just
$\overline{\Q_\ell} \otimes \overline{\Q_\ell} \to \overline{\Q_\ell}$ given by
multiplication. This applies to $\IC (t )_{\mc L} * \IC (n)_{\mc L} * 
\IC (t^{-1})_{\mc L}$. In particular the order of the convolution operations does
not matter, and the result is canonically isomorphic to $\IC (t n t^{-1})_{\mc L}$.
Since Ad$(t)$ stabilizes $\mc T, \mc L, \mc B$ and $\mc U$, it defines an automorphism of
the category ${}_{\mc L} \mc D_{\mc L}$. By definition, that sends $\IC (n)_{\mc L}$
to $\IC (t n t^{-1})_{\mc L}$.
\end{proof}

\subsection{Twisted categorical centres} \

We turn to the first equivalence in \eqref{eq:3.4}. While $\mc G$ acts naturally on 
$\underline{CS}_{W \mc L} (\mc G; \epsilon)$, we immediately face  the problem that 
$N_{\mc G}(\mc T)_{\mc L}^+$ (or even $\mc T$) does not act on 
$Z ({}_{\mc L} \underline{S}_{\epsilon \mc L} ;\epsilon)$. Namely, $\mr{Ad}(t)_*$ sends 
\[
Z ({}_{\mc L} \underline{S}_{\epsilon \mc L} ;\epsilon) \quad \text{to} \quad
Z \big( {}_{\mc L} \underline{S}_{\epsilon \mc L} ;\mr{Ad}(t) \circ \epsilon 
\circ \mr{Ad}(t^{-1}) \big) ,
\]
so only $N_G (T)_{\mc L}^{+ \epsilon} = (N_{\mc G}(\mc T)_{\mc L}^+)(\overline \F_q)^\epsilon$ 
acts naturally on $Z ({}_{\mc L} \underline{S}_{\epsilon \mc L} ;\epsilon)$. 
To improve that, we consider the isomorphism of stacks
\[
\begin{array}{cccc}
\rho_\epsilon : & (\mc U \backslash \mc G / \mc U) / \mc T & \to & 
 (\mc U \backslash \mc G \epsilon / \mc U) / \mc T  ,\\
 & \mc U g \mc U & \mapsto & \mc U g \epsilon \mc U .
\end{array}
\]
Here $t \in \mc T$ acts by $\mr{Ad}(t)$ on $\mc U \backslash \mc G \epsilon / \mc U$ and by
\[
\mr{Ad}_\epsilon (t) : g \mapsto t g \epsilon (t^{-1})
\]
on $\mc U \backslash \mc G / \mc U$. We can regard ${}_{\mc L} \underline{S}_{\epsilon \mc L}$
as a category of sheaves on $(\mc U \backslash \mc G / \mc U) / \mc T$. We let 
$\underline{\IC} (n)_{\mc L}$ be the non-mixed version of $\IC (n)_{\mc L}$, for
$n \in N_{\mc G}(\mc T)_{\mc L} \cup N_{\mc G \epsilon}(\mc T)_{\mc L}$. We define 
\[
{}_{\mc L} \underline{\epsilon S}_{\mc L} \subset \mc D^b_{(\mc T \times \mc T, 
\mc L \boxtimes \mc L^{-1})}(\mc U \backslash \mc G \epsilon / \mc U)
\]
to be the full subcategory of whose objects are direct sums of $\underline{\IC}(n)_{\mc L}$
with $n \in N_{\mc G \epsilon}(\mc T)_{\mc L}$. Like in \cite[Lemma 3.4]{LuYu} one checks that 
\begin{equation}\label{eq:3.23}
(\rho_\epsilon )_* \IC (g)_{\epsilon \mc L} = \IC (g \epsilon )_{\mc L} .
\end{equation}
Therefore ${}_{\mc L} \underline{\epsilon S}_{\mc L}$ is the essential image of 
${}_{\mc L} \underline{S}_{\epsilon \mc L}$ under $(\rho_\epsilon )_*$.

The truncated convolution product $\circ$ from 
\cite{Lus3,Lus4} and \cite[\S 11.4]{LuYu} is defined on subcategories of
$\bigoplus_{\mc L_1,\mc L_2 \in W \mc L} {}_{\mc L_1} \mc D_{\mc L_2}$ associated to
a cell. If we take the direct sum over all possible cells, we obtain a monoidal
structure $\circ$ on $\bigoplus_{\mc L_1,\mc L_2 \in W \mc L} {}_{\mc L_1} \mc D_{\mc L_2}$.
We claim that 
\begin{equation}\label{eq:3.55}
\mc F_1 * \underline{\IC}(n)_{\mc L} = \mc F_1 \circ \underline{\IC}(n)_{\mc L} 
\quad \text{and} \quad \underline{\IC}(n)_{\mc L} * \mc F_2 = 
\underline{\IC}(n)_{\mc L} \circ \mc F_2
\end{equation}
for all $\mc F_1 \in {}_{\mc L_1} \mc D_{\mc L}, \mc F_2 \in {}_{\mc L} \mc D_{\mc L_2}$ and 
$n \in N_{\mc G} (\mc T)_{\mc L}^+$. 
Indeed,  when $\mc F$ is supported on a single cell $\mb c$, all the involved sheaves are
IC-complexes obtained from local systems supported on a single cell $n \mb c$ or
$\mb c n$. Sine multiplication by $n \mc T \in \Omega_{\mc L}$ preserves the order relations
on the cells in $W_{\mc L}$, the truncation does not change anything. A general $\mc F$ 
can be written as a direct sum of objects supported on a single 
cell, and both convolution products are additive.

Truncated convolution provides functors 
\[
{}_{\mc L} \underline{S}_{\epsilon \mc L} \times {}_{\mc L} \underline{S}_{\mc L} \to
{}_{\mc L} \underline{S}_{\epsilon \mc L} \quad \text{and} \quad 
{}_{\mc L} \underline{S}_{\mc L} \times {}_{\mc L} \underline{S}_{\epsilon \mc L} \to 
{}_{\mc L} \underline{S}_{\epsilon \mc L} .
\]
The $\epsilon$-twisted central structure on an object $\mc Z \in Z ({}_{\mc L} 
\underline{S}_{\epsilon \mc L} ;\epsilon)$ means that we are given natural isomorphisms 
\[
z_{\mc F} : \mc F \circ \mc Z \to \mc Z \circ \epsilon_* \mc F 
\qquad \mc F \in {}_{\mc L} \underline{S}_{\mc L} .
\]
Moreover, for all $\mc F_1, \mc F_2 \in {}_{\mc L} \underline{S}_{\mc L}$ we have
\begin{equation}\label{eq:3.54}
z_{\mc F_1 \circ \mc F_2} \cong (z_{\mc F_1} \circ \mr{id}_{\mc F_2})
(\mr{id}_{\mc F_1} \circ z_{\mc F_2}) , 
\end{equation}
where the natural isomorphism is given by the associativity structure of $\circ$. The 
morphisms in $Z ({}_{\mc L} \underline{S}_{\epsilon \mc L} ;\epsilon)$ are the morphisms 
in ${}_{\mc L} \underline{S}_{\epsilon \mc L}$ that intertwine the $\epsilon$-twisted 
central structures.

Truncated convolution also yields functors 
\[
{}_{\mc L} \underline{\epsilon S}_{\mc L} \times {}_{\mc L} \underline{S}_{\mc L} \to
{}_{\mc L} \underline{\epsilon S}_{\mc L} \quad \text{and} \quad 
{}_{\mc L} \underline{S}_{\mc L} \times {}_{\mc L} \underline{\epsilon S}_{\mc L} \to 
{}_{\mc L} \underline{\epsilon S}_{\mc L} ,
\]
which satisfy \eqref{eq:3.55} for $\mc F_1, \mc F_2 \in {}_{\mc L} \underline{\epsilon S}_{\mc L}$.
Let $Z({}_{\mc L} \underline{\epsilon S}_{\mc L})$ be the twisted categorical centre of
${}_{\mc L} \underline{\epsilon S}_{\mc L}$, by which we mean that it consists of
objects $\mc Z$ of ${}_{\mc L} \underline{\epsilon S}_{\mc L}$ equipped with 
natural isomorphisms
\[
z_{\mc F} : \mc F \circ \mc Z \isom \mc Z \circ \mc F \quad \text{for } 
\mc F \in {}_{\mc L} \underline{S}_{\mc L} .
\]
These are required to satisfy the analogue of \eqref{eq:3.54}. 

The group $N_{\mc G} (\mc T)_{\mc L}^+$ acts naturally on ${}_{\mc L} \underline{S}_{\mc L}$
and on ${}_{\mc L} \underline{\epsilon S}_{\mc L}$, by 
\begin{equation}\label{eq:3.11}
{}^n \underline{\IC}(g)_{\mc L} = \underline{\IC}(n g n^{-1})_{\mc L}
\end{equation}
for $n \in N_{\mc G} (\mc T)_{\mc L}^+$ and 
$g \in N_{\mc G}(\mc T)_{\mc L} \cup N_{\mc G \epsilon}(\mc T)_{\mc L}$. 
This action is multiplicative for the convolution product. Furthermore, since conjugation
by $\Omega_{\mc L} \cong N_{\mc G}(\mc T)_{\mc L}^+ / \mc T$ preserves the order relations
between the cells in $W_{\mc L}$, this action is also multiplicative for the truncated 
convolution product. 

By Lemma \ref{lem:3.2}.a and \eqref{eq:3.55} there are canonical isomorphisms
\begin{equation}\label{eq:3.56}
\underline{\IC}(n g n^{-1})_{\mc L} \cong
(\underline{\IC}(n)_{\mc L} * \underline{\IC}(g)_{\mc L}) * \underline{\IC}(n^{-1})_{\mc L} 
\cong (\underline{\IC}(n)_{\mc L} \circ \underline{\IC}(g)_{\mc L}) 
\circ \underline{\IC}(n^{-1})_{\mc L} 
\end{equation}
for $n \in N_{\mc G} (\mc T)_{\mc L}^+$ and 
$g \in N_{\mc G}(\mc T)_{\mc L} \cup N_{\mc G \epsilon}(\mc T)_{\mc L}$. 
The associativity structure of the convolution product provides analogous isomorphisms
when take the other order of the two convolution products.

\begin{lem}\label{lem:3.6}
The equivalence $(\rho_\epsilon)_* : {}_{\mc L} \underline{S}_{\epsilon \mc L} \to
{}_{\mc L} \underline{\epsilon S}_{\mc L}$ induces an equivalence of twisted categorical
centres $Z ({}_{\mc L} \underline{S}_{\epsilon \mc L} ;\epsilon) \isom Z({}_{\mc L} 
\underline{\epsilon S}_{\mc L})$.
\end{lem}
\begin{proof}
For $\mc Z \in Z ({}_{\mc L} \underline{S}_{\epsilon \mc L} ;\epsilon)$ and 
$\mc F \in {}_{\mc L} \underline{S}_{\mc L}$ we compute, using the $\epsilon$-twisted 
central structure:
\begin{equation}\label{eq:3.20}
\mc F \circ (\rho_\epsilon )_* \mc Z = (\rho_\epsilon )_* (\mc F \circ \mc Z) \cong
(\rho_\epsilon )_* (\mc Z \circ \epsilon_* \mc F ) =
\mc Z \circ (\rho_\epsilon )_* (\epsilon_* \mc F) .
\end{equation}
Writing $\lambda_\epsilon : \mc G \to \mc G \epsilon , \lambda_\epsilon (g) = \epsilon g$, 
the right hand side of \eqref{eq:3.20} is naturally isomorphic to
\begin{equation}\label{eq:3.21}
\mc Z \circ (\lambda_\epsilon )_* (\lambda_\epsilon^{-1} \circ 
\rho_\epsilon )_* (\epsilon_* \mc F) 
\cong \mc Z \circ (\lambda_\epsilon )_* \mc F \cong (\rho_\epsilon )_* \mc Z \circ \mc F.
\end{equation}
From \eqref{eq:3.20} and \eqref{eq:3.21} we see that $(\rho_\epsilon )_* \mc Z$ belongs
to $Z({}_{\mc L} \underline{\epsilon S}_{\mc L})$. Similarly one checks that
$(\rho_\epsilon )_*^{-1}$ sends $Z({}_{\mc L} \underline{\epsilon S}_{\mc L})$ to
$Z ({}_{\mc L} \underline{S}_{\epsilon \mc L} ;\epsilon)$.
\end{proof}

Recall that the irreducible objects of ${}_{\mc L} \underline{S}_{\mc L}$ are the 
$\ICL{n}$ with $n \in N_{\mc G}(\mc T)_{\mc L}$, and that $N_{\mc H}(\mc T)$ is the preimage
of $W_{\mc L}^\circ$ in $N_{\mc G}(\mc T)$. We will need a variation on $Z({}_{\mc L} 
\underline{\epsilon S}_{\mc L})$, where the twisted central structure is decomposed along
\[
N_{\mc G}(\mc T)_{\mc L} = N_{\mc G}(\mc T)_{\mc L}^+ \overset{\mc T}{\times} N_{\mc H}(\mc T) .
\]
We define the objects of $Z^+ ({}_{\mc L} \underline{\epsilon S}_{\mc L})$ to be objects
$\mc Z$ of ${}_{\mc L} \underline{\epsilon S}_{\mc L}$ together with natural isomophisms
\[
\begin{array}{llll@{\qquad}l}
z_{\mc F} : & \mc F \circ \mc Z & \isom & \mc Z \circ \mc F & \mc F \in 
{}_{\mc L} \underline{S}^\circ_{\mc L} ,\\
z_n : & {}^n \mc Z & \isom & \mc Z & n \in N_{\mc G}(\mc T)_{\mc L}^+ .
\end{array}
\]
These are required to fulfill the following conditions:
\begin{enumerate}[(i)]
\item $z_{\mc F_1 \circ \mc F_2} \cong (z_{\mc F_1} \circ \mr{id}_{\mc F_2})
(\mr{id}_{\mc F_1} \circ z_{\mc F_2})$ for 
$\mc F_1, \mc F_2 \in {}_{\mc L} \underline{S}^\circ_{\mc L}$;
\item $z_{n_1 n_2} = z_{n_1} \circ {}^{n_1} (z_{n_2})$ for 
$n_1, n_2 \in N_{\mc G}(\mc T)_{\mc L}^+$;
\item for $t \in \mc T$, $z_t$ equals the composition
\[
{}^t \mc Z \xrightarrow{\mr{Lemma} \: \ref{lem:3.2}} \ICL{t} \circ \mc Z \circ \ICL{t^{-1}}
\xrightarrow{z_{\ICL{t}} \circ \: \mr{id}} \mc Z \circ \ICL{t} \circ \ICL{t^{-1}} 
\xrightarrow{\eqref{eq:3.7}} \mc Z ;
\] 
\item for $\mc F \in {}_{\mc L} \underline{S}^\circ_{\mc L}$ and 
$n \in N_{\mc G}(\mc T)_{\mc L}^+$ there is a commutative diagram
\[
\begin{array}{ccccc}
{}^n ( \mc F \circ \mc Z ) & \cong & {}^n \mc F \circ {}^n \mc Z & 
\xrightarrow{\mr{id} \, \circ \, z_n} & {}^n \mc F \circ \mc Z \\
\downarrow {}^n (z_{\mc F}) & & & & \downarrow z_{({}^n \mc F)} \\
{}^n ( \mc Z \circ \mc F ) & \cong & {}^n \mc Z \circ {}^n \mc F & 
\xrightarrow{z_n \circ \: \mr{id}} & \mc Z \circ {}^n \mc F 
\end{array}.
\]
\end{enumerate}
The morphisms in $Z^+ ({}_{\mc L} \underline{\epsilon S}_{\mc L})$ are the morphisms in
${}_{\mc L} \underline{\epsilon S}_{\mc L}$ that intertwine the $z_{\mc F}$ and the $z_n$.

\begin{lem}\label{lem:3.10}
There is an equivalence of categories $Z ({}_{\mc L} \underline{\epsilon S}_{\mc L})
\to Z^+ ({}_{\mc L} \underline{\epsilon S}_{\mc L})$.
\end{lem}
\begin{proof}
The equivalence is the identity on objects $\mc Z$ of 
${}_{\mc L} \underline{\epsilon S}_{\mc L}$ and on the $z_{\mc F}$ for $\mc F \in 
{}_{\mc L} \underline{S}^\circ_{\mc L}$. We only need to work
out what to do with the remainder of the twisted central structures.

We start in $Z ({}_{\mc L} \underline{\epsilon S}_{\mc L})$, so we have $z_{\mc F}$ for all
$\mc F \in {}_{\mc L} \underline{S}_{\mc L}$. We define $z_n$ as the composition
\begin{multline*}
{}^n \mc Z \xrightarrow{\mr{Lemma} \: \ref{lem:3.2}.a} 
(\ICL{n} * \mc Z) * \ICL{n^{-1}} \cong (\ICL{n} \circ \mc Z) \circ \ICL{n^{-1}} \\
\xrightarrow{z_{\ICL{n}} \circ \: \mr{id}} (\mc Z \circ \ICL{n}) \circ \ICL{n^{-1}} 
\xrightarrow{\eqref{eq:3.7}} \mc Z  .
\end{multline*}
Then (ii) follows from the associativity structure of $\circ$ and (iii) holds by definition.
To check (iv), we use Lemma \ref{lem:3.2}.a and \eqref{eq:3.55} for isomorphisms
\[
{}^n \mc F \cong (\ICL{n} \circ \mc F) \circ \ICL{n^{-1}} \quad \text{and} \quad
{}^n \mc Z \cong (\ICL{n} \circ \mc Z) \circ \ICL{n^{-1}} .
\]
Then $z_{{}^n \mc F}$ becomes (up to isomorphisms from the associativity structure)
\begin{multline*}
\ICL{n} \circ \mc F \circ \ICL{n^{-1}} \circ \mc Z \xrightarrow{\mr{id} \, \circ z_n^{-1}}
\ICL{n} \circ \mc F \circ \ICL{n^{-1}} \circ \ICL{n} \circ \mc Z \circ \ICL{n^{-1}} \\ 
\cong \ICL{n} \circ \mc F \circ \mc Z \circ \ICL{n^{-1}} \xrightarrow{\mr{id} \circ z_{\mc F} 
\circ \mr{id}} \ICL{n} \circ \mc Z \circ \mc F \circ \ICL{n^{-1}} \\
\cong \ICL{n} \circ \mc Z \circ \ICL{n^{-1}} \circ \ICL{n} \circ \mc F \circ \ICL{n^{-1}}
\xrightarrow{z_n \circ \, \mr{id}} \mc Z \circ \ICL{n} \circ \mc F \circ \ICL{n^{-1}} .
\end{multline*}
By \eqref{eq:3.54} and \eqref{eq:3.55} this equals the composition in (iv).

Conversely, suppose we start with $\mc Z \in Z^+ ({}_{\mc L} \underline{\epsilon S}_{\mc L})$,
so enriched with $z_{\mc F}$ for $\mc F \in {}_{\mc L} \underline{S}^\circ_{\mc L}$ and
$z_n$ for $n \in N_{\mc G}(\mc T)_{\mc L}^+$, satisfying (i)--(iv). 
We define $z_{\ICL{n} \circ \mc F}$ as the composition
\begin{multline}\label{eq:3.57}
\ICL{n} \circ \mc F \circ \mc Z \xrightarrow{\mr{id} \, \circ z_{\mc F}} \ICL{n} \circ 
\mc Z \circ \mc F \cong \\
\ICL{n} \circ \mc Z \circ \ICL{n^{-1}} \circ \ICL{n} \circ \mc F
\xrightarrow{z_n \circ \, \mr{id}} \mc Z \circ \ICL{n} \circ \mc F .
\end{multline}
If we rewrite $\ICL{n} \circ \mc F = \ICL{nt} \circ (\ICL{t^{-1}} * \mc F)$, then
\eqref{eq:3.54} on ${}_{\mc L} \underline{S}^\circ_{\mc L}$, (ii) and (iii) guarantee that
$z_{\ICL{nt} \circ (\ICL{t^{-1}} * \mc F)}$ equals \eqref{eq:3.56}. This defines
$z_{\ICL{g}}$ unambigously for all $g \in N_{\mc G}(\mc T)_{\mc L}$. These are all the
irreducible objects of ${}_{\mc L} \underline{S}_{\mc L}$, and by additivity we obtain 
$z_{\mc F'}$ fo all $\mc F' \in {}_{\mc L} \underline{S}_{\mc L}$.

For \eqref{eq:3.54}, it remains to check that $z_{({}^n \mc F) \circ \ICL{n}}$ equals 
$(z_{\mc F} \circ \mr{id})(\mr{id} \circ z_{\ICL{n}})$ for $\mc F \in {}_{\mc L} 
\underline{S}^\circ_{\mc L}$ and $n \in N_{\mc G}(\mc T)_{\mc L}^+$. 
With Lemma \ref{lem:3.2}.a and \eqref{eq:3.55} we rewrite
\begin{equation}\label{eq:3.58}
{}^n \mc F \circ \ICL{n} \cong (\ICL{n} \circ \mc F) \circ \ICL{n^{-1}} \circ \ICL{n}
\cong \ICL{n} \circ \mc F. 
\end{equation}
Hence $z_{({}^n \mc F) \circ \ICL{n}}$ is given by \eqref{eq:3.57} conjugated with 
\eqref{eq:3.58}. Using (iv) we express
$(z_{\mc F} \circ \mr{id})(\mr{id} \circ z_{\ICL{n}})$ as 
\begin{multline}\label{eq:3.59}
{}^n \mc F \circ \ICL{n} \circ \mc Z \xrightarrow{\mr{id} \, \circ z_{\ICL{n}}}
{}^n \mc F \circ \mc Z \circ \ICL{n} \xrightarrow{\mr{id} \, \circ z_n^{-1} \circ \, \mr{id}} 
{}^n \mc F \circ {}^n \mc Z \circ \ICL{n} \\ 
\cong {}^n (\mc F \circ \mc Z) \circ \ICL{n} \cong \ICL{n} \circ \mc F \circ \mc Z 
\xrightarrow{\mr{id} \, \circ z_{\mc F}} \ICL{n} \circ \mc Z \circ \mc F \\
\cong {}^n (\mc Z \circ \mc F) \circ \ICL{n} \cong {}^n \mc Z \circ {}^n \mc F \circ \ICL{n}
\xrightarrow{z_n \circ \, \mr{id}} \mc Z \circ {}^n \mc F \circ \ICL{n} .
\end{multline}
Here the composition of the first four isomorphisms equals \eqref{eq:3.58}$\, \circ \, \mr{id}$
and steps four till seven compose to ${}^n (z_{\mc F}) \circ \mr{id}$. 
Furthermore the composition of the last five isomorphisms in \eqref{eq:3.59} equals
\begin{multline*}
\ICL{n} \circ \mc F \circ \mc Z \xrightarrow{\mr{id} \, \circ z_{\mc F}} 
\ICL{n} \circ \mc Z \circ \mc F \cong \\ \ICL{n} \circ \mc Z \circ \ICL{n^{-1}} \circ \ICL{n} 
\circ \mc F \xrightarrow{z_n \circ \, \mr{id}} \mc Z \circ \ICL{n} \circ \mc F
\isom \mc Z \circ {}^n \mc F \circ \ICL{n} ,
\end{multline*}
which is the composition of \eqref{eq:3.57} and the inverse of \eqref{eq:3.58}.
Altogether this means that (up to isomorphisms from the associativity structure)
\eqref{eq:3.59} equals \eqref{eq:3.57} conjugated with \eqref{eq:3.58}.
\end{proof}

The advantage of the more involved twisted central structure in 
$Z^+ ({}_{\mc L} \underline{\epsilon S}_{\mc L})$ is that it arises naturally in the 
following result.

\begin{prop}\label{prop:3.4}
\enuma{
\item The composed equivalence
\[
\underline{CS}_{W \mc L} (\mc G; \epsilon) \isom 
Z ({}_{\mc L} \underline{S}_{\epsilon \mc L} ;\epsilon) \isom 
Z^+ ({}_{\mc L} \underline{\epsilon S}_{\mc L} )
\]
is canonical and $N_{\mc G}(\mc T )_{\mc L}^+$-equivariant.
\item Suppose that the image of $\mc F \in \underline{CS}_{W \mc L} (\mc G; \epsilon)$ 
under part (a) is $\mc Z \in Z^+ ({}_{\mc L} \underline{\epsilon S}_{\mc L} )$, and let
$n \in N_{\mc G}(\mc T )_{\mc L}^+$. Under the equivalence from part (a), the canonical 
isomorphism $\alpha_{n,\mc F} : \mr{Ad}(n)_* \mc F \to \mc F$ from 
\eqref{eq:3.40} corresponds to the isomorphism
$\phi_{n,\mc Z}^c : {}^n \mc Z \isom \mc Z$ 
from the twisted central structure of $\mc Z$.
}
\end{prop}
\begin{proof}(a)
These equivalences already appeared in \eqref{eq:3.4} and Lemmas \ref{lem:3.6} and 
\ref{lem:3.10}. To check the canonicity and equivariance, we spell out how the equivalence
\begin{equation}\label{eq:3.22}
\underline{CS}_{W \mc L} (\mc G; \epsilon) \isom 
Z ({}_{\mc L} \underline{S}_{\epsilon \mc L} ;\epsilon) 
\end{equation}
is constructed in \cite{Lus5,LuYu}. Consider the varieties
\begin{align*}
& Z = \{ (g_1 \mc B, g_2 \mc B, g_\epsilon (g_1 \mc U g_1^{-1}) \in 
\mc G / \mc B \times \mc G / \mc B \times \mc G \epsilon / (g_1 \mc U g_1^{-1}) :
g_\epsilon g_1 \mc B g_1 g_\epsilon^{-1} = g_2 \mc B g_2^{-1} \} ,\\
& \dot Z = \{ (g_1 \mc B, g_2 \mc B, g_\epsilon ) \in 
\mc G / \mc B \times \mc G / \mc B \times \mc G \epsilon :
g_\epsilon g_1 \mc B g_1 g_\epsilon^{-1} = g_2 \mc B g_2^{-1} \} 
\end{align*}
and the maps
\[
\begin{array}{cccc@{\hspace{12mm}}cccc}
\pi : \!\! & \dot Z & \to & \mc G \epsilon & f: \!\! & \dot Z & \to & Z \\
& (g_1 \mc B,g_2 \mc B, g_\epsilon) & \mapsto & g_\epsilon & & 
(g_1 \mc B,g_2 \mc B, g_\epsilon) & \mapsto & (g_1 \mc B,g_2 \mc B, g_\epsilon
(g_1 \mc U g_1^{-1}) ) .
\end{array}
\]
Notice that the term $g_2 \mc B$ (in $Z$ or $\dot Z$) is determined by the 
other coordinates, so it can be omitted without changing the variety. The group
$\mc G$ acts on $\mc G \epsilon$ by conjugation, on $Z$ by
\[
g \cdot (g_1 \mc B, g_2 \mc B, g_\epsilon (g_1 \mc U g_1^{-1})) =
(g g_1 \mc B, g g_2 \mc B, g g_\epsilon (g_1 \mc U g_1^{-1}) g^{-1})
\]
and on $\dot Z$ like on $Z$. The maps $\pi$ and $f$ are $\mc G$-equivariant, so
they induce morphisms of stacks 
\begin{equation}\label{eq:3.31}
\mc G \epsilon / \mc G \xleftarrow{\; \pi \;} \dot Z / \mc G \xrightarrow{\;f\;} Z / \mc G. 
\end{equation}
We let $\mc T$ act on $\mc U \backslash \mc G \epsilon / \mc U$ by conjugation,
then there is an isomorphism of stacks 
\begin{equation}\label{eq:3.32}
\begin{array}{cccc}
\phi_{\mc G \epsilon} : & Z / \mc G & \to & 
(\mc U \backslash \mc G \epsilon / \mc U) / \mc T, \\
& (g_1 \mc B, g_2 \mc B, g_\epsilon (g_1 \mc U g_1^{-1})) & \mapsto & 
\mc U g_1^{-1} g_\epsilon g_1 \mc U .
\end{array}
\end{equation}
With the isomorphism $\rho_\epsilon^{-1}$ from \eqref{eq:3.23} we can pass to
$(\mc U \backslash \mc G/ \mc U ) / \mc T$. In \cite[Theorem 7.3]{Lus5}, the 
equivalence \eqref{eq:3.22} is defined as the restriction of $f_! \circ \pi^*$ to
$\underline{CS}_{W \mc L} (\mc G; \epsilon)$, regarded as a category of
$\mc G$-equivariant (derived) sheaves on $\mc G \epsilon$. The translation from 
$\mc G$-equivariant sheaves on $Z$ to $\mc T$-equivariant sheaves on
$\mc U \backslash \mc G / \mc U$ is implicit in \cite{Lus5} and in the reformulation
\cite[Theorem 12.4]{LuYu}. From the isomorphisms 
\[
\begin{array}{ccccc}
Z / \mc G & \leftarrow & (\mc G / \mc U \times \mc G / \mc U) / (\mc G \times \mc T) &
\to & (\mc U \backslash \mc G/ \mc U ) / \mc T \\
(g_1 \mc B, g_2 \mc B, g_2 \epsilon \mc U g_1^{-1}) & \mapsfrom & (g_1 \mc U, g_2 \mc U) 
& \mapsto & \mc U g_1^{-1} g_2 \mc U
\end{array} 
\]
in \cite[\S 4.1]{Lus5} one sees that this translation is given by
$(\rho_\epsilon )_*^{-1} \circ (\phi_{\mc G \epsilon})_*$. This means that \eqref{eq:3.22}
is given by the functor
\begin{equation}\label{eq:3.24}
(\rho_\epsilon )_*^{-1} \circ (\phi_{\mc G \epsilon})_* \circ f_! \circ \pi^* :
\mc D^b_{\mc G}(\mc G \epsilon) \to \mc D^b_{\mc T}(\mc U \backslash \mc G / \mc U) .
\end{equation}
Hence the equivalence from the statement is given by the functor
\begin{equation}\label{eq:3.25}
(\phi_{\mc G \epsilon})_* \circ f_! \circ \pi^* : \mc D^b_{\mc G}(\mc G \epsilon) \to 
\mc D^b_{\mc T}(\mc U \backslash \mc G \epsilon / \mc U) ,
\end{equation}
which is canonical.

Consider the action of $n \in N_{\mc G}(\mc T)_{\mc L}^+$ on $\mc F \in 
\underline{CS}_{W \mc L} (\mc G; \epsilon)$. It replaces the stalk $i_{g_\epsilon}^* \mc F$
at $g_\epsilon \in N_{\mc G \epsilon}(\mc T)_{\mc L}$ by $i_{n^{-1} g_\epsilon n}^* \mc F$.
Let $\mc Z$ be the image of $\mc F$ under \eqref{eq:3.25}.
From the definitions we see that the stalk of 
$\mc Z$ at $\mc U h \mc U$ comes from $\Qe \boxtimes \mc F$ on 
\begin{equation}\label{eq:3.41}
\{ (g_1,g_\epsilon) \in \mc G / \mc B \times \mc G \epsilon : 
\mc U g_1^{-1} g_\epsilon g_1 \mc U = \mc U h \mc U \} .
\end{equation}
Applying $\mr{Ad}(n)_*$ to $\mc F$ means that in \eqref{eq:3.41} we replace $g_\epsilon$ 
by $n^{-1} g_\epsilon n$. That also gives the stalk of ${}^n \mc Z$ at $\mc U h \mc U$. 
Therefore $\mr{Ad}(n)_* \mc F$ is mapped to ${}^n \mc Z$ by \eqref{eq:3.25}.\\
(b) The central structure on the image of $\mc F$ in $Z ({}_{\mc L} 
\underline{S}_{\epsilon \mc L} ;\epsilon)$ is given in \cite[(6.13) and Theorem 7.3]{Lus5}, 
and transferred to $\mc Z$ by Lemmas \ref{lem:3.6} and \ref{lem:3.10}.
Unfortunately it is described in \cite{Lus5} in terms of $\mc G \times \mc T$-equivariant
sheaves on $(\mc G / \mc U)^2$, and not with $\mc T$-equivariant sheaves on
$\mc U \backslash \mc G \epsilon / \mc U$. To translate it to our setting, we write
$\mc Z = \bigoplus_i \underline{\IC}(h_i)_{\mc L}$ with $h_i \in N_{\mc G \epsilon}
(\mc T)_{\mc L}$. Then the $N_{\mc G}(\mc T)_{\mc L}^+$-part of the twisted central 
structure becomes an isomorphism
\[
z_n : \bigoplus\nolimits_i \underline{\IC}(n h_i n^{-1})_{\mc L} \isom 
\bigoplus\nolimits_i \underline{\IC}(h_i)_{\mc L} .
\] 
In the picture \eqref{eq:3.41} of the stalks of $\mc Z$, the isomorphism obtained 
from \cite{Lus5} via these steps sends $i_{n^{-1} g_\epsilon n}^* \mc F$ to
$i_{g_\epsilon}^* \mc F$, and it does so by $\alpha_{n,\mc F} : \mr{Ad}(n)_* \mc F
\to \mc F$. Thus $\alpha_{n,\mc F}$ and part (a) induce the isomorphism
$z_n : {}^n \mc Z \to \mc Z$ from the twisted central structure of $\mc Z$.
\end{proof}

We turn to the second equivalence in \eqref{eq:3.3} and \eqref{eq:3.4}. Here the 
action of $\Omega_{\mc L,\beta}$ on 
$Z ({}_{\mc L} \underline{S}^{\circ}_{\mc L} ; \mr{Ad}(\dot w^\beta) \circ \epsilon )$ 
is defined as follows. For $\omega \in \Omega_{\mc L,\beta}$ we pick a rigidified 
minimal IC-sheaf $\IC (w)_{\mc L}^\dagger$ \cite[\S 5.2]{LuYu}, whose construction 
relies on the choice of a ``Whittaker sheaf". We abbreviate the
non-mixed version of $\IC (w)_{\mc L}^\dagger$ to $\xi_w$. For 
\[
\mc F \in Z ({}_{\mc L} \underline{S}^{\circ}_{\mc L} ;\mr{Ad}(\dot w^\beta) \circ \epsilon )
\, \subset \,  _{\mc L} \underline{\mc D}_{\mc L} 
\]
the action is
\[
{}^w \mc F = \xi_w * \mc F * \xi_w^{-1} ,
\]
which up to canonical isomorphisms does not depend on the choice of $\IC (w)_{\mc L}^\dagger$. 
In particular we interpret ${}^w \mc F$ as $(\xi_w * \mc F ) * \xi_w^{-1}$, which by the
associativity structure is canonically isomorphic to $\xi_w * ( \mc F * \xi_w^{-1})$.

While the rigidifications  do not influence this action of $\Omega_{\mc L,\beta}$, they do 
pin down unique isomorphisms $\xi_w * \xi_{w'} \cong \xi_{w w'}$. This action of 
$\Omega_{\mc L,\beta}$ determines how an equivariant structure on $\mc F$ is understood 
in \cite{LuYu}, which is used to define  
\[
Z ({}_{\mc L} \underline{S}^{\circ}_{\mc L} ;\mr{Ad}(\dot w^\beta) \circ \epsilon
)^{\Omega_{\mc L,\beta}} \quad \subset \quad _{\mc L} \mc D_{\mc L}^{\Omega_{\mc L,\beta}} . 
\]
The disadvantage of this method is that the non-strict action of $\Omega_{\mc L,\beta}$
is not induced by a strict action of a larger group, because the involved rigidifications
with Whittaker sheaves do not relate well to the functors $\mr{Ad}(t)_*$ for $t \in \mc T$.

Recall from \eqref{eq:3.11} and Lemma \ref{lem:3.2} that there is a canonical isomorphism
\begin{equation}\label{eq:3.60}
{}^n \mc F \cong (\underline{\IC}(n)_{\mc L} * \mc F) * \underline{\IC}(n^{-1})_{\mc L}
\qquad \mc F \in {}_{\mc L} \underline{S}^\circ_{\mc L}, n \in N_{\mc G}(\mc T)_{\mc L}^+ .
\end{equation}
By \eqref{eq:3.55}, we may replace * by $\circ$ in \eqref{eq:3.60}. 
For $t \in T^{\sigma_{\beta \epsilon}}$ and $\mc F \in Z ({}_{\mc L} 
\underline{S}^{\circ}_{\mc L} ;\mr{Ad}(\dot w^\beta) \circ \epsilon )$,
the twisted central structure provides a canonical isomorphism 
$\phi_{t,\mc F}^c : {}^t \mc F \to \mc F$. 
This gives rise to a weak action of $\Omega_{\mc L,\beta}$ on ${}_{\mc L}
\underline{S}^\circ_{\mc L}$. From \eqref{eq:3.60} we see that it agrees with the
weak action used in \cite{LuYu}, provided we pick representatives $\dot \omega$
such that 
\begin{equation}\label{eq:3.16}
\IC (\dot \omega )_{\mc L} = \IC (\omega )_{\mc L}^\dagger .
\end{equation}
Since in \cite{LuYu} no choice of $\IC (\omega )_{\mc L}^\dagger$ is made (all 
possible choices are regarded as equivalent and then identified), we may align
the setup in \cite{LuYu} with ours by decreeing that we first pick $\dot \omega$
and then fix  $\IC (\omega )_{\mc L}^\dagger$ so that \eqref{eq:3.16} holds.

We note that we fixed $\mr{Ad}(\dot w^\beta) \circ \epsilon$, whereas in \cite{LuYu} an 
equivalence class of automorphisms (differing from $\mr{Ad}(\dot w^\beta) \circ \epsilon$ 
by conjugation with elements of $\mc T$) is used. This means that compared to
\cite{LuYu} we only have a natural action of 
$N_{G}(T)_{\mc L}^{\beta,\sigma_{\beta \epsilon}}$ on 
$Z ({}_{\mc L} \underline{S}^{\circ}_{\mc L} ;\mr{Ad}(\dot w^\beta) \circ \epsilon)$, 
not of the larger group $N_{\mc G}(\mc T)_{\mc L}^\beta$. In view of \eqref{eq:3.19}, 
that suffices to define a non-strict action of $\Omega_{\mc L,\beta}$.
In particular the equivalence of categories 
\begin{equation}\label{eq:3.13} 
Z ({}_{\mc L} \underline{S}_{\epsilon \mc L} ;\epsilon) \isom
\bigoplus_{\beta \in \mf B_{\mc L} / \mr{Ad}_\epsilon (\Omega_{\mc L})} 
Z ({}_{\mc L} \underline{S}^{\circ}_{\mc L} ; \mr{Ad}(\dot w^\beta) \circ \epsilon
)^{\Omega_{\mc L,\beta}} \cong
\Big( \bigoplus_{\beta \in \mf B_{\mc L}} Z ({}_{\mc L} \underline{S}^{\circ}_{\mc L} ;
\mr{Ad}(\dot w^\beta) \circ \epsilon) \Big)^{\Omega_{\mc L}}
\end{equation}
from \eqref{eq:3.4} remains valid with our action and the above isomorphisms 
$\phi_{t,\mc F}^c$, provided that in the proof 
\cite[\S 12.10]{LuYu} we use $\IC (n(\omega,\beta))_{\mc L}$ instead of 
$\xi_\beta = \underline{\IC}(\omega )_{\mc L}^\dagger$. This works  
because $\mr{Ad}(n(\omega,\beta))$ commutes with $\dot w^\beta \epsilon$.
Now \eqref{eq:3.13} apparently depends
on the choice of representatives $n(\omega,\beta)$, but for any alternative
representative $t n(\omega,\beta)$ with $t \in T^{\sigma_{\beta \epsilon}} $ 
the new version of \eqref{eq:3.13} is related to the earlier version by
$\phi_t^c$. Thus \eqref{eq:3.13} is still canonical up to unique isomorphisms,
as throughout \cite{LuYu}. A big advantage is that we are not forced to work with
\eqref{eq:3.16}, any other representative $\dot \omega \in N_{G}(T)_{\mc L}^{
\beta,\sigma_{\beta \epsilon}}$ of $\omega \in \Omega_{\mc L,\beta}$ works equally 
well in our setup. Thus we have avoided the usage of Whittaker data.

Recall that $N_{\mc H}(\mc T)$ is the preimage of $W_{\mc L}^\circ$ in $N_{\mc G}(\mc T)$. 
Lemma \ref{lem:3.6} is also valid with $\mr{Ad}(\dot w^\beta) \circ \epsilon$, 
then the appropriate version of ${}_{\mc L} \underline{\epsilon S}_{\mc L}$ is
\[
{}_{\mc L} \underline{\dot w^\beta \epsilon S}^\circ_{\mc L} 
\subset \mc D^b_{(\mc T \times \mc T, \mc L \boxtimes \mc L^{-1})}(
\mc U \backslash \mc G \dot w^\beta \epsilon / \mc U)
\]
the full subcategory formed by direct sums of objects $\underline{\IC}(g)_{\mc L}$ 
with $g \in N_{\mc H}(\mc T) \dot w^\beta \epsilon$. It admits 
truncated convolution products with ${}_{\mc L} \underline{S}^\circ_{\mc L}$, and 
$N_{\mc G}(\mc T)_{\mc L}^\beta$ acts on ${}_{\mc L} \underline{\dot w^\beta \epsilon 
S}^\circ_{\mc L}$ as in \eqref{eq:3.11}. We record that
\[
\phi_{\dot w^\beta \epsilon} : (\mc U \backslash \mc G / \mc U ) / \mc T \isom 
(\mc U \backslash \mc G \dot w^\beta \epsilon / \mc U ) / \mc T 
\]
is $N_{G}(T)_{\mc L}^{\beta, \sigma_{\beta \epsilon}}$-equivariant and
induces $T^{\sigma_{\beta \epsilon}}$-equivariant equivalences of categories
\[
\begin{array}{llll}
(\phi_{\dot w^\beta \epsilon})_* : & {}_{\mc L} \underline{S}^\circ_{\mc L} & \isom &
{}_{\mc L} \underline{\dot w^\beta \epsilon S}^\circ_{\mc L} ,\\
(\phi_{\dot w^\beta \epsilon} )_* : & Z({}_{\mc L} \underline{S}^\circ_{\mc L}; 
\mr{Ad}(\dot w^\beta) \circ \epsilon) & \isom & 
Z({}_{\mc L} \underline{\dot w^\beta \epsilon S}^\circ_{\mc L}) ,\\
(\phi_{\dot w^\beta \epsilon})_* : & Z({}_{\mc L} \underline{S}^\circ_{\mc L}; 
\mr{Ad}(\dot w^\beta) \circ \epsilon)^{\Omega_{\mc L,\beta}} & \isom & 
Z({}_{\mc L} \underline{\dot w^\beta \epsilon S}^\circ_{\mc L})^{\Omega_{\mc L,\beta}} .
\end{array}
\]
Recall that $\Omega_{\mc L} = N_{\mc G}(\mc T)_{\mc L}^+ / \mc T$. As explained in
\eqref{eq:A.10}, the group $N_{\mc G}(\mc T)_{\mc L}^+$ acts on the category
\begin{equation}\label{eq:3.44}
\big( \bigoplus\nolimits_{\beta' \in \mr{Ad}_\epsilon (\Omega_{\mc L}) \beta}
Z({}_{\mc L} \underline{\dot w^{\beta'} \epsilon S}^\circ_{\mc L}) \big)^{\Omega_{\mc L}} 
\cong Z({}_{\mc L} \underline{\dot w^\beta \epsilon S}^\circ_{\mc L})^{\Omega_{\mc L,\beta}} ,
\end{equation}
where the equivalence comes from Lemma \ref{lem:A.1}. For $n \in N_{\mc G}(\mc T)_{\mc L}^+$ 
and an object $\mc Z$ on the left hand side of \eqref{eq:3.44}, we note in \eqref{eq:A.11}
that the equivariant structure provides a natural isomorphism
\begin{equation}\label{eq:3.43}
\phi_{n,\mc Z}^c : {}^n \mc Z \isom \mc Z .
\end{equation}
Consider the diagram of equivalences 
\begin{equation}\label{eq:3.27}
\begin{array}{ccc}
Z( {}_{\mc L} \underline{S}_{\epsilon \mc L} ;\epsilon) & \isom & 
\big( {\ds \bigoplus\nolimits_{\beta \in \mf B_{\mc L}}} Z ({}_{\mc L} 
\underline{S}^{\circ}_{\mc L} ; \mr{Ad}(\dot w^\beta) \circ \epsilon ) 
\big)^{\Omega_{\mc L}} \\
\downarrow (\phi_\epsilon )_* & & 
\downarrow \bigoplus_\beta (\phi_{\dot w^\beta \epsilon})_* \\
Z^+ ( {}_{\mc L} \underline{\epsilon S}_{\mc L} ) & \isom & 
\big( {\ds \bigoplus\nolimits_{\beta \in \mf B_{\mc L}}}  Z ({}_{\mc L} 
\underline{\dot w^\beta \epsilon S}^{\circ}_{\mc L} ) \big)^{\Omega_{\mc L}} 
\end{array}
\end{equation}
where the lower horizontal arrow is defined so that the diagram commutes.

\begin{prop}\label{prop:3.8}
The lower horizontal equivalence in \eqref{eq:3.27} is 
$N_{\mc G}(\mc T)_{\mc L}^+$-equivariant and matches 
the canonical isomorphisms $\phi_n^c : {}^n ( \; ) \to \mr{id}$ on the two sides.
\end{prop}
\begin{proof}
First we work out a formula for this functor. The easiest is to start with
$\mc F_\beta \in Z ({}_{\mc L} \underline{S}^{\circ}_{\mc L} ; \mr{Ad}(\dot w^\beta) \circ
\epsilon )^{\Omega_{\mc L,\beta}}$. The proof of Lemma \ref{lem:A.1}.a identifies it with
\[
\bigoplus\nolimits_{\omega \in \Omega_{\mc L} / \Omega_{\mc L,\beta}} {}^{\dot \omega} 
\mc F_\beta \in \big( {\ds \bigoplus\nolimits_{\beta' \in \mf B_{\mc L}}} Z ({}_{\mc L} 
\underline{S}^{\circ}_{\mc L} ; \mr{Ad}(\dot w^{\beta'}) \circ \epsilon ) 
\big)^{\Omega_{\mc L}} ,
\]
where the representatives $\dot \omega$ are chosen so that
$\dot \omega \dot w^\beta \epsilon \dot \omega = \dot w^{\omega \cdot \beta} \epsilon$.
From \cite[\S 12.10]{LuYu} we see that this is mapped to 
\[
\bigoplus\nolimits_{\omega \in \Omega_{\mc L} / \Omega_{\mc L,\beta}} {}^{\dot \omega} 
\mc F_\beta * \underline{\IC}(\dot w^{\omega \cdot \beta})_{\mc L} \in 
Z( {}_{\mc L} \underline{S}_{\epsilon \mc L} ;\epsilon) .
\]
By \eqref{eq:3.7} ${}^{\dot \omega} \mc F_\beta * \underline{\IC}(\dot w^{\omega \cdot
\beta})_{\mc L}$ is the push-forward of ${}^{\dot \omega} \mc F_\beta$ under right
multiplication by $\dot w^{\omega \cdot \beta}$, so
\[
(\phi_\epsilon )_* \big( \bigoplus\nolimits_{\omega \in \Omega_{\mc L} / 
\Omega_{\mc L,\beta}} {}^{\dot \omega} \mc F_\beta * \underline{\IC}(\dot w^{\omega 
\cdot \beta})_{\mc L} \big) =
\bigoplus\nolimits_{\omega \in \Omega_{\mc L} / \Omega_{\mc L,\beta}} 
(\phi_{\dot w^{\omega \cdot \beta} \epsilon})_* {}^{\dot \omega} \mc F_\beta .
\]
Hence the map 
\begin{equation}\label{eq:3.28}
\big( {\ds \bigoplus\nolimits_{\beta \in \mf B_{\mc L}}} Z ({}_{\mc L} 
\underline{S}^{\circ}_{\mc L} ; \mr{Ad}(\dot w^\beta) \circ \epsilon ) 
\big)^{\Omega_{\mc L}} 
\to Z^+ ( {}_{\mc L} \underline{\epsilon S}_{\mc L} )
\end{equation}
in \eqref{eq:3.27} is just the identity on objects. Conversely, start with any 
$\mc Z \in Z( {}_{\mc L} \underline{S}_{\epsilon \mc L} ;\epsilon)$. Its support is 
contained in 
\[
\bigcup_{w \in {}_{\mc L} W_{\epsilon \mc L}} \mc B w \epsilon \mc B =
\bigsqcup_{\beta \in \mf B_{\mc L}} \bigcup_{w \in W^\circ_{\mc L}} 
\mc B w w^\beta \epsilon \mc B .
\]
We decompose $\mc Z = \bigoplus_{\beta \in \mf B_{\mc L}} \mc Z_\beta$ with
supp$(\mc Z_\beta) \subset  \bigcup_{w \in W^\circ_{\mc L}} \mc B w w^\beta \epsilon \mc B$. 
As in \cite[\S 12.10]{LuYu}, we note that the twisted central structure of $\mc Z$ means 
that $\mc Z_\beta$ determines $\mc Z_{\beta'}$ for any $\beta' \in \mr{Ad}_\epsilon 
(\Omega_{\mc L}) \beta$. In the same way the twisted central structure of $\mc Z$ provides
a twisted central structure on $\bigoplus_{\beta' \in \mr{Ad}_\epsilon (\Omega_{\mc L}) \beta}
\mc Z_\beta$, as required. Now \eqref{eq:3.28} shows that the lower horizontal equivalence 
in \eqref{eq:3.27} is 
\begin{equation}\label{eq:3.29}
Z^+ ( {}_{\mc L} \underline{\epsilon S}_{\mc L} ) \isom 
\big( \bigoplus\nolimits_{\beta \in \mf B_{\mc L}} Z ({}_{\mc L} \underline{\dot w^\beta
\epsilon S}^{\circ}_{\mc L} )\big)^{\Omega_{\mc L}} : 
\mc Z \mapsto \bigoplus\nolimits_{\beta \in \mf B_{\mc L}} \mc Z_\beta .
\end{equation}
This functor is $N_{\mc G}(\mc T)_{\mc L}^+$-equivariant, and by construction the
$\Omega_{\mc L}$-equivariant structure on the right hand side arises from the
$N_{\mc G}(\mc T)_{\mc L}^+$-part of the twisted equivariant structure of $\mc Z$.
Thus the canonical isomorphisms $\phi_n^c$ on the two sides are by construction the same.
\end{proof}

\subsection{Modification from unipotent to essentially unipotent} \

Now we deviate substantially from \cite{LuYu} and \eqref{eq:3.4}, by using 
$\mc L$ for $\mc H$ (instead of the trivial local system on $\mc T$). Recall that 
${}_{\mc L} \underline{S}^{\mc H}_{\mc L}$ is the analogue of ${}_{\mc L} 
\underline{S}^\circ_{\mc L}$ for $\mc H$. (It is also the analogue of 
${}_{\mc L} \underline{S}_{\mc L}$, since $W(\mc H,\mc T) = W_{\mc L}^\circ$.)
These objects are involved in our alternative version of the third equivalence in
\eqref{eq:3.4}. Recall from \eqref{eq:E.51} that $N_{\mc G}(\mc T)_{\mc L}^+ 
\overset{\mc T}{\times} \mc H$ is canonically a $\overline{\F_q}$-group.
For $h \in N_{\mc G}(\mc T)_{\mc L}$ we define 
\[
\IC (h)_{\mc H,\mc L} \in \mc D^b_{m,(\mc T \times \mc T, \mc L \boxtimes \mc L^{-1})}
\big( N_{\mc G}(\mc T)_{\mc L}^+ \overset{\mc T}{\times} \mc H \big)
\]
like $\IC (h)_{\mc L}$. Then $N_{\mc G}(\mc T)_{\mc L}^+$ acts on 
${}_{\mc L} \underline{S}^{\mc H}_{\mc L}$ by
\begin{equation}\label{eq:3.46}
{}^n \underline{\IC}(h)_{\mc H,\mc L} = \underline{\IC}(n h n^{-1})_{\mc H,\mc L}.
\end{equation}
Let ${}_{\mc L} \underline{\sigma_{\beta \epsilon} S}^{\mc H}_{\mc L}$ be the version of
${}_{\mc L} \underline{\epsilon S}_{\mc L}$ for $(\mc H; \sigma_{\beta \epsilon})$. 
We define $\sigma_{\beta \epsilon}$ on $N_{\mc G}(\mc T)_{\mc L}^+$ as 
$\mr{Ad}(\dot w^\beta) \circ \epsilon$. Then $N_{\mc G}(\mc T)_{\mc L}^+$ acts by
conjugation on 
\[
\bigcup\nolimits_{\beta \in \mf B_{\mc L}} \mc H \sigma_{\beta \epsilon} = 
\bigcup\nolimits_{\beta \in \mf B_{\mc L}} \sigma_{\beta \epsilon} \mc H  ,
\]
so \eqref{eq:3.46} also defines an action of $N_{\mc G}(\mc T)_{\mc L}^+$ on
$\bigoplus_{\beta \in \mf B_{\mc L}} {}_{\mc L} \underline{\sigma_{\beta \epsilon} 
S}^{\mc H}_{\mc L}$.

\begin{thm}\label{thm:3.3}
\enuma{
\item There is a canonical $N_{\mc G}(\mc T)_{\mc L}^+$-equivariant monoidal 
equivalence of categories 
\[
{}_{\mc L} \underline{S}^{\mc H}_{\mc L} \isom {}_{\mc L} \underline{S}^\circ_{\mc L},
\]
which sends $\underline{\IC} (n)_{\mc H,\mc L}$ to $\underline{\IC} (n)_{\mc L}$ for all 
$n \in N_{\mc H}(\mc T)$.
\item Part (a) induces a canonical $N_{\mc G}(\mc T)_{\mc L}^+$-equivariant equivalence 
of categories
\[
\bigoplus\nolimits_{\beta \in \mf B_{\mc L}} {}_{\mc L} \underline{\sigma_{\beta \epsilon} 
S}^{\mc H}_{\mc L} \isom \bigoplus\nolimits_{\beta \in \mf B_{\mc L}} 
{}_{\mc L} \underline{\dot w^\beta \epsilon S}^\circ_{\mc L},
\]
which sends $\underline{\IC} (n \sigma_{\beta \epsilon})_{\mc H,\mc L}$ to 
$\underline{\IC} (n \dot w^\beta \epsilon)_{\mc L}$ for all $n \in N_{\mc H}(\mc T)$. 
When we restrict this equivalence to the twisted categorical centres on both sides, 
the two canonical isomorphisms $\phi_n^c : {}^n ( \; ) \to \mr{id}$ with 
$n \in N_{\mc H}(\mc T)$ correspond via the equivalence.
\item Parts (a) and (b) give rise to a canonical diagram of equivalences
\[
\begin{array}{ccc}
\big( \bigoplus_{\beta \in \mf B_{\mc L}}
Z ({}_{\mc L} \underline{S}^{\mc H}_{\mc L} ;\sigma_{\beta \epsilon} )
\big)^{\Omega_{\mc L}} & \isom & \big( \bigoplus_{\beta \in \mf B_{\mc L}}
Z({}_{\mc L} \underline{S}^\circ_{\mc L} ;
\mr{Ad}(\dot w^\beta) \circ \epsilon ) \big)^{\Omega_{\mc L}} \\
\downarrow \bigoplus_\beta (\rho_{\sigma_{\beta \epsilon}})_* & & 
\downarrow \bigoplus_\beta (\rho_{\dot w^\beta \epsilon})_* \\
\big( \bigoplus_{\beta \in \mf B_{\mc L}} Z ({}_{\mc L} \underline{\sigma_{\beta 
\epsilon} S}^{\mc H}_{\mc L}) \big)^{\Omega_{\mc L}} & \isom & 
\big( \bigoplus_{\beta \in \mf B_{\mc L}} Z({}_{\mc L} \underline{\dot w^\beta \epsilon S
}^\circ_{\mc L}) \big)^{\Omega_{\mc L}} 
\end{array}.
\]
}
\end{thm}
\emph{Remark.} Notice that part (a) does not work with $\underline{S}_{\mc H}$ 
instead of ${}_{\mc L} \underline{S}^{\mc H}_{\mc L}$, because
$\IC (n t)_{\mc H} = \IC (n)_{\mc H}$ for all $t \in \mc T$, while 
$\IC (n t)_{\mc L} = \IC (n)_{\mc L} \otimes \mc L_t \neq \IC (n)_{\mc L}$ for most $t$.
\begin{proof}
(a) The canonical monoidal equivalence of categories can be obtained by applying 
\cite[Theorem 9.2]{LuYu} twice, once for $\mc G$ and once for $\mc H$, that gives
\begin{equation}\label{eq:3.14}
{}_{\mc L} \underline{S}^{\mc H}_{\mc L} \isom \underline{S}_{\mc H} \isom
{}_{\mc L} \underline{S}^\circ_{\mc L} . 
\end{equation} 
We may also apply the arguments in \cite[\S 9]{LuYu} directly, with 
${}_{\mc L} \underline{S}^{\mc H}_{\mc L}$ instead of 
${}_{\mc L} \underline{S}^\circ_{\mc L}$. In \cite{LuYu}, the equivalence \eqref{eq:3.14}
is characterized by
\[
\IC (w)_{\mc H,\mc L}^\dagger \mapsto \IC (w)_{\mc L}^\dagger 
\quad \text{for } w \in W_{\mc L}^\circ . 
\]
Here $\IC (w)_{\mc H,\mc L}^\dagger$ means $\IC (w)_{\mc L}^\dagger$ for $\mc H$, which is
in fact an equivalence class of objects that are related by unique isomorphisms. We may
specialize $\IC (w)_{\mc H,\mc L}^\dagger$ to $\IC (\dot w)_{\mc H,\mc L}$ and
$\IC (w)_{\mc L}^\dagger$ to $\IC (\dot w)_{\mc L}$, for any representative $\dot w \in 
N_{\mc G}(\mc T)$ of $w \in W_{\mc L}^\circ$, then \eqref{eq:3.14} attains the desired
property $\underline{\IC} (n)_{\mc H,\mc L} \mapsto \underline{\IC} (n)_{\mc L}$.

All objects in ${}_{\mc L} \underline{S}^{\mc H}_{\mc L}$ and ${}_{\mc L} 
\underline{S}^\circ_{\mc L}$ are direct sums of such IC-sheaves (or more precisely their
non-mixed versions). The definitions of the actions \eqref{eq:3.11} and \eqref{eq:3.46}
show that \eqref{eq:3.14} is $N_{\mc G}(\mc T)_{\mc L}^+$-equivariant.\\
(b) The equivalence is given by composing part (a) with $(\rho_{\dot w^\beta \epsilon})_*$
and $(\rho_{\sigma_{\beta \epsilon}})_*^{-1}$. The $N_{\mc G}(\mc T)_{\mc L}^+
$-equivariance can be shown in the same way as in part (a).

Recall that the automorphisms $\mr{Ad}(\dot w^\beta) \circ \epsilon$ of $\mc G$ and
$\sigma_{\beta \epsilon}$ of $\mc H$ restrict to the same automorphism of 
$N_{\mc H}(\mc T)$. The equivalences from parts (a) and (b) turn the convolution products 
\[
{}_{\mc L} \underline{S}^{\mc H}_{\mc L} \times 
{}_{\mc L} \underline{S}^{\mc H}_{\mc L} \to 
{}_{\mc L} \underline{S}^{\mc H}_{\mc L}  \quad \text{and} \quad
{}_{\mc L} \underline{\sigma_{\beta \epsilon} S}^{\mc H}_{\mc L} \times 
{}_{\mc L} \underline{S}^{\mc H}_{\mc L} \to 
{}_{\mc L} \underline{\sigma_{\beta \epsilon} S}^{\mc H}_{\mc L} 
\]
into the convolution products
\[
{}_{\mc L} \underline{S}^\circ_{\mc L} \times 
{}_{\mc L} \underline{S}^\circ_{\mc L} \to 
{}_{\mc L} \underline{S}^\circ_{\mc L} \quad \text{and} \quad
{}_{\mc L} \underline{\dot w^\beta \epsilon S}^\circ_{\mc L} \times 
{}_{\mc L} \underline{S}^\circ_{\mc L} \to 
{}_{\mc L} \underline{\dot w^\beta \epsilon S}^\circ_{\mc L} ,
\]
and similarly in the other multiplication order. It follows that the same holds for the
truncated convolution products, so the equivalence from (b) remains valid if we pass to the 
twisted categorical centres on both sides. Then the isomorphisms 
$\phi_n^c : {}^n (\;) \to \mr{id}$ are given by \eqref{eq:3.43} on both
sides, so these correspond under the equivalence.\\
(c) Part (b) and Lemma \ref{lem:3.6} give 
a canonical diagram of $N_{\mc G}(\mc T)_{\mc L}^+$-equivariant equi\-va\-lences
\begin{equation}\label{eq:3.30}
\begin{array}{ccc}
\bigoplus_{\beta \in \mf B_{\mc L}}
Z ({}_{\mc L} \underline{S}^{\mc H}_{\mc L} ;\sigma_{\beta \epsilon} ) & 
\isom & \bigoplus_{\beta \in \mf B_{\mc L}} Z({}_{\mc L} \underline{S}^\circ_{\mc L} ;
\mr{Ad}(\dot w^\beta) \circ \epsilon ) \\
\downarrow \bigoplus_\beta (\rho_{\sigma_{\beta \epsilon}})_* & & 
\downarrow \bigoplus_\beta (\rho_{\dot w^\beta \epsilon})_* \\
\bigoplus_{\beta \in \mf B_{\mc L}} Z ({}_{\mc L} \underline{\sigma_{\beta 
\epsilon} S}^{\mc H}_{\mc L}) & \isom & 
\bigoplus_{\beta \in \mf B_{\mc L}} Z({}_{\mc L} \underline{\dot w^\beta \epsilon S
}^\circ_{\mc L}) 
\end{array}.
\end{equation}
For $t \in T^{\sigma_{\beta \epsilon}}$, the canonical isomorphisms $\phi_t^c : 
{}^t (\;) \to \mr{id}$ in the four terms of \eqref{eq:3.30} associated to $\beta$ 
correspond via the equivalences in the diagram, that is checked for the lower horizontal 
arrow in part (b) and for the other arrows it follows in the same way. Thus we can form 
non-strict actions of 
\[
\Omega_{\mc L,\beta} = N_{\mc G}(\mc T)_{\mc L}^{\beta, \sigma_{\beta \epsilon}} / 
T^{\sigma_{\beta \epsilon}} .
\]
on all four terms indexed by $\beta$ in \eqref{eq:3.30}, and the correcting isomorphisms
\eqref{eq:A.4} match.
Therefore \eqref{eq:3.30} remains valid if we enrich each term from $\beta$ with 
$\Omega_{\mc L,\beta}$-equivariant strutures. With Lemma \ref{lem:A.1} we promote 
that to $\Omega_{\mc L}$-equivariant structures.
\end{proof}

We proceed to the final equivalence in \eqref{eq:3.4}, with $\mc L$ instead of the trivial
local system on $\mc T$.

\begin{prop}\label{prop:3.7}
\enuma{
\item The canonical equivalence
\[
\underline{CS}_{\mc L} (\mc H; \sigma_{\beta \epsilon}) \isom
Z( {}_{\mc L} \underline{\sigma_{\beta \epsilon} S}^{\mc H}_{\mc L})
\]
from Proposition \ref{prop:3.4}.a is $N_{\mc G}(\mc T)_{\mc L}^\beta$-equivariant. For
$t \in \mc T$ it matches the canonical isomorphisms $\phi_t^c : \mr{Ad}(t)_* \to \mr{id}$ 
on the left with $\phi_t^c : {}^t (\;) \to \mr{id}$ on the right.
\item Part (a) induces a canonical $N_{\mc G}(\mc T)_{\mc L}^+$-equivariant 
equivalence of categories
\[
\big( \bigoplus\nolimits_{\beta \in \mf B_{\mc L}} \underline{CS}_{\mc L} 
(\mc H; \sigma_{\beta \epsilon}) \big)^{\Omega_{\mc L}} \isom 
\big( \bigoplus\nolimits_{\beta \in \mf B_{\mc L}} Z( {}_{\mc L} 
\underline{\sigma_{\beta \epsilon} S}^{\mc H}_{\mc L}) \big)^{\Omega_{\mc L}} .
\]
For $n \in N_{\mc G}(\mc T)_{\mc L}^+$, with the notations from Appendix \ref{sec:A}, 
it matches the isomorphisms $\phi_n^{\Omega_{\mc L},c} : \mr{Ad}(n)_*^{\Omega_{\mc L}} 
\to \mr{id}$ on the left with $\phi_n^{\Omega_{\mc L},c}$ on the right.
}
\end{prop}
\begin{proof}
(a) The $N_{\mc G}(\mc T)_{\mc L}^\beta$-action on $\underline{CS}_{\mc L} (\mc H;
\sigma_{\beta \epsilon})$ is by the functors $\mr{Ad}(n)_*$, where the conjugation 
is performed in
\[
\sigma_{\beta \epsilon} \mc H \overset{\mc T}{\times} N_{\mc G}(\mc T)_{\mc L}^\beta =
N_{\mc G}(\mc T)_{\mc L}^\beta \overset{\mc T}{\times} \mc H \sigma_{\beta \epsilon} ,
\]
a coset of the group $\mc H \overset{\mc T}{\times} N_{\mc G}(\mc T)_{\mc L}^\beta$.
As $\Omega_{\mc L,\beta}$ and $N_{\mc G}(\mc T)_{\mc L}^\beta$ stabilize $\Phi_{\mc L}^+$,
the group $N_{\mc G}(\mc T)_{\mc L}^\beta$ normalizes $\mc B_{\mc H}$. That enables us
to extend the $\mc T$-action on $\dot Z$ (from the proof of Proposition \ref{prop:3.4},
but now for $\mc H$ and $\sigma_{\beta \epsilon}$) to $N_{\mc G}(\mc T)_{\mc L}^\beta$ by
\[
n \cdot (h_1 \mc B_{\mc H}, h_2 \mc B_{\mc H}, h_\sigma) =
(n h_1 n^{-1} \mc B_{\mc H}, n h_2 n^{-1} \mc B_{\mc H}, n h_\sigma n^{-1}).
\]
Similarly $N_{\mc G}(\mc T)_{\mc L}^\beta$ acts on the version of $Z$ for $(\mc H,
\sigma_{\beta \epsilon})$, by conjugation. This makes $f, \pi$ and $f_! \circ \pi^*$
from \eqref{eq:3.31} $N_{\mc G}(\mc T)_{\mc L}^\beta$-equivariant. The explicit formula
for $\phi_{\mc H \sigma_{\beta \epsilon}}$ from \eqref{eq:3.32} shows that 
$(\phi_{\mc H \sigma_{\beta \epsilon}})_*$ is also equivariant for the conjugation
action of $N_{\mc G}(\mc T)_{\mc L}^\beta$. Hence the entire functor \eqref{eq:3.25}
for $(\mc H ,\sigma_{\beta \epsilon})$ is $N_{\mc G}(\mc T)_{\mc L}^\beta$-equivariant.
The compatibility of the two isomorphisms $\phi_t^c$ was proven in Proposition 
\ref{prop:3.4}.b.\\
(b) Part (a) induces a canonical equivalence of categories
\begin{equation}\label{eq:3.33}
\underline{CS}_{\mc L} (\mc H; \sigma_{\beta \epsilon})^{\Omega_{\mc L,\beta}} \isom
Z( {}_{\mc L} \underline{\sigma_{\beta \epsilon} S}^{\mc H}_{\mc L} )^{\Omega_{\mc L,\beta}}.
\end{equation}
It is $\mc T$-equivariant because part (a) is $\mc T$-equivariant, and the 
isomorphisms $\phi_t^c$ on the two sides of 
\eqref{eq:3.33} are already matched in the setting of part (a). By the
$N_{\mc G}(\mc T)_{\mc L}^\beta$-equivariance in part (a) and by \eqref{eq:A.10}, 
\eqref{eq:3.33} is also $N_{\mc G}(\mc T)_{\mc L}^\beta$-equivariant. 

For $n \in N_{\mc G}(\mc T)_{\mc L}^\beta$, the isomorphisms $\phi_n^{\Omega_{\mc L},c}$ 
on the two sides of \eqref{eq:3.33} are already matched in the setting of part (a), 
so it remains to compare their effects on the equivariant structures. In \eqref{eq:A.11} 
we checked that this is given by composition with instances of that very equivariant
structure on an object. That is the same recipe
on both sides of \eqref{eq:3.33}, and matches via that equivalence.

With Lemma \ref{lem:A.1}, we promote \eqref{eq:3.33} to an equivalence
\begin{equation}\label{eq:3.45}
\big( \bigoplus_{\beta' \in \mr{Ad}_\epsilon (\Omega_{\mc L}) \beta} \underline{CS}_{\mc L}
(\mc H; \sigma_{\beta' \epsilon}) \big)^{\Omega_{\mc L}} \isom
\big( \bigoplus_{\beta' \in \mr{Ad}_\epsilon (\Omega_{\mc L}) \beta} 
Z( {}_{\mc L} \underline{\sigma_{\beta' \epsilon} S}^{\mc H}_{\mc L}) 
\big)^{\Omega_{\mc L}} 
\end{equation}
which is $N_{\mc G}(\mc T)_{\mc L}^+$-equivariant and matches the isomorphisms
$\phi_{n}^{\Omega_{\mc L},c}$ on the two sides. Finally, we take the direct sum 
of \eqref{eq:3.45} over all $\mr{Ad}_\epsilon (\Omega_{\mc L})$-orbits in $\mf B_{\mc L}$.
\end{proof}

We are ready to formulate and prove an alternative version of \eqref{eq:3.4}.

\begin{thm}\label{thm:3.5}
\enuma{
\item There exist canonical $N_{\mc G}(\mc T)_{\mc L}^+$-equivariant 
equivalences of categories
\begin{align*}
& \underline{CS}_{W \mc L} (\mc G; \epsilon) \isom 
Z^+ ({}_{\mc L} \underline{\epsilon S}_{\mc L}) \isom
\big( \bigoplus\nolimits_{\beta \in \mf B_{\mc L}} 
Z ({}_{\mc L} \underline{\dot w^\beta \epsilon S}^{\circ}_{\mc L}) \big)^{\Omega_{\mc L}} \\
& \mosi \big( \bigoplus\nolimits_{\beta \in \mf B_{\mc L}} Z ( {}_{\mc L}
\underline{\sigma_{\beta \epsilon} S}^{\mc H}_{\mc L}) \big)^{\Omega_{\mc L}}
\mosi \big( \bigoplus\nolimits_{\beta \in \mf B_{\mc L}} 
\underline{CS}_{\mc L} (\mc H; \sigma_{\beta \epsilon}) \big)^{\Omega_{\mc L}} .
\end{align*}
Furthermore, for $n \in N_{\mc G}(\mc T)_{\mc L}^+$ the canonical isomorphisms 
$\phi_n^c$ in the five categories are matched via these equivalences.  
\item There exist canonical equivalences of categories
\begin{multline}
\Rep_{W \mc L} (G^\epsilon) \cong \underline{CS}_{W \mc L} (\mc G; \epsilon) \cong
\big( \bigoplus\nolimits_{\beta \in \mf B_{\mc L}}
\underline{CS}_{\mc L} (\mc H; \sigma_{\beta \epsilon} ) \big)^{\Omega_{\mc L}} \\
\cong \big( \bigoplus\nolimits_{\beta \in \mf B_{\mc L}} 
\Rep_{\mc L} (H^{\sigma_{\beta \epsilon}}) \big)^{\Omega_{\mc L}} \cong
\bigoplus\nolimits_{\beta \in \mf B_{\mc L} / \mr{Ad}_\epsilon (\Omega_{\mc L})} 
\Rep_{\mc L} (H^\Omega_{\mc L} (\beta,\beta) ).
\end{multline}
}
\end{thm}
\begin{proof}
(a) This is proven in Propositions \ref{prop:3.4}, \ref{prop:3.8}, \ref{prop:3.7} and 
Theorem \ref{thm:3.3}.\\
(b) This follows from part (b), Lemma \ref{lem:3.1} and \eqref{eq:3.50}.
\end{proof}

We also record the consequences of Theorem \ref{thm:3.5} for the rational series in
$\Rep (G^\epsilon)$. Recall the notations and the bookkeeping with $W(\mc L,\beta)$
and $s_\beta$ in Proposition \ref{prop:E.6} and \eqref{eq:E.23}.

\begin{thm}\label{thm:3.9}
Suppose that $W(\mc L,\beta)$ corresponds to $\mr{Ad}(G^{\vee \epsilon}) s_\beta$ via
\eqref{eq:E.23}. Theorem \ref{thm:3.5}.b restricts to canonical equivalences of categories
\[
\Rep_{s_\beta}(G^\epsilon) = \Rep_{W(\mc L,\beta)}(G^\epsilon) \cong
\Rep_{s_{\mc L}} (H^{\sigma_{\beta \epsilon}} )^{\Omega_{\mc L,\beta}} \cong
\Rep_{\mc L} (H^\Omega_{\mc L} (\beta,\beta)) .
\] 
\end{thm}
\begin{proof}
The first equivalence can be shown in the same way as Theorem \ref{thm:E.7}, we only have 
to replace $\Rep_1 (H^{\sigma_{\beta \epsilon}})$ by 
$\Rep_{\mc L} (H^{\sigma_{\beta \epsilon}} )$ everywhere, and invoke Theorem \ref{thm:3.5} 
instead of Theorem \ref{thm:E.1}. The equality $\Rep_{\mc L} (H^{\sigma_{\beta \epsilon}} ) = 
\Rep_{s_{\mc L}} (H^{\sigma_{\beta \epsilon}} )$ comes from \eqref{eq:E.38}, in 
combination with $s_{\mc L} \in T^{\sigma_{\beta \epsilon}} \subset 
H^{\sigma_{\beta \epsilon}}$. The last equivalence is \eqref{eq:3.50}.
\end{proof}

\section{Disconnected reductive groups} 
\label{sec:3}

We no longer require that $\mc G$ is connected. Instead, we assume the following 
throughout this section:

\begin{cond}\label{cond:E.3}
$\mc G^\epsilon$ is a smooth $\F_q$-group scheme with reductive neutral component 
$\mc G^{\circ \epsilon}$, such that every connected component of $\mc G$ is $\epsilon$-stable. 
All irreducible $\Qe$-representations of $G^\epsilon$ have finite dimension.
\end{cond}

If $\mc G$ would have connected components that are not $\epsilon$-stable, then they would not 
contribute to $G^\epsilon$ and could be ignored for the study of $G^\epsilon$-representations. 
Hence this part of Condition \ref{cond:E.3} is not a restriction, it merely serves 
to simplify the notations. 

Let $\pi_0 (\mc G)$ be the group of connected components of $\mc G$. We note that it
may be infinite. If $x \in \mc G$ represents an element of $\pi_0 (\mc G)$, then 
$x \epsilon (x^{-1}) \in \mc G^\circ$. By Lang's theorem there exist $g \in \mc G^\circ$ 
such that $g \epsilon (g^{-1}) = x \epsilon (x^{-1})$. Then $g^{-1} x \in \mc G^\circ x$ 
is fixed by $\epsilon$, so $\mc G^\circ x$ has $\F_q$-points.

The group $G^\epsilon = \mc G^\epsilon (\F_q)$ acts strictly on $\Rep (G^{\circ \epsilon})$,
by conjugation. This induces a weak action on $\pi_0 (\mc G)$ on 
$\Rep (G^{\circ \epsilon})$, see Appendix \ref{par:A.2}.
Let $\Rep (G^\epsilon)$ be the category of all finite dimensional
$\Qe G^\epsilon$-representations. By Condition \ref{cond:E.3}, it includes all
irreducible representations. Lemma \ref{lem:A.2} provides a canonical equivalence of categories
\begin{equation}
\Rep (G^{\circ \epsilon})^{\pi_0 (\mc G)} \cong \Rep (G^\epsilon) .
\end{equation}
Let $\mc B \supset \mc T$ be a Borel subgroup and a maximal torus of $\mc G^\circ$, both
$\epsilon$-stable.

\begin{lem}\label{lem:E.9}
The group $G^\epsilon$ equals $G^{\circ \epsilon} N_{G^\epsilon}(T^\epsilon, B^\epsilon)$.
\end{lem}
\begin{proof}
For $x \in G^\epsilon$, $x \mc B x^{-1}$ is an $\epsilon$-stable Borel subgroup of 
$\mc G^\circ$. All such subgroups are $G^{\circ \epsilon}$-conjugate 
\cite[Theorem 15.4.6.ii]{Spr}, so there exists $g \in G^{\circ \epsilon}$ such that 
$g x \mc B x^{-1} g^{-1} = \mc B$. 

As $gx \in N_{G^\epsilon}(B^\epsilon)$, $gx \mc T x^{-1} g^{-1}$ is a maximal 
$\F_q$-torus in the solvable group $\mc B$. All such tori are $B^\epsilon$-conjugate \cite[Theorem 14.4.3]{Spr}, so we can find $b \in B^\epsilon$ such that  
$bgx \mc T (bgx)^{-1} = \mc T$. Then $bgx \in N_{G^\epsilon}(T^\epsilon, B^\epsilon)$,
and hence $x \in G^{\circ \epsilon} N_{G^\epsilon}(T^\epsilon,B^\epsilon)$.
\end{proof}

Lemma \ref{lem:E.9} implies that the natural map 
\begin{equation}\label{eq:E.33}
N_{G^\epsilon}(T^\epsilon, B^\epsilon) / T^\epsilon \to \pi_0 (\mc G) \quad
\text{is a group isomorphism.} 
\end{equation}
That yields split short exact sequences
\begin{equation}\label{eq:E.10}
\begin{array}{ccccccccc}
1 & \to & W^\circ = W(\mc G^\circ,\mc T) & \to & W = W (\mc G,\mc T) & \to & \pi_0 (\mc G) & \to & 1 ,\\
1 & \to & W(G^{\circ \epsilon},T^\epsilon) & \to & W (G^\epsilon,T^\epsilon) & 
\to & \pi_0 (\mc G) & \to & 1 . \\
\end{array}
\end{equation}
Condition \ref{cond:E.3} enables us to pick representatives $\dot \omega \in 
N_{G^\epsilon}(T^\epsilon, B^\epsilon)$ for $\omega \in \pi_0 (\mc G)$, unique up to $T^\epsilon$. 
The (strict) action of $N_{G^\epsilon}(T^\epsilon, B^\epsilon)$ on 
$\Rep (G^{\circ \epsilon})$ permutes the subcategories 
$\Rep_{W^\circ \mc L} (G^{\circ \epsilon})$ and stabilizes the subcategory
\begin{equation}\label{eq:E.8}
\Rep_{W \mc L}(G^{\circ \epsilon}) := \bigoplus_{w \in W^\circ \setminus W / W_{\mc L}}
\Rep_{W^\circ w \mc L} (G^{\circ \epsilon})
\end{equation}
generated by (irreducible) $G^{\circ\epsilon}$-representations with semisimple parameter 
in $W \mc L$. Via \eqref{eq:E.33}, this strict action induces a weak action
of $\pi_0 (\mc G)$ on \eqref{eq:E.8}.

We define the geometric series $\Rep_{W \mc L}(G^\epsilon)$ as the category of 
$G^\epsilon$-representations whose restriction to $G^{\circ \epsilon}$ lies in 
$\Rep_{W \mc L}(G^{\circ \epsilon})$. Thus 
Lemma \ref{lem:A.2} provides a natural equivalence of categories
\begin{equation}\label{eq:E.7}
\Rep_{W \mc L}(G^\epsilon) \cong \Rep_{W \mc L}(G^{\circ \epsilon})^{\pi_0 (\mc G)} .
\end{equation}
The group $\pi_0 (\mc G) \cong W / W^\circ \cong N_{G^\epsilon}(T^\epsilon, B^\epsilon) 
/ T^\epsilon $ permutes the direct summands in \eqref{eq:E.8} 
transitively, which by Lemma \ref{lem:A.1} leads to a natural equivalence
\begin{equation}\label{eq:E.4}
\Rep_{W \mc L}(G^{\circ \epsilon})^{\pi_0 (\mc G)} \cong 
\Rep_{W^\circ \mc L}(G^{\circ \epsilon})^{\pi_0 (\mc G)_{W^\circ \mc L}} . 
\end{equation}
We define $\Omega_{\mc L} = \mr{Stab}_{W_{\mc L}}(\Phi_{\mc L}^+) \cong W_{\mc L} / W_{\mc L}^\circ$ 
as before, and we let $\Omega_{\mc L}^\circ \cong (W^\circ)_{\mc L} / W_{\mc L}^\circ$ 
be its version for $\mc G^\circ$. By \eqref{eq:E.10} there are natural isomorphisms
\begin{equation}\label{eq:E.5}
\Omega_{\mc L} / \Omega_{\mc L}^\circ \cong \pi_0 (\mc G)_{W^\circ \mc L} \cong  
W(G^\epsilon,T^\epsilon )_{W^\circ \mc L} / W(G^{\circ \epsilon},T^\epsilon ) \cong 
N_{G^\epsilon}(T^\epsilon)_{W^\circ \mc L} / N_{G^{\circ \epsilon}}(T^\epsilon) .
\end{equation}
Let $\mf B_{\mc L}^\circ$ be $\mf B_{\mc L}$ for $\mc G^\circ$. For $w \in W^\circ$
and $\gamma \in W$ we have $\gamma w \epsilon (\gamma)^{-1} \in W^\circ$ because 
$\epsilon$ fixes $\pi_0 (\mc G)$ pointwise. Hence $\mr{Ad}_\epsilon$ defines an action 
of $\Omega_{\mc L}$ on $\mf B_{\mc L}^\circ$.

Let $N_{\mc G}(\mc T)_{\mc L}^+ = \mr{Stab}_{N_{\mc G}(\mc T)} (\mc L, \Phi_{\mc L}^+)$,
so that
\begin{equation}\label{eq:4.1}
N_{\mc G}(\mc T)_{\mc L}^+ / \mc T \cong \Omega_{\mc L} ,\quad
N_{\mc G^\circ}(\mc T)_{\mc L}^+ / \mc T \cong \Omega_{\mc L}^\circ ,\quad
N_{\mc G}(\mc T)_{\mc L}^+ / N_{\mc G^\circ}(\mc T)_{\mc L}^+ \cong 
\Omega_{\mc L} / \Omega_{\mc L}^\circ .
\end{equation}
Consider the five categories in Theorem \ref{thm:3.5}, for $\mc G^\circ$. They are
endowed with actions of $N_{\mc G^\circ}(\mc T)_{\mc L}^+$. In fact we do not need
connectedness of $\mc G^\circ$ for that, as we will check now. The conjugation
action of $N_{\mc G}(\mc T)_{\mc L}^+$ on $\mc G^\circ \epsilon$ provides an action on
$\underline{CS}_{W^\circ \mc L}(\mc G^\circ ; \epsilon)$, by the functors $\mr{Ad}(n)_*$.
The action of $N_{\mc G^\circ}(\mc T)_{\mc L}^+$ on 
$Z( {}_{\mc L} \underline{\epsilon \mc L} {}_{\mc L} )$ from \eqref{eq:3.11}
extends naturally to an action of $N_{\mc G}(\mc T)_{\mc L}^+$. Conjugation by
$N_{\mc G}(\mc T)_{\mc L}^+$ stabilizes
\[
\bigcup\nolimits_{\beta \in \mf B_{\mc L}^\circ} N_{\mc H}(\mc T)
\dot w^\beta \epsilon = N_{\mc G^\circ \epsilon}(\mc T)_{\mc L}.
\]
Hence \eqref{eq:3.11} also defines an action of $N_{\mc G}(\mc T)_{\mc L}^+$ on 
$\bigoplus_{\beta \in \mf B_{\mc L}^\circ}
{}_{\mc L} \underline{\dot w^\beta \epsilon S}^\circ_{\mc L}$ and on the twisted
categorical centre of that category. The actions of $N_{\mc G}(\mc T)_{\mc L}^+$ on
${}_{\mc L} \underline{S}^{\mc H}_{\mc L}$ and $\bigoplus_{\beta \in \mf B_{\mc L}^\circ}
{}_{\mc L} \underline{\sigma_{\beta \epsilon} S}^{\mc H}_{\mc L}$ constructed in
\eqref{eq:3.46} can be defined equally well for $N_{\mc G}(\mc T)_{\mc L}^+$. 
As in Theorem \ref{thm:3.3}.c, this also gives a $N_{\mc G}(\mc T)_{\mc L}^+$-action
on $\bigoplus_{\beta \in \mf B_{\mc L}^\circ} Z(
{}_{\mc L} \underline{\sigma_{\beta \epsilon} S}^{\mc H}_{\mc L})$. Inside the group
$N_{\mc G}(\mc T)_{\mc L}^+ \overset{\mc T}{\times} \mc H$, $N_{\mc G}(\mc T)_{\mc L}^+$
acts by conjugation on $\bigcup_{\beta \in \mf B_{\mc L}^\circ} \mc H \sigma_{\beta 
\epsilon}$, lifting the action $\mr{Ad}_\epsilon$ on $\mf B_{\mc L}^\circ$. With that
we can let $N_{\mc G}(\mc T)_{\mc L}^+$ act on $\bigoplus_{\beta \in \mf B_{\mc L}^\circ}
\underline{CS}_{\mc L} (\mc H; \sigma_{\beta \epsilon})$, by the functors $\mr{Ad}(n)_*$.

For each of the above categories $\bigoplus_{\beta \in \mf B_{\mc L}^\circ} \mc C_\beta$,
there is a canonical way to extend the $N_{\mc G}(\mc T)_{\mc L}^+$-action to
$\big( \bigoplus_{\beta \in \mf B_{\mc L}^\circ} \mc C_\beta \big)^{\Omega_{\mc L}^\circ}$.
Given  $(X, \alpha_X ) \in \big( \bigoplus_{\beta \in \mf B_{\mc L}^\circ} \mc C_\beta 
\big)^{\Omega_{\mc L}^\circ}$, we extend the $\alpha_{\omega,X}$ (for a set of 
representatives $\dot \omega$ of $\Omega_{\mc L}^\circ$ in 
$N_{\mc G^\circ}(\mc T)_{\mc L}^+$) to $\alpha_{n,X}$ for all $n \in N_{\mc G^\circ}
(\mc T)_{\mc L}^+$, as in \eqref{eq:A.13}. Then we define, like in \eqref{eq:A.11}
\begin{equation}\label{eq:4.2}
F_\gamma^{\Omega_{\mc L}^\circ} (X, \alpha_X) = \big( F_\gamma (X), F_\gamma 
(\alpha_{\gamma^{-1} \dot \omega \gamma} )_{\omega \in \Omega_{\mc L}^\circ} \big) 
\qquad \gamma \in N_{\mc G}(\mc T)_{\mc L}^+ .
\end{equation}
Notice that we need the normality of $N_{\mc G^\circ}(\mc T)_{\mc L}^+$ in
$N_{\mc G}(\mc T)_{\mc L}^+$.

\begin{thm}\label{thm:4.4}
\enuma{
\item The equivalences of categories in Theorem \ref{thm:3.5}.a for $\mc G^\circ$ are\\
$N_{\mc G}(\mc T)_{\mc L}^+$-equivariant, with respect to the above actions.
\item Theorem \ref{thm:3.5}.a and part (a) give canonical $N_{\mc G}(\mc T)_{\mc L}^+
$-equivariant equivalences of categories
\begin{align*}
& \underline{CS}_{W^\circ \mc L} (\mc G^\circ; 
\epsilon)^{\Omega_{\mc L} / \Omega_{\mc L}^\circ} \isom Z^+ ({}_{\mc L} 
\underline{\epsilon S}_{\mc L})^{\Omega_{\mc L} / \Omega_{\mc L}^\circ} \isom 
\big( \bigoplus\nolimits_{\beta \in \mf B_{\mc L}^\circ} Z ({}_{\mc L} 
\underline{\dot w^\beta \epsilon S}^{\circ}_{\mc L}) \big)^{\Omega_{\mc L}} \\
& \mosi \big( \bigoplus\nolimits_{\beta \in \mf B_{\mc L}^\circ} Z ( {}_{\mc L}
\underline{\sigma_{\beta \epsilon} S}^{\mc H}_{\mc L}) \big)^{\Omega_{\mc L}}
\mosi \big( \bigoplus\nolimits_{\beta \in \mf B_{\mc L}^\circ} 
\underline{CS}_{\mc L} (\mc H; \sigma_{\beta \epsilon}) \big)^{\Omega_{\mc L}} .
\end{align*}
}
\end{thm}
\begin{proof}
(a) For the first equivalence, this can be shown in the same way as in Proposition
\ref{prop:3.4}.b. For the second equivalence it follows from \eqref{eq:3.29}. For
the third equi\-va\-lence the argument from Theorem \ref{thm:3.3}.b applies, taking into
account that Lemma \ref{lem:3.2} also holds when $\mc G$ is disconnected and
$n \in N_{\mc G^\circ}(\mc T)_{\mc L} \cup N_{\mc G^\circ \epsilon}(\mc T)_{\mc L}$.
For the fourth equivalence we may use the argument from the first equivalence, and
extend to the $\Omega_{\mc L}^\circ$-equivariant structures with \eqref{eq:4.2}.\\
(b) By Theorem \ref{thm:3.5}.a, part (a) and \eqref{eq:4.1}, the equivalences in 
Theorem \ref{thm:3.5}.a remain valid when we enrich all five categories with
$\Omega_{\mc L} / \Omega_{\mc L}^\circ$-equivariant structures. It remains to see that 
for the three categories of the form $\mc C = \bigoplus_{\beta \in \mf B_{\mc L}^\circ} 
\mc C_\beta$ in Theorem \ref{thm:3.5}.a there are natural equivalences
\[
\big( \mc C^{\Omega_{\mc L}^\circ} \big)^{\Omega_{\mc L} / \Omega_{\mc L}^\circ} \cong
\mc C^{\Omega_{\mc L}} .
\]
Those equivalences are instances of Lemma \ref{lem:A.3}.
\end{proof}

Next we investigate the equivalences between categories of character sheaves and of
representations for $\mc G^\epsilon$, like in Lemma \ref{lem:3.1}. Recall the groupoid
$\tilde{\mf B}_{\mc L}$ from \eqref{eq:E.39} and \eqref{eq:E.35}. It can be defined
equally well when $\mc G$ is disconnected, and provides the notion of
$\Omega_{\mc L}$-equivariant structures on $\bigoplus_{\beta \in \mf B_{\mc L}^\circ}
\Rep_{\mc L} (H^{\sigma_{\beta \epsilon}})$.

\begin{prop}\label{prop:4.5}
\enuma{
\item There is a canonical equivalence of categories
\[
\underline{CS}_{W^\circ \mc L} (\mc G^\circ; \epsilon)^{\Omega_{\mc L} / \Omega_{\mc L
}^\circ} \cong \Rep_{W^\circ \mc L} (G^{\circ \epsilon})^{\pi_0 (\mc G)_{W^\circ \mc L}} .
\]
\item There is a canonical equivalence of categories
\[
\big( \bigoplus\nolimits_{\beta \in \mf B_{\mc L}^\circ} 
\underline{CS}_{\mc L} (\mc H; \sigma_{\beta \epsilon}) \big)^{\Omega_{\mc L}} \cong
\big( \bigoplus\nolimits_{\beta \in \mf B_{\mc L}^\circ} 
\Rep_{\mc L} (H^{\sigma_{\beta \epsilon}}) \big)^{\Omega_{\mc L}} .
\]
}
\end{prop}
\begin{proof}
(a) As already observed in the proof of Lemma \ref{lem:3.1}.b, the equivalence of 
categories
\begin{equation}\label{eq:4.6}
\underline{CS}_{W^\circ \mc L} (\mc G^\circ; \epsilon) \cong
\Rep_{W^\circ \mc L} (G^{\circ \epsilon}) 
\end{equation}
from Lemma \ref{lem:3.1}.a is $N_{G^\epsilon}(T^\epsilon )_{W^\circ \mc L}$-equivariant
and respects the isomorphisms
\[
\phi_n^c : \mr{Ad}(n)_* \to \mr{id} \quad \text{with } n \in N_{G^{\circ \epsilon}}
(T^\epsilon )_{W^\circ \mc L} = N_{G^{\circ \epsilon}}(T^\epsilon ) .
\]
In view of \eqref{eq:E.5}, the equivalence \eqref{eq:4.6} can be enriched to a
canonical equivalence
\begin{align*}
\underline{CS}_{W^\circ \mc L} (\mc G^\circ; \epsilon)^{\Omega_{\mc L} / \Omega_{\mc L}^\circ}
& = \underline{CS}_{W^\circ \mc L} (\mc G^\circ; \epsilon)^{N_{G^\epsilon}(T^\epsilon 
)_{W^\circ \mc L} / N_{G^{\circ \epsilon}}(T^\epsilon )} \\
& \cong \Rep_{W^\circ \mc L} (G^{\circ \epsilon}) ^{N_{G^\epsilon}(T^\epsilon 
)_{W^\circ \mc L} / N_{G^{\circ \epsilon}}(T^\epsilon )} = 
\Rep_{W^\circ \mc L} (G^{\circ \epsilon})^{\pi_0 (\mc G)_{W^\circ \mc L}} .
\end{align*}
(b) Lemma \ref{lem:A.1} gives a canonical equivalence 
\begin{equation}\label{eq:4.7}
\big( \bigoplus\nolimits_{\beta \in \mf B_{\mc L}^\circ} 
\underline{CS}_{\mc L} (\mc H; \sigma_{\beta \epsilon}) \big)^{\Omega_{\mc L}} \cong
\bigoplus\nolimits_{\beta \in \mf B_{\mc L}^\circ / \mr{Ad}_\epsilon (\Omega_{\mc L})} 
\underline{CS}_{\mc L} (\mc H; \sigma_{\beta \epsilon}) \big)^{\Omega_{\mc L,\beta}} .
\end{equation}
From \eqref{eq:E.39} and \eqref{eq:3.19} we know that 
\[
\Omega_{\mc L,\beta} = N_G (T)_{\mc L}^{\beta, \sigma_{\beta \epsilon}} /
T^{\sigma_{\beta \epsilon}}
\]
and that $\Omega_{\mc L,\beta}$ can be represented by elements $n(\omega,\beta)$ that
commute with $\sigma_{\beta \epsilon}$. Lemma \ref{lem:3.1}.b provides a canonical
equivalence 
\begin{equation}\label{eq:4.8}
\underline{CS}_{\mc L} (\mc H; \sigma_{\beta \epsilon})^{\Omega_{\mc L, \beta}} \cong
\Rep_{\mc L} (H^{\sigma_{\beta \epsilon}})^{\Omega_{\mc L, \beta}} . 
\end{equation}
We go back to $\Omega_{\mc L}$-equivariant structures with Lemma \ref{lem:A.1}.a:
\begin{equation}\label{eq:4.9}
\bigoplus\nolimits_{\beta \in \mf B_{\mc L}^\circ / \mr{Ad}_\epsilon (\Omega_{\mc L})} 
\Rep_{\mc L} (H^{\sigma_{\beta \epsilon}})^{\Omega_{\mc L,\beta}} \cong
\big( \bigoplus\nolimits_{\beta \in \mf B_{\mc L}^\circ} 
\Rep_{\mc L} (H^{\sigma_{\beta \epsilon}}) \big)^{\Omega_{\mc L}} .
\end{equation} 
Finally we combine \eqref{eq:4.7}, \eqref{eq:4.8} and \eqref{eq:4.9}.
\end{proof}

We come to our main result about disconnected reductive groups. Like in \eqref{eq:3.52}
we define 
\[
H^\Omega_{\mc L}(\beta,\beta) = N_G (T)_{\mc L}^{\beta,\sigma_{\beta \epsilon}}
\overset{T^{\sigma_{\beta \epsilon}}}{\times} H^{\sigma_{\beta \epsilon}} ,
\]
this works just as well when $\mc G$ is disconnected.

\begin{thm}\label{thm:E.4}
Suppose that $\mc G^\epsilon$ satisfies Condition \ref{cond:E.3}, and let $\mc L$ be a 
rank one character sheaf on $\mc T$ such that $W^\circ \mc L$ 
is $\epsilon$-stable. There exists canonical equivalences of categories
\[
\Rep_{W \mc L} (G^\epsilon) \cong \big( \bigoplus\nolimits_{\beta \in \mf B_{\mc L}^\circ} 
\Rep_{\mc L} ( H^{\sigma_{\beta \epsilon}} ) \big)^{\Omega_{\mc L}} \cong
\bigoplus\nolimits_{\beta \in \mf B_{\mc L}^\circ / \mr{Ad}_\epsilon (\Omega_{\mc L})} 
\Rep_{\mc L} \big( H^\Omega_{\mc L}(\beta,\beta) \big)
\]
\end{thm}
\begin{proof}
The first equivalence follows from Theorem \ref{thm:4.4}.b, Proposition \ref{prop:4.5} 
and \eqref{eq:E.4}. Then we are on the right hand side of \eqref{eq:4.9}, and we can
pass to the left side. To that we apply the equivalences of categories
\begin{equation}\label{eq:4.10}
\Rep_{\mc L} (H^{\sigma_{\beta \epsilon}})^{\Omega_{\mc L,\beta}} \cong 
\Rep_{\mc L} (H^\Omega_{\mc L}(\beta,\beta)) 
\end{equation}
from Lemma \ref{lem:A.3}, for each 
$\beta \in \mf B_{\mc L}^\circ / \mr{Ad}_\epsilon (\Omega_{\mc L})$.
\end{proof}

The following example shows that for disconnected reductive $\F_q$-groups it does not
suffice to use unipotent representations of (a form of) the endoscopic group, we have
to involve essentially unipotent representations.

\begin{ex}\label{ex:4.6}
Let $q$ be an odd prime number and take $G^{\circ \epsilon} = \F_q^\times$. Let 
$Q_8$ be the quaternion group of order 8, with centre $Z(Q_8) = \{ \pm 1\}$. Consider
the $\F_q$-group $\mc G = Q_8 \overset{\{\pm 1\}}{\times} \mc G^\circ$, which has 
component group $\pi_0 (\mc G) = Q_8 / Z(Q_8) \cong (\Z / 2 \Z)^2$. Then $G$ is a central
extension of $\pi_0 (\mc G)$ by $\F_q^\times$. Let $\mc L$ be
a rank one character sheaf on the split torus $\mc G^\circ$, such that the corresponding
character $\chi_{\mc L}$ of $G^{\circ \epsilon}$ satisfies $\chi_{\mc L}(-1) = -1$. 
This is possible because $q$ is odd, and $\Omega_{\mc L} = W_{\mc L} = \pi_0 (\mc G)$.

The category $\Rep_{\mc L}(G^\epsilon)$ has just one irreducible object,
namely the unique twodimensional irreducible representation $\pi$ of $Q_8$ extended to
$G^\epsilon$ as $\chi_{\mc L}$ on $G^{\circ \epsilon}$. Notice that this makes sense 
because $\chi_{\mc L}(-1) = -1 = \pi (-1)$. The endoscopic group is simply 
$\mc H = \mc G^\circ$. Since $H_{\mc L}^\Omega (\beta,\beta) = G^\epsilon$, the category 
$\Rep_{\mc L} (H^\epsilon)^{\Omega_{\mc L}} \cong \Rep_{\mc L}(G^\epsilon)$ has again
just one irreducible object. On the other hand
\[
\Rep_1 (H^\epsilon )^{\Omega_{\mc L}} \cong \Rep_1 (G^\epsilon) = \Rep (G^\epsilon /
H^\epsilon) = \Rep (\pi_0 (\mc G)) ,
\]
which has four irreducible objects. In particular $\Rep_{\mc L}(G^\epsilon)$ is not
equivalent to $\Rep_1 (H^\epsilon )^{\Omega_{\mc L}}$, in contrast to Theorems
\ref{thm:E.1} and \ref{thm:E.7} for connected $\mc G$.
\end{ex}

We turn to rational series for $G^\epsilon$. First we have to define them properly.
By \eqref{eq:E.23}, we can parametrize the rational series in $\Rep (G^{\circ \epsilon})$
by $W^\circ$-orbits of pairs $(\mc L,\beta)$. The groups 
$\pi_0 (\mc G) \cong N_{G^\epsilon}(T^\epsilon,B^\epsilon)/T^\epsilon$
and $W(\mc G,\mc T) = W$ and act naturally on the set of such parameters. Then the 
definition of $\Rep_{W^\circ (\mc L,\beta)}(G^{\circ \epsilon})$ in terms of generators 
obtained by Deligne--Lusztig induction shows that $\pi_0 (\mc G)$ and $W$ and permute the
various rational series according to the actions on their parameters. 
It follows that the category
\[
\Rep_{W(\mc L,\beta)} (G^{\circ \epsilon}) := \bigoplus\nolimits_{\gamma \in W^\circ 
\backslash W / W_{\mc L,\beta}} \Rep_{W^\circ (\gamma \mc L, \gamma \beta \epsilon
 (\gamma)^{-1})} (G^{\circ \epsilon})
\]
is stable under the action of $G^\epsilon$ on $\Rep (G^{\circ \epsilon})$. 
We define $\Rep_{W(\mc L,\beta)} (G^\epsilon)$ as the ca\-te\-gory of 
$G^\epsilon$-representations whose restriction to $G^{\circ \epsilon}$ lies in 
$\Rep_{W(\mc L,\beta)} (G^{\circ \epsilon})$. When $W^\circ (\mc L,\beta)$ corresponds to 
$\mr{Ad}(G^{\circ \vee \epsilon}) s_\beta$, we write 
\[
\Rep_{W(\mc L,\beta)} (G^\epsilon) = \Rep_{s_\beta}(G^\epsilon). 
\]
We call this the rational series of $G^\epsilon$ associated to $s_\beta$ or to $W(\mc L,\beta)$. 
It actually corresponds to a set of semisimple elements in $G^{\circ \vee \epsilon}$ larger than 
the rational conjugacy class of $s_\beta$, but it is difficult to express that precisely because 
in general we do not have a group dual to $\mc G$.

As the analysis of $\Rep_{W(\mc L,\beta)} (G^\epsilon)$ proceeds like for $G^{\circ \epsilon}$, 
we will carry it out more briefly. By Lemma \ref{lem:A.2} we have, like in \eqref{eq:E.7} 
and \eqref{eq:E.4}:
\begin{equation}\label{eq:E.25}
\Rep_{W (\mc L,\beta)}(G^\epsilon) \cong \Rep_{W (\mc L,\beta)}(G^{\circ \epsilon}
)^{\pi_0 (\mc G)} \cong 
\Rep_{W^\circ (\mc L,\beta)}(G^{\circ \epsilon})^{\pi_0 (\mc G)_{W^\circ (\mc L,\beta)}} . 
\end{equation}
There are isomorphisms
\begin{equation}\label{eq:E.27}
\pi_0 (\mc G)_{W^\circ (\mc L,\beta)} \cong W_{\mc L,\beta} / (W^\circ)_{\mc L,\beta} =
(W_{\mc L}^\circ \rtimes \Omega_{\mc L,\beta}) / (W_{\mc L}^\circ \rtimes 
\Omega^\circ_{\mc L,\beta}) \cong \Omega_{\mc L, \beta} / \Omega_{\mc L,\beta}^\circ .
\end{equation}

\begin{thm}\label{thm:E.8}
Suppose that $\mc G^\epsilon$ satisfies Condition \ref{cond:E.3}. Let $\mc L$ be an 
$\epsilon$-stable rank one character sheaf on $\mc T$ and suppose that $W^\circ (\mc L,\beta)$ 
corresponds to $s_\beta \in G^{\circ \vee \epsilon}$ via \eqref{eq:E.23}. There exists 
canonical equivalences of categories 
\[
\Rep_{s_\beta} (G^\epsilon) = \Rep_{W(\mc L,\beta)} (G^\epsilon) \cong 
\Rep_{s_{\mc L}} (H^{\sigma_{\beta \epsilon}} )^{\Omega_{\mc L, \beta}} 
\cong \Rep_{\mc L} \big( H^\Omega_{\mc L}(\beta,\beta) \big) .
\] 
\end{thm}
\begin{proof}
By \eqref{eq:E.25} and Theorem \ref{thm:3.9}, there are canonical equivalences of categories
\[
\Rep_{W (\mc L,\beta)} (G^\epsilon) \cong \Rep_{W^\circ (\mc L,\beta)} 
(G^{\circ \epsilon})^{\pi_0 (\mc G)_{W^\circ (\mc L,\beta)}} \cong 
\Big( \Rep_{\mc L} (H^{\sigma_{\beta \epsilon}})^{\Omega_{\mc L,\beta}^\circ} 
\Big)^{\pi_0 (\mc G)_{W^\circ (\mc L,\beta)}} .
\]
In the rightmost term, the stacked equivariant structures combine to one 
$\Omega_{\mc L,\beta}$-equivariant structure, by \eqref{eq:E.27} and Lemma \ref{lem:A.3}.
Finally, we apply \eqref{eq:4.10}.
\end{proof}

\appendix

\section{Equivariant structures for non-strict actions} \label{sec:A}

In this paper arguments with equivariant structures on objects in categories play
an important role. Since this notion is rather subtle, we spell it out in this appendix.

Consider a category $\mc C$ and a group $G$ acting on $\mc C$, by means of functors
$F_g$ for $g \in G$. Suppose first that the action is strict, so 
$F_e = \mr{id}_{\mc C}$ and
\begin{equation}\label{eq:A.1}
F_{g_1 g_2} = F_{g_1} \circ F_{g_2} \quad \text{for all } g_1, g_2 \in G.
\end{equation}
Then a $G$-equivariant structure $\alpha_X$ on an object $X$ of $\mc C$ consists of
an isomorphism $\alpha_{g,X} : F_g (X) \to X$ for every $g \in G$, such that 
\begin{equation}\label{eq:A.2}
\alpha_{g_1 g_2,X} = \alpha_{g_1,X} \circ F_{g_1} (\alpha_{g_2,X}) : 
F_{g_1 g_2}(X) \to X \quad \text{for all } g_1, g_2 \in G .
\end{equation}
Next we relax the conditions and we only assume that $G$ acts weakly on $\mc C$.
Then $F_e = \mr{id}_{\mc C}$ and for all $g_1, g_2 \in G$ we are given a natural
isomorphism 
\[
\phi_{g_1,g_2} : F_{g_1 g_2} \to F_{g_1} \circ F_{g_2} ,
\]
satisfying certain conditions for triples $g_1, g_2, g_3$ in $G$. Now a $G$-equivariant
structure on an object $X$ of $\mc C$ consists of
an isomorphism $\alpha_{g,X} : F_g (X) \to X$ for every $g \in G$, such that 
\begin{equation}\label{eq:A.3}
\alpha_{g_1 g_2} = \alpha_{g_1,X} \circ F_{g_1} (\alpha_{g_2,X}) \circ \phi_{g_1,g_2,X} 
: F_{g_1 g_2}(X) \to X \quad \text{for all } g_1, g_2 \in G .
\end{equation}
The category of objects of $\mc C$ enriched with $G$-equivariant structures is denoted
$\mc C^G$. The morphisms between $X$ and $Y$ in $\mc C^G$ are the 
$f \in \Hom_{\mc C}(X,Y)$ such that 
\begin{equation}\label{eq:A.9}
f \circ \alpha_{g,X} = \alpha_{g,Y} \circ F_g (f) : F_g (X) \to Y 
\quad \text{for all } g \in G . 
\end{equation}

\subsection{Equivariant structures for quotient groups} \
\label{par:A.1}

In typical situations (which we encounter in this paper) $G$ is a quotient group
$\Gamma / N$, such that $\Gamma$ acts strictly on $\mc C$ and its normal subgroup
$N$ acts by isomorphisms. Suppose further that for each $n \in N$ we are given a
natural isomorphism $\phi_n^c : F_n \to \mr{id}_{\mc C}$, such that
\begin{equation}\label{eq:A.8}
\phi^c_{n_1 n_2} = \phi^c_{n_1} \circ F_{n_1} (\phi^c_{n_2}) \quad
\text{for all } n_1, n_2 \in N.
\end{equation}
Then we can define a weak action of $G$ on $\mc C$, as follows. For each $g \in G$
we pick a representative $\dot g \in \Gamma$, with $\dot e = e$. Then we put 
$F_g = F_{\dot g}$ for all $g \in G$, and 
\begin{equation}\label{eq:A.7}
\phi_{g_1,g_2} = F_{\dot{g_1} \dot{g_2}} \big( \phi^c_{(\dot{g_1} \dot{g_2})^{-1}
\dot{g_1 g_2}} \big) : F_{\dot{g_1 g_2}} \to F_{\dot{g_1}} \circ F_{\dot{g_2}} = 
F_{\dot{g_1} \dot{g_2}} .
\end{equation}
More explicitly, if $\dot{g_1 g_2} = \dot{g_1} \dot{g_2} n$ with $n \in N$, then
\begin{equation}\label{eq:A.4}
\phi_{g_1, g_2, X} = F_{\dot{g_1} \dot{g_2}} (\phi^c_{n,X}) :
F_{\dot{g_1 g_2}}(X) \to F_{\dot{g_1}}(F_{\dot{g_2}} (X)) .
\end{equation}
In this setting a $G$-equivariant structure $\alpha_X$ on an object $X$ of $\mc C$ consists 
of an isomorphism $\alpha_{g,X} : F_{\dot g}(X) \to X$ for every $g \in G$, such that
\begin{equation}\label{eq:A.12}
\alpha_{g_1 g_2,X} = \alpha_{g_1,X} \circ F_{\dot{g_1}} (\alpha_{g_2,X}) \circ F_{\dot{g_1}
\dot{g_2}} (\phi^c_{n,X}) \quad \text{if } \dot{g_1 g_2} = \dot{g_1} \dot{g_2} n .
\end{equation}
Again, this yields a category $\mc C^G$ of objects of $\mc C$ enriched with
$G$-equivariant structures. We can extend $\alpha_X$ to a definition of $\alpha_{\gamma,X}$ 
for any $\gamma \in \Gamma$, by the formula
\begin{equation}\label{eq:A.13}
\alpha_{\dot g n, X} = \alpha_{g,X} \circ F_{\dot g}( \phi_{n,X}^c ) \qquad g \in G, n \in N.
\end{equation}
Then \eqref{eq:A.12} becomes the simpler property
\begin{equation}\label{eq:A.14}
\alpha_{\gamma_1 \gamma_2,X} = \alpha_{\gamma_1,X} \circ F_{\gamma_1}(\alpha_{\gamma_2,X}) 
\qquad \gamma_1, \gamma_2 \in \Gamma .
\end{equation}
We note that the group $\Gamma$ acts strictly on $\mc C^G$, on $\mc C$ as before and on
equivariant structures by
\begin{equation}\label{eq:A.10}
\begin{aligned}
& F_\gamma^G ( \alpha_{g,X} ) := F_\gamma (\alpha_{\gamma^{-1} \dot g \gamma,X}) : 
F_{\dot g}(F_\gamma (X)) \to F_\gamma (X) ,\\
& F_\gamma^G \big( X, \alpha_X ) = (F_\gamma (X),   
F_\gamma (\alpha_{\gamma^{-1} \dot g \gamma,X} )_{g \in G} \big) 
\end{aligned}
\end{equation}
Moreover there are natural isomorphisms $\phi_\gamma^{G,c} :
F_\gamma^G \to \mr{id}_{\mc C^G}$, given by
\begin{equation}\label{eq:A.11}
\phi_{\gamma, (X,\alpha_X)}^{G,c} = \alpha_{\gamma,X} : 
F_\gamma^G (X,\alpha_X) \to (X,\alpha_X) .
\end{equation}
One can check that this is consistent with \eqref{eq:A.9}.

Let $\Gamma^\circ$ be a normal subgroup of $\Gamma$ containing $N$. Then
$G^\circ = \Gamma^\circ / N$ is a normal subgroup of $G$. There is a transitivity
property for stacking equivariant structures:

\begin{lem}\label{lem:A.3}
There is a natural equivalence of categories
$\big( \mc C^{G^\circ} \big)^{G / G^\circ} \cong \mc C^G$.
\end{lem}
\begin{proof}
First we rewrite the statement as
\begin{equation}\label{eq:4.4}
\big( \mc C^{\Gamma^\circ / N} \big)^{\Gamma / \Gamma^\circ} \cong \mc C^{\Gamma / N} ,
\end{equation}
where the action of $\Gamma$ is strict. By \eqref{eq:A.13}, a $\Gamma / 
\Gamma^\circ$-equivariant structure $\alpha'_X$ on $(X,\alpha_X) \in \mc C^{\Gamma^\circ / N}
= \mc C^{G^\circ}$ means that for each $\gamma \in \Gamma$ we have an isomorphism
\begin{equation}\label{eq:4.5}
\alpha'_{\gamma, (X,\alpha_X)} : F_\gamma^{G^\circ} (X, \alpha_X) \to (X, \alpha_X)
\end{equation}
satisfying \eqref{eq:A.14}, such that $\alpha'_{n,(X,\alpha_X)} = \alpha_{n,X}$ for
$n \in \Gamma^\circ$. By forgetting $\alpha_X$ in \eqref{eq:4.5}, we get $(X,\alpha'_X) 
\in \mc C^{\Gamma / N} =  \mc C^{G}$. This defines a functor from left to
right in \eqref{eq:4.4}.

Conversely, suppose we are given $(X,\alpha'_X) \in \mc C^{\Gamma / N}$. By restricting 
$\alpha'_X$ to $\alpha'_{n,X}$ with $n \in \Gamma^\circ$, we obtain 
$(X,\alpha_X) \in \mc C^{G^\circ}$. Then the isomorphisms
\[
\alpha'_{\gamma,X} : F_\gamma (X) \isom X \qquad \text{for } \gamma \in \Gamma
\]
extend to morphisms in $\mc C^{\Gamma / N}$ because
\[
\alpha'_{\gamma,} \circ F_\gamma (\alpha_{\gamma^{-1} n \gamma, X}) =
\alpha'_{n \gamma, X} = \alpha_{n,X} \circ F_n (\alpha'_{\gamma,X})
\qquad \text{for } n \in \Gamma^\circ .
\]
This defines a functor from the right to the left in \eqref{eq:4.4}.
Clearly these two functors are natural and mutually inverse, so 
\eqref{eq:4.4} is indeed a natural equivalence.
\end{proof}

Next we investigate a kind of Morita equivalence for equivariant
structures. We suppose that $\mc C$ is a direct sum of subcategories:
\[
\mc C = \bigoplus\nolimits_{\beta \in \mf B} \mc C_\beta . 
\]
We also suppose that a weak action of
$G$ on $\mc C$ permutes the $\mc C_\beta$ transitively, lifting a strict action of
$G$ on the index set $\mf B$. Let $G_\beta$ be the $G$-stabilizer of $\beta$.

\begin{lem}\label{lem:A.1}
\enuma{
\item The restriction functor $\mr{res}_\beta : \mc C^G \to \mc C_\beta^{G_\beta}$
is an equivalence.
\item Suppose that $G = \Gamma / N$ and that the weak action of $G$ comes from
a strict action of $\Gamma$ on $\mc C$, as in \eqref{eq:A.8}. Then $\mr{res}_\beta$ is
$\Gamma_\beta$-equivariant and sends the isomorphism $\phi_\gamma^{G,c} : 
F_\gamma^G \to \mr{id}_{\mc C^G}$ to $\phi_\gamma^{G_\beta,c} : 
F_\gamma^{G_\beta} \to \mr{id}_{\mc C_\beta^{G_\beta}}$.
}
\end{lem}
\begin{proof}
We construct an inverse to $\mr{res}_\beta$. We pick a set of representatives 
$[G / G_\beta]$ for $G / G_\beta$ in $G$, containing $e$. Then we define
\begin{align*}
& \ind_\beta : \mc C_\beta^{G_\beta} \to \mc C^G ,\\
& \ind_\beta (Y,\alpha_Y) = \big( \bigoplus\nolimits_{h \in [G / G_\beta]} F_h (Y),
\bigoplus\nolimits_{h \in [G / G_\beta]} \alpha_Y^h \big),
\end{align*}
where $\alpha_{g,Y}^h$ is the composition
\[
F_g (F_h (Y)) \xrightarrow{\phi_{g,h,Y}^{-1}} F_{gh}(Y) \xrightarrow{\phi_{g h g_\beta^{-1},
g_\beta,Y}} F_{g h g_\beta^{-1}} (F_{g_\beta}(Y)) \xrightarrow{F_{gh g_\beta^{-1}}
(\alpha_{g_\beta,Y})} F_{g h g_\beta^{-1}} (Y) .
\]
Here $g_\beta \in G_\beta$ is determined by the condition 
$g h g_\beta^{-1} \in [G / G_\beta]$. We note that
\[
\mr{res}_\beta \ind_\beta (Y, \alpha_Y) = (F_e (Y), \alpha^e_Y) = (Y,\alpha_Y) .
\]
Let $(X,\alpha_X) = \big( \bigoplus_{\beta'} X_{\beta'}, \bigoplus_{\beta'} 
\alpha^{\beta'} \big) \in \mc C^G$, where $X_{\beta'} \in \mc C_{\beta'}$ and 
\[
\alpha^{\beta'}_g : F_g (X_{\beta'}) \isom X_{g \cdot \beta'} .
\]
We claim that there is a natural isomorphism
\[
\ind_\beta \mr{res}_\beta (X, \alpha_X) = \ind_\beta (X_\beta, \alpha^\beta |_{G_\beta})
\isom (X,\alpha_X) .
\]
On the summand $F_h (X_\beta)$ of $\ind_\beta (X_\beta)$ it is given by 
$\alpha_h^\beta : F_h (X_\beta) \isom X_{h \cdot \beta}$. We check that \eqref{eq:A.9}
holds on the summand $F_g (F_h (X_\beta))$ of $F_g (\ind_\beta (X_\beta))$:
\begin{align*}
\alpha_{gh g_\beta^{-1}}^\beta \circ \alpha_{g,X_\beta}^h & = 
\alpha_{gh g_\beta^{-1}}^\beta \circ F_{g h g_\beta^{-1}} (\alpha_{g_\beta,X}^\beta)
\circ \phi_{g h g_\beta^{-1},g_\beta,X_\beta} \circ \phi_{g,h,X_\beta}^{-1} \\
& = \alpha_{gh}^\beta \circ \phi_{g,h,X_\beta}^{-1} = 
\alpha_g^{h\cdot \beta} \circ F_g (\alpha_h^\beta) .
\end{align*}
Hence $\bigoplus_{h \in [G/G_\beta]} \alpha_h^\beta$ is indeed an isomorphism from
$\ind_\beta (X_\beta, \alpha^\beta |_{G_\beta})$ to $(X,\alpha_X)$ in $\mc C^G$. It
is natural because the construction uses only the data from $(X,\alpha_X)$.\\
(b) The existence of $\phi_n^c : F_n \isom \mr{id}_{\mc C}$ for $n \in N$ implies 
that $N$ stabilizes each subcategory $\mc C_{\beta'}$, and in particular 
$G_\beta = \Gamma_\beta / N$. The $\Gamma_\beta$-equivariance of $\mr{res}_\beta$ is 
clear from \eqref{eq:A.10}, and the claim about $\phi_\gamma^{G,c}$ and 
$\phi_\gamma^{G_\beta,c}$ follows from \eqref{eq:A.11}.
\end{proof}

\subsection{Equivariant structures on categories of representations} \
\label{par:A.2}

Consider the category $\Rep (H)$ of representations of an abstract 
group $H$, say on vector spaces over a chosen coefficient field. Let $\Gamma$ be
a group acting on $H$ by group automorphisms, and suppose that $N = \Gamma \cap H$
is a normal subgroup of $\Gamma$, whose action on $H$ is simply by conjugation. Then
$\Gamma$ acts strictly on $\Rep (H)$ by
\begin{align*}
& F_\gamma (\pi) = \gamma_* \pi : h \mapsto \pi (\gamma^{-1}(h)) \quad
\gamma \in \Gamma, \pi \in \Rep (H), h \in H ,\\
& F_n (\pi) = \mr{Ad}(n)_* \pi : h \mapsto \pi (n^{-1} h n) \quad  n \in N .
\end{align*}
There is a natural choice for $\phi^c_n$, namely
\begin{equation}\label{eq:A.5}
\phi^c_{n,\pi} = \pi (n) : \mr{Ad}(n)_* \pi \to \pi .
\end{equation}
As in Paragraph \ref{par:A.1}, any choice of representatives $\dot g$ for 
$G = \Gamma / N$ in $\Gamma$ determines a weak action of $G$ on $\Rep (H)$. 

We specialize to the case $H = N$, which is then a normal subgroup of $\Gamma$.
Let us write a $G$-equivariant structure on $\pi \in \Rep (N)$ in terms of morphisms
\[
\pi^G (\dot g) : \mr{Ad}(\dot g)_* \pi \to \pi \qquad g \in G.
\]
The multiplicativity axiom for an equivariant structure becomes
\begin{equation}\label{eq:A.6}
\pi^G (\dot g_3) = \pi^G (\dot g_1) \pi^G (\dot g_2) 
\pi ( (\dot g_1 \dot g_2)^{-1} \dot g_3) 
\end{equation}
whenever $g_1 g_2 = g_3$ in $G = \Gamma / N$.

\begin{lem}\label{lem:A.2}
The category $\Rep (N)^G$ of $N$-representations enriched with a $G$-equivariant
structure is naturally equivalent to $\Rep (\Gamma)$.
\end{lem}
\begin{proof}
For any $\pi^G \in \Rep (\Gamma)$, write $\pi = \Res^\Gamma_N \pi^G$. We obtain a 
$G$-equivariant structure on $\pi$ by defining the required morphism 
$\mr{Ad}(\dot g)_* \pi \to \pi$ as $\pi^G (\dot g)$, 
for any representative $\dot g$ of $g \in G$.

Conversely, suppose that $\pi \in \Rep (N)$ has a $G$-equivariant structure. Then we have 
a morphism of $N$-representations $\pi^G (\dot g) : \mr{Ad}(\dot g)_* \pi \to \pi$, for 
a chosen set of representatives $\dot g \in \Gamma$ of $G$. We extend this to all
possible representatives by the rule
\[
\pi^G (\dot g n) = \pi^G (\dot g) \pi (n) \qquad g \in G, n \in N .
\]
By the multiplicativity \eqref{eq:A.6} of
equivariant structures, these $\pi^G (\dot g n)$ combine to a representation of 
\[
\bigcup\nolimits_{g \in G} \, \dot g N = \Gamma 
\]
on the vector space underlying $\pi$. It follows from $\pi (1) = \mr{id}$ that 
$\pi^G (n) = \pi (n)$ for all $n \in N$, so $\pi^G$ extends $\pi$.

It is clear that the functors $\pi^G \mapsto \pi$ and $\pi \mapsto \pi^G$ are mutually
inverse, so they define equivalences of categories.
\end{proof}

\textbf{Acknowledgements.}
The author thanks Tasho Kaletha for many enlightening discussions on this and related 
topics and Naoki Imai for pointing out the dependence of \cite{LuYu} on Whittaker data.
We also thank Amoru Fujii for detailed comments on an earlier version of this paper.\\

\textbf{Usage of AI.}
No AI was used to write this paper.

\end{document}